%% file: main_arxiv.tex
\documentclass[11pt]{article}

\usepackage[utf8]{inputenc}

\usepackage{amssymb}
\usepackage{makeidx}
\usepackage[english]{babel}
\usepackage{graphicx}
\usepackage{amsfonts,amsmath,amssymb,amsthm}
\usepackage{oldgerm}
\usepackage{mathrsfs}
\usepackage[active]{srcltx}
\usepackage{verbatim}
\usepackage{enumitem} 
\usepackage[toc,page]{appendix}
\usepackage{aliascnt,bbm}
\usepackage{bm}
\usepackage{array}
\usepackage[colorlinks = true, citecolor = blue]{hyperref}
\usepackage{xargs}
\usepackage{cellspace}
\usepackage{slashbox}

\usepackage[Symbolsmallscale]{upgreek}

\usepackage{accents}
\usepackage{dsfont}
\usepackage{aliascnt}
\usepackage{cleveref}
\makeatletter
\newtheorem{theorem}{Theorem}
\crefname{theorem}{theorem}{Theorems}
\Crefname{Theorem}{Theorem}{Theorems}

\newaliascnt{lemma}{theorem}
\newtheorem{lemma}[lemma]{Lemma}
\aliascntresetthe{lemma}
\crefname{lemma}{lemma}{lemmas}
\Crefname{Lemma}{Lemma}{Lemmas}

\newaliascnt{corollary}{theorem}
\newtheorem{corollary}[corollary]{Corollary}
\aliascntresetthe{corollary}
\crefname{corollary}{corollary}{corollaries}
\Crefname{Corollary}{Corollary}{Corollaries}

\newaliascnt{proposition}{theorem}
\newtheorem{proposition}[proposition]{Proposition}
\aliascntresetthe{proposition}
\crefname{proposition}{proposition}{propositions}
\Crefname{Proposition}{Proposition}{Propositions}

\newaliascnt{definition}{theorem}

\aliascntresetthe{definition}
\crefname{definition}{definition}{definitions}
\Crefname{Definition}{Definition}{Definitions}

\newaliascnt{remark}{theorem}

\aliascntresetthe{remark}
\crefname{remark}{remark}{remarks}
\Crefname{Remark}{Remark}{Remarks}

\crefname{example}{example}{examples}
\Crefname{Example}{Example}{Examples}

\crefname{figure}{figure}{figures}
\Crefname{Figure}{Figure}{Figures}

\newtheorem{assumption}{\textbf{H}\hspace{-3pt}}
\Crefname{assumption}{\textbf{H}\hspace{-3pt}}{\textbf{H}\hspace{-3pt}}
\crefname{assumption}{\textbf{H}}{\textbf{H}}

\Crefname{assumptionG}{\textbf{G}\hspace{-3pt}}{\textbf{G}\hspace{-3pt}}
\crefname{assumptionG}{\textbf{G}}{\textbf{G}}

\newtheorem{assumptionAR}{\textbf{AR}\hspace{-3pt}}
\Crefname{assumptionAR}{\textbf{AR}\hspace{-3pt}}{\textbf{AR}\hspace{-3pt}}
\crefname{assumptionAR}{\textbf{AR}}{\textbf{AR}}

\newenvironment{enumerateList}{ \begin{enumerate}[label=(\roman*),wide=0pt, labelindent=\parindent]}{\end{enumerate}}

\usepackage[a4paper]{geometry} 

\input{def}
\usepackage{xr}
\usepackage{authblk}

\author[1]{Alain Durmus}
\author[2]{\'Eric Moulines}

\title{High-dimensional Bayesian inference via  the Unadjusted Langevin Algorithm}
\affil[1]{CMLA - \'Ecole normale supérieure Paris-Saclay, CNRS, Université Paris-Saclay, 94235 Cachan, France.}
\affil[2]{Centre de Math\'ematiques Appliqu\'ees, UMR 7641, Ecole Polytechnique, France.}

\begin{document}
\footnotetext[1]{
Email: alain.durmus@cmla.ens-cachan.fr}
 \footnotetext[2]{eric.moulines@polytechnique.edu }
\maketitle

\smallskip

\smallskip


\begin{abstract}:
  We consider in this paper the problem of sampling a
  high-dimensional probability distribution $\pi$ having a density
  \wrt\ the Lebesgue measure on $\mathbb{R}^d$, known up to a normalization
  constant $x \mapsto \pi(x)= \mathrm{e}^{-U(x)}/\int_{\rset^d} \mathrm{e}^{-U(y)} \rmd y$.
  Such problem naturally occurs for example in Bayesian inference and
  machine learning.  Under the assumption that $U$ is continuously
  differentiable, $\nabla U$ is globally Lipschitz and $U$ is strongly
  convex, we obtain non-asymptotic bounds for the convergence to
  stationarity in Wasserstein distance of order $2$ and total
  variation distance of the sampling method based on the Euler
  discretization of the Langevin stochastic differential equation, for
  both constant and decreasing step sizes. The dependence on the
  dimension of the state space of these bounds is explicit.  The convergence of an
  appropriately weighted empirical measure is also investigated and
  bounds for the mean square error and exponential deviation
  inequality are reported for functions which are measurable and bounded.
   An illustration to Bayesian inference for binary regression is presented to support our claims.
\end{abstract}

 \input{introduction}
 \input{convergence}

\input{convergence-euler}
\input{WassersteinDiscret}
 \input{total_variation}

\input{MSE_tv}
 \input{NumericalExp}

\input{convergence_fun_autoreg_D}

\section*{Acknowledgements}
The authors would like to thank Arnak Dalalyan for helpful discussions.
The work of A.D. and E.M. is supported by the Agence Nationale de la Recherche, under grant ANR-14-CE23-0012 (COSMOS),  Initiative Data Science from Ecole Polytechnique and Chaire BayeScale "P. Laffitte".
\bibliographystyle{plain}
\bibliography{biblio}

\appendix
\input{proof_arxiv}

\end{document}

%% file: def.tex
\def\xstar{x^\star}

\def\trace{\operatorname{Tr}}
\newcommandx{\functionspace}[2][1=+]{\mathbb{F}_{#1}(#2)}

\newcommand{\estimateur}[1]{\hat{\pi}_n^N(#1)}

\newcommandx{\VarDeux}[3][3=]{\operatorname{Var}^{#3}_{#1}\left\{#2 \right\}}
\newcommand{\VarDeuxLigne}[2]{\operatorname{Var}_{#1}\{#2 \}}

\newcommand{\1}{\mathbbm{1}}

\newcommand{\B}{\mathcal{B}}

\newcommand{\LeftEqNo}{\let\veqno\@@leqno}

\newcommand{\floor}[1]{\left\lfloor #1 \right\rfloor}
\newcommand{\floorLigne}[1]{\lfloor #1 \rfloor}
\newcommand{\ceil}[1]{\left\lceil #1 \right\rceil}

\newcommand{\N}{\ensuremath{\mathbb{N}}}

\newcommand{\PE}{\mathbb{E}}
\newcommand{\PP}{\mathbb{P}}

\newcommand{\abs}[1]{\left\vert #1 \right\vert}

\newcommand{\tvnorm}[1]{\| #1 \|_{\mathrm{TV}}}
\newcommand{\tvnormEq}[1]{\left \| #1 \right \|_{\mathrm{TV}}}
\newcommandx{\Vnorm}[2][1=V]{\| #2 \|_{#1}}
\newcommandx{\VnormEq}[2][1=V]{\left\| #2 \right\|_{#1}}
\newcommandx{\norm}[2][1=]{\ifthenelse{\equal{#1}{}}{\left\Vert #2 \right\Vert}{\left\Vert #2 \right\Vert^{#1}}}

\newcommandx{\normLigne}[2][1=]{\ifthenelse{\equal{#1}{}}{\Vert #2 \Vert}{\Vert #2\Vert^{#1}}}

\newcommand{\parenthese}[1]{\left(#1 \right)}
\newcommand{\parentheseDeux}[1]{\left[ #1 \right]}

\newcommand{\defEns}[1]{\left\lbrace #1 \right\rbrace }

\newcommand{\defEnsG}[1]{\left\lbrace #1 \right. }
\newcommand{\defEnsD}[1]{\left. #1 \right  \rbrace }

\newcommand{\ps}[2]{\left\langle#1,#2 \right\rangle}
\newcommand{\eqdef}{=}

\newcommand{\proba}[1]{\mathbb{P}\left( #1 \right)}

\newcommandx\probaMarkovTilde[2][2=]
{\ifthenelse{\equal{#2}{}}{{\widetilde{\mathbb{P}}_{#1}}}{\widetilde{\mathbb{P}}_{#1}\left[ #2\right]}}
\newcommand{\probaMarkov}[2]{\mathbb{P}_{#1}\left[ #2\right]}
\newcommand{\expe}[1]{\PE \left[ #1 \right]}

\newcommand{\expeMarkov}[2]{\PE_{#1} \left[ #2 \right]}

\newcommand{\expeMarkovTilde}[2]{\widetilde{\PE}_{#1} \left[ #2 \right]}

\newcommand{\bigO}{\ensuremath{\mathcal O}}

\newcommand{\couplage}[2]{\Pi(#1,#2)}
\newcommand{\Pens}{\mathcal{P}}

\newcommand{\filtration}{\mathcal{F}}
\newcommand{\filtrationTilde}{\widetilde{\mathcal{F}}}

\newcommand{\plusinfty}{+\infty}

\newcounter{hypoconbis}
\newcounter{saveconbis}
\newcommand\debutH{\begin{list}
{\textbf{H\arabic{hypoconbis}}}{\usecounter{hypoconbis}}\setcounter{hypoconbis}{\value{saveconbis}}}
\newcommand\finH{\end{list}\setcounter{saveconbis}{\value{hypoconbis}}}

\def\ie{\textit{i.e.}}
\def\eqsp{\;}
\newcommand{\coint}[1]{\left[#1\right)}
\newcommand{\ocint}[1]{\left(#1\right]}
\newcommand{\ooint}[1]{\left(#1\right)}
\newcommand{\ccint}[1]{\left[#1\right]}
\renewcommand{\iint}[2]{\{#1,\ldots,#2\}}

\newcommandx{\weight}[2][2=n]{\omega_{#1,#2}^N}
\newcommandx{\weightD}[2][2=n]{\omega_{#1,#2}^0}

\newcommand{\boule}[2]{\operatorname{B}(#1,#2)}

\def\rmd{\mathrm{d}}
\newcommandx\sequence[3][2=,3=]
{\ifthenelse{\equal{#3}{}}{\ensuremath{\{ #1_{#2}\}}}{\ensuremath{\{ #1_{#2}, \eqsp #2 \in #3 \}}}}
\newcommandx{\sequencen}[2][2=n\in\N]{\ensuremath{\{ #1, \eqsp #2 \}}}
\newcommandx\sequenceDouble[4][3=,4=]
{\ifthenelse{\equal{#3}{}}{\ensuremath{\{ (#1_{#3},#2_{#3}) \}}}{\ensuremath{\{  (#1_{#3},#2_{#3}), \eqsp #3 \in #4 \}}}}
\newcommandx{\sequencenDouble}[3][3=n\in\N]{\ensuremath{\{ (#1_{n},#2_{n}), \eqsp #3 \}}}
\newcommand{\wrt}{w.r.t.}

\def\iid{i.i.d.}
\def\rme{\mathrm{e}}

\def\eg{e.g.}

\def\rset{\mathbb{R}}
\def\setProba{\mathcal{P}}

\def\nset{\mathbb{N}}
\def\tildem{\varpi}

\def\MSE{\operatorname{MSE}}
\def\Lip{\operatorname{Lip}}

\newcommandx{\CPE}[3][1=]{{\mathbb E}^{#3}_{#1}\left[#2 \right]} 
\newcommandx{\CPELigne}[3][1=]{{\mathbb E}^{#3}_{#1}[#2 ]} 
\newcommandx{\CPVar}[3][1=]{\mathrm{Var}^{#3}_{#1}\left\{ #2 \right\}}
\newcommand{\CPP}[3][]
{\ifthenelse{\equal{#1}{}}{{\mathbb P}\left(\left. #2 \, \right| #3 \right)}{{\mathbb P}_{#1}\left(\left. #2 \, \right | #3 \right)}}

\def\generator{\mathscr{A}}

\newcommandx{\osc}[2][1=]{\mathrm{osc}_{#1}(#2)}

\def\mcf{\mathcal{F}}
\def\mcg{\mathcal{G}}

\def\overlineY{\overline{Y}}

\newcommand{\chunk}[4][]%
{\ifthenelse{\equal{#1}{}}{\ensuremath{{#2}_{#3:#4}}}{\ensuremath{#2^#1}_{#3:#4}}
}

\def\Ybar{\bar{Y}}
\def\Id{\operatorname{Id}}
\def\IdM{\operatorname{I}_d}

\def\bfbeta{\pmb{\beta}}

\def\vrho{\varrho}

\def\VStar{V}

\def\boreleanA{\mathrm{A}}

\def\martInc{\Phi}
\def\martIncF{\Psi}

\def\Phibf{\mathbf{\Phi}}

\def\QKer{Q}
\def\RKer{R}
\def\PKer{P}
\def\gaStep{\gamma}
\def\GaStep{\Gamma}

\def\nFun{\mathsf{n}}
\def\nFunD{n}
\def\DnFunD{2^n}

\def\rateFun{\vartheta}
\def\FunMomentDEuler{\varrho}

\def\eventA{\boreleanA}

\def\diagSet{\mathsf{D}}

\newcommandx\sequenceg[3][2=,3=]
{\ifthenelse{\equal{#3}{}}{\ensuremath{( #1_{#2})}}{\ensuremath{( #1_{#2})_{ #2 \geq #3}}}}

\def\XSDE{\mathsf{X}}

\def\YSDE{\mathsf{Y}}
\def\fune{\mathtt{e}}

\def\LambdarMSE{\Lambda}

\def\constD{\mathsf{D}}

\def\constE{\mathsf{E}}

\def\rated{\chi}
\def\transar{\tau}

\newcommand{\defEnsE}[2]{\left\lbrace \left. #1 \, \right| #2 \right\rbrace}
\newcommand{\expeMarkovTildeD}[3]{\widetilde{\PE}_{#1}^{#3} \left[ #2 \right]}

\def\PPtilde{\widetilde{\PP}}
\def\PEtilde{\widetilde{\PE}}
\def\transfrr{\mathrm{F}}
\def\Deltar{\diagSet}
\def\complem{\operatorname{c}}
\def\alphar{\alpha}
\def\tildex{\tilde{x}}

\def\tildey{\tilde{y}}
\def\ar{\mathrm{a}}
\def\Kr{\mathsf{K}}
\def\Xr{\mathrm{X}}
\def\Yr{\mathrm{Y}}
\def\Xrd{\mathit{X}}
\def\Yrd{\mathit{Y}}
\def\Zr{\mathrm{Z}}

\def\sigmaD{\sigma^2}
\def\sigmakD{\sigma^2_k}
\newcommandx{\phibfs}[1][1=]{\pmb{\varphi}_{\sigmaD_{#1}}}
\def\phibfvs{\pmb{\varphi}_{\varsigma^2}}
\def\funreg{h}
\def\kappar{\varpi}
\def\Pr{\mathsf{P}}
\def\Qr{\mathsf{Q}}
\def\eventA{\mathtt{A}}
\newcommandx\sequenceD[3][2=,3=]
{\ifthenelse{\equal{#3}{}}{\ensuremath{\{ #1\}}}{\ensuremath{\{ #1, \eqsp #2 \in #3 \}}}}
\def\borelSet{\B}
\def\Er{\mathrm{E}}
\def\er{\mathrm{e}}
\def\transp{\operatorname{T}}

\newcommand\fracmiddle[2]{\left. #1 \middle / #2 \right.}

\newcommand\sqrta[1]{(#1)^{1/2}}
\def\tmsu{\tilde{\mathsf{U}}}
\def\tT{\tilde{T}}


%% file: introduction.tex
 \section{Introduction}
Interest for Bayesian inference methods for high-dimensional models  has recently received renewed attention often motivated by  machine learning
 applications. Rather than obtaining a point estimate, Bayesian methods attempt to sample the full
 posterior distribution over the parameters and possibly latent
 variables which provides a way to assert uncertainty in the
 model and prevents from overfitting \cite{neal:1992}, \cite{welling:the:2011}.

The problem can be formulated as follows.
We aim at sampling a posterior distribution $\pi$ on $\rset^d$, $d \geq 1$, with  density $x \mapsto
 \rme^{-U(x)}/\int_{\rset^d} \rme^{-U(y)} \rmd y$ \wrt~the Lebesgue
 measure, where $U$ is continuously differentiable. The Langevin stochastic differential equation
 associated with $\pi$ is defined by:
\begin{equation}
\label{eq:langevin_2}
\rmd Y_t = -\nabla U (Y_t) \rmd t + \sqrt{2} \rmd B_t \eqsp,
\end{equation}
where $(B_t)_{t\geq0}$ is a $d$-dimensional Brownian motion defined on
the filtered probability space $(\Omega,\mcf,(\mcf_t)_{t \geq
  0},\PP)$, satisfying the usual conditions.  Under mild technical
conditions, the Langevin diffusion admits $\pi$ as its unique
invariant distribution.

We study the sampling method based on the
Euler-Maruyama discretization of \eqref{eq:langevin_2}. This scheme
defines the (possibly) non-homogeneous, discrete-time Markov chain
$(X_k)_{k \geq 0}$
given by
\begin{equation}
\label{eq:euler-proposal-2}
X_{k+1}= X_k - \gamma_{k+1} \nabla U(X_k) + \sqrt{2 \gamma_{k+1}} Z_{k+1} \eqsp,
\end{equation}
where $(Z_k)_{k \geq 1}$ is an \iid\ sequence of $d$-dimensional standard Gaussian
random variables and $(\gamma_k)_{k \geq 1}$ is a sequence of
step sizes, which can either be held constant or be chosen to decrease
to $0$.  This algorithm has been first proposed by \cite{ermak:1975}
and \cite{parisi:1981} for molecular dynamics applications. Then it
has been popularized in machine learning by \cite{grenander:1983},
\cite{grenander:miller:1994} and computational statistics by
\cite{neal:1992} and \cite{roberts:tweedie:1996}. Following
\cite{roberts:tweedie:1996}, in the sequel this method will be referred
to as the \emph{unadjusted} Langevin algorithm (ULA). When the step sizes are held constant, under appropriate conditions on $U$, the homogeneous
Markov chain $(X_k)_{k \geq 0}$ has a unique stationary distribution
$\pi_{\gaStep}$, which in most  cases differs from
the distribution $\pi$. It has been proposed in
\cite{rossky:doll:friedman:1978} and \cite{roberts:tweedie:1996} to
use a Metropolis-Hastings step at each iteration to enforce
reversibility \wrt\ $\pi$. This algorithm is referred to as the
Metropolis adjusted Langevin algorithm (MALA).

The ULA algorithm has already been studied in depth for constant
step sizes in \cite{talay:tubaro:1991}, \cite{roberts:tweedie:1996} and
\cite{mattingly:stuart:higham:2002}. In particular, \cite[Theorem
4]{talay:tubaro:1991} gives an asymptotic expansion for the weak error
between $\pi$ and $\pi_{\gamma}$. When
$\lim_{k\to \plusinfty} \gamma_k = 0$ and $\sum_{k=1} ^{\infty}
\gamma_k = \infty$, weak convergence of the weighted empirical
distribution of the ULA algorithm has been established in
\cite{lamberton:pages:2002}, \cite{lamberton:pages:2003} and
\cite{lemaire:2005}.

Contrary to these reported works, we focus in this paper on
non-asymptotic results. These questions have been addressed previously in
\cite{dalalyan:2014} and \cite{durmus:moulines:2016}. 
\cite{dalalyan:2014} establishes  explicit bounds on the total
variation distance between the distribution of the $n$-th iterate of
the Markov chain defined in \eqref{eq:euler-proposal-2} and the target
distribution $\pi$ for fixed step size and a strongly convex potential
$U$. It is shown that if the initial distribution is an appropriately chosen Gaussian or if a warm-start is used, the number of iterations required to
get a sample $\epsilon$-close to $\pi$ in total variation is of order $\bigO(d^3 \varepsilon^{-2})$ and $\bigO(d \varepsilon^{-2})$ respectively.
The results of \cite{dalalyan:2014} were later sharpened in \cite{durmus:moulines:2016}, using different technical arguments. In particular, \cite{durmus:moulines:2016} shows that starting from a minimizer of $U$, the number of iterations to
get a sample $\varepsilon$-close from $\pi$ in total variation is of order $\bigO(d \varepsilon^{-2})$ and that therefore a warm start is not necessary. \cite{durmus:moulines:2016} also extends the results of \cite{dalalyan:2014} to non-convex potentials and non-increasing sequences of step sizes. It also establish some bounds between $\pi$ and $\pi_{\gamma}$ in $V$-norm which scale as $\gamma^{1/2}$ as $\gamma \to 0$.

In this work, we focus on the case where $U$ is strongly
convex. Compared to \cite{dalalyan:2014} and
\cite{durmus:moulines:2016}, our contributions are as follows.
\begin{enumerate}[label=$\bullet$]
\item We give explicit bounds between the distribution of the $n$-th
  iterate of the Markov chain defined in \eqref{eq:euler-proposal-2}
  and the target distribution $\pi$ in Wasserstein and total variation
  distance for fixed and non-increasing step sizes. The obtained
  bounds improve those reported in \cite{dalalyan:2014} and
  \cite{durmus:moulines:2016} for the total variation distance.
\item For fixed step sizes ($\gamma_k= \gamma$ for all $k \geq 0$), 
  we analyse both fixed horizon (the total
  computational budget is fixed and the step size is chosen to
  minimize the upper bound on the Wasserstein or total variation
  distance) and fixed precision (for a fixed target precision, the
  number of iterations and the step size are optimized simultaneously
  to meet this constraint). For a fixed precision
  $\varepsilon >0$, we show that the number of
  iterations $n \geq 0$, for ULA to get a sample  $\varepsilon$-close to  $\pi$ in Wasserstein distance / total variation of order
  $\bigO(d \varepsilon^{-2})$ or $\bigO(d \varepsilon^{-1})$ (up to logarithmic terms),
  depending on the smoothness of $U$.  We show that our result
  is optimal (up to logarithmic factors again) for $d$-dimensional Gaussian distribution. 
  We show  in the finite horizon setting that if the total number of iterations is $n$, we may choose the step size $\gamma= \gamma_n >0$ such that the Wasserstein distance between the
  distribution of the $n$-th iterate and $\pi$ is bounded by $\bigO(n^{-1/2})$ and $\bigO(n^{-1})$ depending on the smoothness of $U$.
\item
  When $\lim_{k\to \plusinfty}\gamma_k = 0$
 and $\sum_{k=1}^\infty \gamma_k = \infty$, we show
that the marginal distribution of the non-homogeneous Markov chain
$(X_k)_{k \geq 0}$ converges to the target distribution $\pi$ and
provide explicit convergence bounds  in the case $\gamma_k = \gamma_1 k^{-\alpha}$, $\alpha \in \ocint{0,1}$. The optimal rate of convergence derived from our bounds for the Wasserstein/total variation distance is obtained for $\alpha=1$ with $\gamma_1>0$ large enough. The convergence rates we report, improve those given in \cite{durmus:moulines:2016}. 
\item Quantitative estimates between $\pi$ and
  $\pi_{\gamma}$ are obtained in Wasserstein and total variation distance. The bound on the total variation distance between $\pi$ and $\pi_{\gamma}$ we derive improves the one reported in \cite{durmus:moulines:2016}. In particular, when $U$ is smooth enough, $\tvnorm{\pi-\pi_{\gamma}}$ scales as $\gamma$ as $\gamma \to 0$. 
\item Convergence of   weighted empirical measure is
  studied through bounds on the mean square error and exponential
  deviation of an estimator of $\int_{\rset^d} f(x) \rmd \pi(x)$, for
  functions $f : \rset^d\to \rset$ which are either Lipschitz or
  bounded and measurable.  When $f$ is Lipschitz, $U$
  is smooth enough and in the any-time setting, the
  optimal rate of convergence for the MSE, using non-increasing
  sequences $\gamma_k = \gamma_1/k^{\alpha}$, is obtained for
  $\alpha = 1/3$ (which coincides with the rate used in \cite{lamberton:pages:2002} to derive a central limit theorem).  If the step size is held
  constant, we get that the  number of iterations for the
  mean square error to be smaller than $\varepsilon >0$ is of order
  $\bigO(d \varepsilon^{-4})$ or $\bigO(d \varepsilon^{-3})$,
  depending on the smoothness of $U$. 
  The case where $f$ is bounded
  and measurable is an important result in Bayesian statistics to
  estimate credibility regions. For that purpose, we study the
  convergence of the Euler-Maruyama discretization towards its
  stationary distribution in total variation using a discrete time
  version of reflection coupling introduced in
  \cite{bubley:dyer:jerrum:1998}. For fixed step size, the conclusion
  on the sufficient number of iterations for the mean square error to
  be smaller than $\varepsilon >0$ is the same (up to logarithmic
  terms) as for Lipschitz functions.
\end{enumerate}
In this paper, a special attention is paid
to the dependency of the  obtained bounds on the dimension of the state space, since we are particularly interested in the applications of this method to sampling in high-dimension.


  The paper is organized as follows. In \Cref{sec:non-asympt-bounds},
  we study the convergence in the Wasserstein distance of order $2$ of
  the Euler discretization for constant and decreasing step sizes.  In
  \Cref{sec:quant-bounds-total}, we give non asymptotic bounds in
  total variation distance between the Euler discretization and $\pi$.
  This study is completed in \Cref{SEC:MSE_TV} by non-asymptotic
  bounds of convergence of the weighted empirical measure applied to
  functions which are either Lipschitz or bounded and measurable.
Our claims are supported in a Bayesian inference for a binary
regression model in
\Cref{sec:numerics}.  Finally in \Cref{sec:conv-total-vari_AR}, some
results of independent interest, used in the proofs, on functional
autoregressive models are gathered. Most proofs and derivations are
postponed and carried out in Appendices and a supplementary paper
\cite{durmus:moulines:2015:supplement}.


%% file: convergence.tex
\def\transp{\operatorname{T}}

\subsection*{Notations and conventions}
Denote by $\mathcal{B}(\rset^d)$ the Borel $\sigma$-field of
$\rset^d$, $\functionspace[]{\rset^d}$ the set of all Borel measurable
functions on $\rset^d$ and for $f \in \functionspace[]{\rset^d}$,
$\Vnorm[\infty]{f}= \sup_{x \in \rset^d} \abs{f(x)}$.  For $\mu$ a probability measure
on $(\rset^d, \mathcal{B}(\rset^d))$ and $f \in
\functionspace[]{\rset^d}$ a $\mu$-integrable function, denote by
$\mu(f)$ the integral of $f$ \wrt~$\mu$.  We say that $\zeta$ is a
transference plan of $\mu$ and $\nu$ if it is a probability measure on
$(\rset^d \times \rset^d, \mathcal{B}(\rset^d \times \rset^d) )$ such
that for all measurable set $\boreleanA$ of $\rset^d$,
$\zeta(\boreleanA \times \rset^d) = \mu(\boreleanA)$ and
$\zeta(\rset^d \times \boreleanA) = \nu(\boreleanA)$. We denote by
$\couplage{\mu}{\nu}$ the set of transference plans of $\mu$ and
$\nu$. Furthermore, we say that a couple of $\rset^d$-random variables
$(X,Y)$ is a coupling of $\mu$ and $\nu$ if there exists $\zeta \in
\couplage{\mu}{\nu}$ such that $(X,Y)$ are distributed according to
$\zeta$.  For two probability measures $\mu$ and $\nu$, we define the
Wasserstein distance of order $p \geq 1$ as
\begin{equation*}
W_p(\mu,\nu) \eqdef \left( \inf_{\zeta \in \couplage{\mu}{\nu}} \int_{\rset^d \times \rset^d} \norm[p]{x-y}\rmd \zeta (x,y)\right)^{1/p} \eqsp.
\end{equation*}
By \cite[Theorem 4.1]{VillaniTransport}, for all $\mu,\nu$ probability measures on $\rset^d$, there exists a transference plan $\zeta^\star \in \couplage{\mu}{\nu}$ such that for any coupling $(X,Y)$ distributed according to $\zeta^\star$, $W_p(\mu,\nu) = \PE[\norm[p]{X-Y}]^{1/p}$. This kind of transference plan (respectively coupling) will be called an optimal transference plan (respectively optimal coupling) associated with $W_p$.  We denote by $\setProba_p(\rset^d)$ the set of probability measures with finite $p$-moment:  for all $\mu \in \setProba_p(\rset^d)$, $\int_{\rset^d} \norm[p]{x} \rmd \mu( x) < \plusinfty$. By \cite[Theorem 6.16]{VillaniTransport}, $\setProba_p(\rset^d)$ equipped with the Wasserstein distance $W_p$ of order $p$  is a complete separable metric space.

Let $f : \rset^d \to \rset$ be a Lipschitz function, namely there exists $C \geq 0$ such that for all $x,y \in \rset^d$, $ \abs{f(x) - f(y)} \leq C \norm{x-y}$. Then we denote 
\begin{equation*}
\norm{f}_{\Lip} = \inf \{  \abs{f(x) - f(y)} \norm[-1]{x-y} \ | \ x,y \in \rset^d , x \not = y \} \eqsp.
\end{equation*}
The Monge-Kantorovich theorem (see \cite[Theorem 5.9]{VillaniTransport}) implies that for all $\mu,\nu$ probability measures on $\rset^d$,
\begin{equation*}
W_1(\mu,\nu) = \sup \defEns{\int_{\rset^d} f(x) \rmd  \mu (x) - \int_{\rset^d} f(x) \rmd \nu ( x) \ | \ f: \rset^d \to \rset \ ; \   \norm{f}_{\Lip} \leq 1}\eqsp.
\end{equation*}
Denote by $\functionspace[b]{\rset^d}$ the set of all bounded Borel measurable functions on $\rset^d$. For $f \in \functionspace[b]{\rset^d}$ set $\osc{f} = \sup_{x,y \in \rset^d}\abs{f(x)-f(y)}$. 
For two probability measures $\mu$ and $\nu$ on $\rset^d$,  the total variation distance distance between $\mu$ and $\nu$ is defined by
$\tvnorm{\mu-\nu} = \sup_{\eventA \in \B(\rset^d)}\abs{\mu(\eventA) - \nu(\eventA)}$.
By the Monge-Kantorovich theorem the total variation distance between $\mu$ and $\nu$ can be written on the form:
\begin{equation*}
  \tvnorm{\mu-\nu} = \inf_{\zeta \in \couplage{\mu}{\nu}} \int_{\rset^d \times \rset^d} \1_{\diagSet^{\operatorname{c}}}(x,y) \rmd \zeta (x,y) \eqsp,
\end{equation*}
where $\diagSet = \{(x,y)\in \rset^d \times \rset^d \, | x=y \}$.
For all $x \in \rset^d$
and $M >0$, we denote by $\boule{x}{M}$, the ball centered at $x$ of
radius $M$. For a subset $\eventA \subset \rset^d$, denote by
$\eventA^{\operatorname{c}}$ the complementary of $\eventA$.  Let $n
\in \nset^*$ and $M$ be a $n \times n$-matrix, then denote by
$M^{\transp}$ the transpose of $M$ and $\norm{M}$ the operator norm
associated with $M$ defined by $\norm{M} = \sup_{\norm{x} = 1} \norm{M
  x}$. Define the Frobenius norm associated with $M$ by 
$\norm{M}_{\operatorname{F}}^2 = \trace(M^T M)$.  Let $n,m \in \nset^*$ and $F : \rset^n \to \rset^m$ be a twice
continuously differentiable function. Denote by $\nabla F$ and
$\nabla^2 F$ the Jacobian and the Hessian of $F$ respectively. Denote
also by $\vec{\Delta} F$ the vector Laplacian of $F$ defined by:
for all $x \in \rset^d$, $\vec{\Delta} F(x)$ is the vector of
$\rset^m$ such that for all $i\in \{1,\cdots, m\}$, the $i$-th
component of $\vec{\Delta}F(x)$ equals to $\sum_{j=1}^d (\partial^2
F_i /\partial x_j^2)(x)$. 
 In the sequel, we take the convention that
$\sum_{p}^n =0$ and $\prod_p ^n = 1$ for $n,p \in \nset$, $n <p$.


%% file: convergence-euler.tex
\section{Non-asymptotic bounds in Wasserstein distance of order $2$ for ULA}
\sectionmark{Non-asymptotic bounds in $W_2$ for ULA}
\label{sec:non-asympt-bounds}

Consider the following assumption on the potential $U$:
\begin{assumption}
\label{assum:regularity_2}
\label{assum:regularity}
The function $U$ is  continuously differentiable on $\rset^d$ and gradient Lipschitz: there exists $L \geq 0$ such that for all $x,y \in \rset^d$, $
\norm{\nabla U(x) - \nabla U(y)} \leq L \norm{x-y}$.
\end{assumption}
Under \Cref{assum:regularity_2}, for all $x \in \rset^d$
 by  \cite[Theorem~2.5, Theorem~2.9~Chapter 5]{karatzas:shreve:1991} there exists a unique strong
solution $(Y_t)_{t \geq 0}$ to \eqref{eq:langevin_2} with $Y_0 = x$.
Denote by $(P_t)_{t \geq 0}$ the semi-group associated with \eqref{eq:langevin_2}. It is well-known that $\pi$ is its (unique) invariant probability.
To get geometric convergence of $(P_t)_{t \geq 0}$ to $\pi$ in Wasserstein distance of order $2$, we make the following additional assumption on the potential $U$.
\begin{assumption}
\label{assum:potentialU}
$U$ is strongly convex, \ie~there exists $m >0$ such that for all $x,y \in \rset^d$,
\begin{equation*}
 U(y) \geq U(x)+  \ps{ \nabla U(x)}{y-x} +  (m/2) \norm[2]{x-y} \eqsp.
\end{equation*}
\end{assumption}
Under  \Cref{assum:potentialU}, \cite[Theorem 2.1.8]{nesterov:2004} shows that $U$ has a unique minimizer $x^\star \in \rset^d$.
We briefly summarize some background material on the stability
  and the convergence in $W_2$ of the overdamped Langevin diffusion
  under \Cref{assum:regularity_2} and \Cref{assum:potentialU}. Most
of the statements in \Cref{theo:convergence-WZ-strongly-convex} are
known and are recalled here for ease of references; see \eg~\cite{chen:shao:1989}.
\begin{proposition}
\label{theo:convergence-WZ-strongly-convex}
  Assume \Cref{assum:regularity_2} and \Cref{assum:potentialU}.
\begin{enumerate}[label= (\roman*)]
\item
\label{item:moment_diffusion_2}
For all $t \geq 0$ and  $x \in \rset^d$,
\begin{equation*}
\int_{\rset^d} \norm[2]{y-x^{\star}} P_t(x,\rmd y)  \leq \norm[2]{x-x^\star} \rme^{-2 m t} + (d/m) (1-\rme^{-2 m t}) \eqsp.
\end{equation*}
\item \label{item:moment_diffusion_item}
The stationary distribution $\pi$  satisfies $\int_{\rset^d}  \norm[2]{x-x^\star}  \pi(\rmd x)  \leq d/m$.
\item
\label{item:1:theo:convergence-WZ-strongly-convexD}
For any $x, y \in \rset^d$ and  $t > 0$, $W_{2}(\delta_x P_t , \delta_y  P_t ) \leq \rme^{-mt} \norm{x-y}$.
\item
\label{item:2:theo:convergence-WZ-strongly-convexD}
For any $x\in \rset^d$ and  $t > 0$, $W_{2}(\delta_x P_t , \pi) \leq \rme^{-mt} \defEns{\norm{x-\xstar}+ (d/m)^{1/2}}$.
\end{enumerate}
\end{proposition}
\begin{proof}
The proof  is given  in the supplementary document \Cref{proof:theo:convergence-WZ-strongly-convex}.
\end{proof}
Note that the convergence rate in \Cref{theo:convergence-WZ-strongly-convex}-\ref{item:2:theo:convergence-WZ-strongly-convexD} does not depend on the dimension. Let $(\gamma_k)_{k \geq 1}$ be a sequence of positive and non-increasing step sizes and for $n,\ell \in \nset$, denote by
 \begin{equation}
 \label{eq:def_Gamma}
 \Gamma_{n,\ell} \eqdef \sum_{k=n}^\ell\gamma_k \eqsp, \qquad \Gamma_n = \Gamma_{1,n} \eqsp.
\end{equation}
For $\gamma >0$, consider the Markov kernel $R_\gamma$ given for all $\boreleanA \in \mathcal{B}(\rset^d)$
and $x \in \rset^d$ by
\begin{equation}
\label{eq:definition-Rgamma}
R_\gamma(x,\boreleanA) =
\int_\boreleanA (4\uppi \gamma)^{-d/2} \exp \parenthese{-(4 \gamma)^{-1}\norm[2]{y-x+ \gamma \nabla U(x)}} \rmd y
\eqsp.
\end{equation}
The process $(X_k)_{k \geq 0}$ given in \eqref{eq:euler-proposal-2} is an inhomogeneous  Markov chain with respect to the family of Markov kernels $(R_{\gamma_k})_{k \geq 1}$.
For $\ell,n \in \nset^* $, $\ell \geq n$, define
\begin{equation}
\label{eq:iterate_kernel}
Q^{n,\ell}_\gamma = R_{\gamma_n} \cdots R_{\gamma_{\ell}}  \eqsp, \qquad Q^{n}_\gamma = Q^{1,n}_\gamma
\end{equation}
with the convention that for $n,\ell \in \nset$, $\ell < n$, $Q^{n,\ell}_\gamma$ is the identity operator.
\label{sec:bound_W2}


%% file: WassersteinDiscret.tex
We first derive a Foster-Lyapunov drift condition for $Q^{n,\ell}_\gamma $,  $\ell,n \in \nset^*$, $\ell \geq n$. Set
\begin{equation}
\label{eq:definition-kappa}
\kappa = \frac{2 m L }{m+L} 
\end{equation}
where $m$ and $L$ are defined in \Cref{assum:regularity_2}
\begin{proposition}
\label{theo:kind_drift}
Assume \Cref{assum:regularity_2} and \Cref{assum:potentialU}.
\begin{enumerate}[label=(\roman*)]
\item
\label{item:kind_drift_1}
 Let $(\gamma_k)_{k \geq 1}$ be a non-increasing sequence with $\gamma_1 \leq 2/(m+L)$. Let $x^\star$ be the unique minimizer of $U$.
Then for all $x \in \rset^d$ and  $n,\ell \in \nset^*$,
 \begin{equation*}
  \int_{\rset^d} \norm[2]{y - x^\star}  Q^{n,\ell}_\gamma(x,\rmd y) \leq
 \FunMomentDEuler_{n,\ell}(x) \eqsp,
 \end{equation*}
where $\FunMomentDEuler_{n,\ell}(x)$ is given by
\begin{equation}
  \label{eq:FunMomentDEuler}
  \FunMomentDEuler_{n,\ell}(x)=  \prod_{k=n}^\ell (1-\kappa \gamma_k)  \norm[2]{x - x^\star}
+2d \kappa^{-1}\defEns{1-\kappa^{-1} \prod_{i=n}^{\ell}(1 - \kappa \gamma_i)}\eqsp,
\end{equation}
\item
\label{item:kind_drift_2}
 For any $\gamma \in \ocint{0,2/(m+L)}$, $R_\gamma$ has a unique stationary distribution $\pi_\gamma$ and
 \[
 \int_{\rset^d} \norm[2]{x-\xstar} \pi_{\gamma} (\rmd x) \leq 2 d \kappa^{-1} \eqsp.
 \]
\end{enumerate}
\end{proposition}
\begin{proof}
The proof is postponed to \Cref{proof:theo:kind_drift}.
\end{proof}

We now proceed to establish that $\QKer_{\gamma}^{n}$ is a strict contraction in $W_2$ for any $n \geq 1$. This result implies the geometric convergence of the sequence $(\delta_x \RKer_{\gaStep}^n)_{n \geq 1}$ to $\pi_{\gaStep}$ in $W_2$ for all $x \in \rset^d$. Note that the convergence rate again does not depend on the dimension.
\begin{proposition}
  \label{theo:convergence_p_Euler}
  Assume \Cref{assum:regularity_2} and \Cref{assum:potentialU}. Then,
  \begin{enumerate}[label=(\roman*)]
  \item \label{item:contraction_Euler-i} Let $(\gamma_k)_{k \geq 1}$ be a non-increasing sequence with $\gamma_1 \leq 2/(m+L)$.
  For all $x,y \in \rset^d$ and $\ell \geq n \geq 1$,
\[
W_{2}(\delta_x Q^{n,\ell}_\gamma,\delta_y Q^{n,\ell}_\gamma) \leq \defEns{\prod_{k=n}^{\ell} (1- \kappa \gamma_k) }^{1/2} \norm{x-y}\eqsp.
\]
\item \label{eq:rate_euler_p}
 For any $\gamma \in \ooint{0,2/(m+L)}$, for all $x \in \rset^d$ and $n \geq 1$,
$$W_{2}(\delta_x R_{\gamma}^n,\pi_{\gamma}) \leq (1-\kappa \gamma)^{n/2}\defEns{\norm[2]{x-\xstar} + 2 \kappa^{-1}d}^{1/2} \eqsp.$$
\end{enumerate}
\end{proposition}
\begin{proof}
The proof is postponed to \Cref{sec:proof:convergence_p_Euler}.
\end{proof}
\begin{corollary}
\label{propo:contraction_Euler-new}
Assume \Cref{assum:regularity_2} and  \Cref{assum:potentialU}. Let  $(\gamma_k)_{k \geq 1}$ be a non-increasing sequence with $\gamma_1 \leq 2/(m+L)$. Then
 for all Lipschitz functions $f : \rset^d \to \rset$ and $ \ell \geq n \geq 1$, $Q^{n,\ell}_\gamma f$ is a Lipschitz function with
$\normLigne{Q^{n,\ell}_\gamma f}_{\Lip} \leq \prod_{k=n}^\ell (1- \kappa \gamma_k)^{1/2} \normLigne{f}_{\Lip}$.
\end{corollary}

\begin{proof}
The proof  follows from \Cref{theo:convergence_p_Euler}-\ref{item:contraction_Euler-i} using
\begin{equation*}
\abs{Q^{n,\ell}_\gamma f(y) - Q^{n,\ell}_\gamma f(z)} \leq  \norm{f}_{\Lip}W_2(\delta_y Q^{n,\ell}_\gamma , \delta_z Q_\gamma ^{n,\ell})\eqsp.
\end{equation*}
\end{proof}
We now proceed to establish explicit bounds for $W_2(\delta_x Q^n_\gamma, \pi)$, with $x \in \rset^d$.
\begin{theorem}
\label{theo:distance_Euler_diffusion}
Assume \Cref{assum:regularity_2} and \Cref{assum:potentialU}. Let $(\gamma_k)_{k \geq 1}$ be a non-increasing sequence with $\gamma_1 \leq 1/(m+L)$. Then for all $x \in \rset^d$ and $n \geq 1$,
\begin{equation*}
W_2^2(\delta_x Q^n_\gamma, \pi)\leq u_n^{(1)}(\gamma)\defEns{\norm[2]{x-\xstar} +d/m} + u_n^{(2)}(\gamma) \eqsp,
\end{equation*}
where
\begin{equation}
\label{eq:def_u1}
u_n^{(1)}(\gamma) =  2 \prod_{k=1}^n (1 - \kappa \gamma_k/2) 
\end{equation}
 $\kappa$ is defined in \eqref{eq:definition-kappa} and
\begin{multline}
\label{eq:def_u2}
u_n^{(2)}(\gamma)=  L^2 d \sum_{i=1}^n \left[ \gamma_i^2 \defEns{\kappa^{-1} + \gamma_i} \defEns{2 +\frac{ L^2 \gamma_i}{ m} +\frac{ L^2 \gamma_i^2}{6}} \prod_{k=i+1}^n(1- \kappa \gamma_k/2) \right]
\eqsp.
\end{multline}
\end{theorem}
\begin{proof}
The proof is postponed to \Cref{sec:proof:theo:distance_Euler_diffusion}.
\end{proof}

\begin{corollary}
\label{coro:distance_Euler_target}
Assume \Cref{assum:regularity_2} and \Cref{assum:potentialU}. Let $(\gamma_k)_{ k\geq 1}$ be a non-increasing sequence with $\gamma_1 \leq 1/(m+L)$. Assume that $\lim_{k \to \infty} \gamma_k= 0$ and $\lim_{n \to \plusinfty} \Gamma_n =\plusinfty$.
Then for all $x \in \rset^d$,
$
\lim_{n \to \infty} W_2( \delta_x Q^n_\gamma, \pi)=0 .
$
\end{corollary}
\begin{proof}
  The proof is postponed to \Cref{sec:proof-lem_rec_2}.
\end{proof}

In the case of constant step sizes $\gamma_k = \gamma$ for all $k \geq 1$, we can deduce from \Cref{theo:distance_Euler_diffusion}, a bound between $\pi$ and the stationary distribution $\pi_{\gamma}$ of $R_{\gamma}$.
\begin{corollary}
\label{coro:asympt_bias_2}
  Assume \Cref{assum:regularity_2} and \Cref{assum:potentialU}. Let $(\gamma_k)_{ k\geq 1}$ be a constant sequence $\gamma_k = \gamma$ for all $k \geq 1$ with $\gamma\leq 1/(m+L)$. Then
  \begin{equation*}
    W_2^2(\pi,\pi_{\gamma}) \leq 2 \kappa^{-1}  L^2  \gamma \defEns{\kappa^{-1} + \gamma} (2d +d L^2 \gamma /m +d L^2 \gamma^2/6) \eqsp.
  \end{equation*}
\end{corollary}

\begin{proof}
Since by \Cref{theo:convergence_p_Euler}, for all $x \in \rset^d$, $(\delta_x R_{\gamma}^n)_{n\geq 0}$ converges to $\pi_{\gamma}$ as $n \to \infty$ in $(\Pens_{2}(\rset^d),W_2)$, the proof then follows from \Cref{theo:distance_Euler_diffusion} and \cite[\Cref{lem:suite_recurrence_2}]{durmus:moulines:2016} applied with $\ell =1$.
\end{proof}
We can improve the bound provided by \Cref{theo:distance_Euler_diffusion} under additional regularity assumptions on the potential $U$.

\begin{assumption}
  \label{assum:reg_plus}
The potential $U$ is three times continuously differentiable and there exists $\tilde{L}$ such that for all $x,y \in \rset^d$, $\norm{  \nabla^2 U(x) - \nabla^2 U(y)} \leq \tilde{L}\norm{x-y}$.
\end{assumption}

Note that under \Cref{assum:regularity_2} and \Cref{assum:reg_plus}, we have that for all $x,y \in \rset^d$,
\begin{equation}
  \label{eq:borne_Lap_vec}
 \norm{\nabla^2 U(x)y} \leq L \norm{y} \eqsp, \eqsp \norm[2]{ \vec{\Delta} (\nabla U)(x) } \leq d^2 \tilde{L}^2 \eqsp.
\end{equation}

\begin{theorem}
\label{theo:distance_Euler_diffusionD}
Assume \Cref{assum:regularity_2}, \Cref{assum:potentialU} and \Cref{assum:reg_plus}. Let $(\gamma_k)_{k \geq 1}$ be a non-increasing sequence with $\gamma_1 \leq 1/(m+L)$. Then for all $x \in \rset^d$ and $n \geq 1$,
\begin{equation*}
W_2^2(\delta_x Q^n_\gamma, \pi)\leq u_n^{(1)}(\gamma) \defEns{\norm[2]{x-\xstar} +d/m} + u_n^{(3)}(\gamma) \eqsp,
\end{equation*}
where $u_n^{(1)}$ is given by \eqref{eq:def_u1}, $\kappa$ in \eqref{eq:definition-kappa} and
\begin{align}
\nonumber
u_n^{(3)}(\gamma) =\sum_{i=1}^n  \left[ d\gamma_{i}^3\defEns{2L^2 + \gamma_{i} L^4\parenthese{\frac{\gamma_i}{6} +m^{-1}} + \kappa^{-1} \parenthese{\frac{4 d \tilde{L}^2}{3} + \gamma_{i}L^4+ \frac{4L^4}{3m}}  } \right.\\
\label{eq:def_u2D}
\left. \times \prod_{k=i+1}^n\parenthese{1- \frac{\kappa\gamma_k}{2}} \right]
\eqsp.
\end{align}
\end{theorem}

\begin{proof}
  The proof is postponed to \Cref{sec:proof:distance_Euler_diffusionD}.
\end{proof}
If $\gamma_k = \gamma$ for all $k \geq 1$, we can deduce from \Cref{theo:distance_Euler_diffusionD}, a sharper bound between $\pi$ and the stationary distribution $\pi_{\gamma}$ of $R_{\gamma}$.
\begin{corollary}
\label{coro:asympt_biasD}
  Assume \Cref{assum:regularity_2}, \Cref{assum:potentialU} and \Cref{assum:reg_plus}. Let $(\gamma_k)_{ k\geq 1}$ be a constant sequence $\gamma_k = \gamma$ for all $k \geq 1$ with $\gamma \leq 1/(m+L)$. Then
  \begin{equation*}
    W_2^2(\pi,\pi_{\gamma}) \leq 2 \kappa^{-1} d\gamma^2 \defEns{2L^2+ \gamma L^4(\gamma /6 +m^{-1})+ \kappa^{-1}\parenthese{\frac{4 d \tilde{L}^2}{3} + \gamma L^4+ \frac{4L^4}{3m}} }  \eqsp.
  \end{equation*}
\end{corollary}
\begin{proof}
  The proof follows the same line as the proof of \Cref{coro:asympt_bias_2} and is omitted.
\end{proof}
Using \Cref{theo:convergence_p_Euler}-\ref{eq:rate_euler_p} and
\Cref{coro:distance_Euler_target} or \Cref{coro:asympt_biasD}, given
$\varepsilon > 0$, we determine the number of iterations
$n_{\varepsilon}$ and an associated step size $\gamma_{\varepsilon}$
to ensure that
$W_2(\delta_{\xstar} \RKer_{\gamma_{\varepsilon}}^n, \pi) \leq
\varepsilon$ for all $n \geq n_{\varepsilon}$. The precise expression
of $n_{\varepsilon}$ directly computed using
\Cref{theo:distance_Euler_diffusion} and
\Cref{theo:distance_Euler_diffusionD} are also given in
\cite[\Cref{sec:expl-bounds-fixed_precis_1}-\Cref{sec:expl-bounds-fixed_precis_2}]{durmus:moulines:2015:supplement}. Dependencies in dimension $d$ and precision $\varepsilon$ of
$n_{\varepsilon}$ are reported in
\Cref{tab:comparison_nDZero}. Under \Cref{assum:regularity} and
\Cref{assum:potentialU}, the complexity matches the results reported in \cite{durmus:moulines:2016} for the total variation
distance. Under \Cref{assum:reg_plus}, the dependency in  the precision
$\varepsilon$ can be improved. If $\tilde{L}=0$ (for
example for non-degenerate $d$-dimensional Gaussian distributions),
then the dependency in $d$ given by \Cref{theo:distance_Euler_diffusionD} is of order  $\bigO(d^{1/2} \log(d))$.

In a recent work \cite{dalalyan:2017} (based on a previous version of this paper), an improvement of
the proof of \Cref{theo:distance_Euler_diffusion} has been proposed for constant step size.
Whereas the constants are sharper, dependency in
dimension $d$ and  precision  $\varepsilon >0$  is
the same (first line of \Cref{tab:comparison_nDZero}).


\begin{table}[h]
\centering
\begin{normalsize}
\begin{tabular}{|c|c|c|c|}
    \hline
    Parameter & $d,\varepsilon$  \\
    \hline
    \Cref{theo:distance_Euler_diffusion}  and \Cref{theo:convergence_p_Euler}-\ref{eq:rate_euler_p} & $ \bigO(d \log(d) \varepsilon^{-2} \abs{\log(\varepsilon)})$
 \\
    \Cref{theo:distance_Euler_diffusionD}  and \Cref{theo:convergence_p_Euler}-\ref{eq:rate_euler_p} & $\bigO(d \log(d) \varepsilon^{-1} \abs{\log(\varepsilon)})$   \\
   \hline
\end{tabular}
\end{normalsize}
\caption{\normalsize{Dependencies of the number of iterations $n_{\varepsilon}$ to get $W_2(\delta_{\xstar} R^{n_{\varepsilon}}_{\gamma_{\varepsilon}}, \pi) \leq \varepsilon$}}
\label{tab:comparison_nDZero}
\end{table}

Under \Cref{assum:regularity_2} and \Cref{assum:potentialU}, by \Cref{theo:distance_Euler_diffusion}, in the finite horizon setting, then  for any $n  \geq 1$, we may choose a step size $\gamma = \gamma_n >0$ such that $W_2^2(\delta_{\xstar} R_{\gamma_n}^n , \pi) = \bigO(\log(n) / n)$ and $W_2^2(\delta_{\xstar} R_{\gamma_n}^n , \pi) \leq \bigO (\log(n) / n)^2$ if \Cref{assum:reg_plus} holds by \Cref{theo:distance_Euler_diffusionD}. The precise statement of these results are given by \cite[\Cref{coro:fixed_iteration}-\Cref{coro:fixed_iterationD}]{durmus:moulines:2015:supplement} in  \cite[\Cref{sapp:optimal_strategy}-\Cref{sapp:optimal_strategyD}]{durmus:moulines:2015:supplement}.

For simplicity, consider sequences $(\gamma_k)_{k\geq 1}$ defined for
all $k \geq 1$ by $\gamma_k = \gamma_1 /k^{\alpha}$, for $\gamma_1 <
1/(m+L)$ and $\alpha \in \ooint{0,1}$.  Then for $n \geq 1$,
$u_n^{(1)} = \bigO(\rme^{-\kappa \Gamma_n /2})$, $u_n^{(2)} = d
\bigO(n^{-\alpha})$ and $u_n^{(3)} = d^2 \bigO(n^{-2 \alpha})$ (see
\cite[~\Cref{sapp:explicit_bound_k_alpha}-\Cref{sapp:explicit_bound_k_alphaD}]{durmus:moulines:2015:supplement}
for details).  For $\gamma_k = \gamma_1 /k$, we need to extend \Cref{theo:distance_Euler_diffusion} and \Cref{theo:distance_Euler_diffusionD} to non-increasing sequence such that there exists $n_1\geq 1$ such that $\gamma_{n_1} < 1/(m+L)$. It is done in \cite[\Cref{theo:distance_Euler_diffusion_annexe} in \Cref{sec:gener-crefth-1}]{durmus:moulines:2015:supplement}. Using this result in \cite[\Cref{sec:explicit-bound-based_gamma_1}]{durmus:moulines:2015:supplement}, we get that under \Cref{assum:regularity_2} and \Cref{assum:potentialU}, that $W^2_2(\delta_{\xstar} Q^n_{\gamma},\pi) = \bigO(n^{-1})$ for $\gamma_1 > 2 \kappa^{-1}$. If in addition \Cref{assum:reg_plus} holds, we have $W^2_2(\delta_{\xstar} Q^n_{\gamma},\pi) = \bigO(n^{-1})$ for $\gamma_1 > 4 \kappa^{-1}$. However, note that the constants are exponential in $\gamma_1$.
The conclusions of this discussion are summarized in \Cref{tab:order_convergence_k_alpha}.


Note that these rates are explicit compared to those reported in \cite[Proposition 3]{durmus:moulines:2016}.
In addition, two regimes can be observed as in stochastic
approximation in the case $\alpha =1$.

\begin{table}[h]
\centering
\begin{normalsize}
\begin{tabular}{|c|c|c|}
    \hline
    &$\alpha \in \ooint{0,1}$ & $\alpha =1$ \\
    \hline
 \Cref{theo:distance_Euler_diffusion}& $d \, \bigO( n^{-\alpha})$& $d  \,  \bigO( n^{-1})$ for $\gamma_1 > 2\kappa^{-1}$ see \cite[\Cref{sec:explicit-bound-based_gamma_1}]{durmus:moulines:2015:supplement}  \\
   \hline
 \Cref{theo:distance_Euler_diffusionD}& $ d^2 \,  \bigO( n^{-2\alpha})$ & $d^2  \,  \bigO( n^{-2})$ for $\gamma_1 > 4\kappa^{-1}$ see \cite[\Cref{sec:explicit-bound-based_gamma_1}]{durmus:moulines:2015:supplement} \\
   \hline
\end{tabular}
\end{normalsize}
\caption{\normalsize{Order of convergence of $W_2^2(\delta_{\xstar} Q^n_\gamma,\pi)$ for $\gamma_k = \gamma_1 /k^{\alpha}$ }}
\label{tab:order_convergence_k_alpha}
\end{table}
Details and further discussions are included in \cite[\Cref{sec:disc-crefth_theo_1}
-\Cref{sec:disc-crefth_theo_2}]{durmus:moulines:2015:supplement}.
In particular, the dependencies of the obtained bounds with respect to
the constants $m$ and $L$ which appear in \Cref{assum:regularity},
\Cref{assum:potentialU} are evidenced.

%% file: total_variation.tex
\newcommand{\defEnsPointDeux}[1]{\left. #1 \right  \rbrace}

\section{Quantitative bounds in total variation distance}
\sectionmark{Quantitative bounds in TV}
\label{sec:quant-bounds-total}
We develop in this section quantitative bounds in total variation
distance. For Bayesian inference application, total variation bounds
are useful for computing highest posterior density (HPD)
credible regions and intervals. For computing such bounds we will use
the results of \Cref{sec:non-asympt-bounds} combined with the
regularizing property of the semigroup $(P_t)_{t \geq 0}$.

The first key result consists in upper-bounding the total variation distance $\tvnorm{\mu P_t- \nu P_t}$ for $\mu,\nu \in \Pens_1(\rset^d)$.
To that purpose, we use the coupling by reflection; see \cite[Section 3]{lindvall:rogers:1986} or
\cite[Example 3.7]{chen:shao:1989} for its construction, and \cite{eberle:2015,egz-16,Bubeck:2015} for applications.
It is defined as the
unique strong solution $(\XSDE_t,\YSDE_t)_{ t \geq 0}$  of the SDE:
\begin{equation}
  \label{eq:def_couplage_par_reflection}
\begin{cases}
  \rmd \XSDE_t &= -\nabla U(\XSDE_t) \rmd t + \sqrt{2}\rmd B_t^d \\
 \rmd \YSDE_t &= -\nabla U(\YSDE_t) \rmd t +  \sqrt{2} (\Id - 2 e_t e_t^T)
  \rmd B_t^d   \eqsp,
\end{cases}
\quad \text{ where } e_t =\fune(\XSDE_t -\YSDE_t)
\end{equation}
with $\XSDE_0 = x$, $\YSDE_0=y$, $\fune(z) = z/\norm{z}$ for $z \not=0$ and $\fune(0) = 0$ otherwise.
Define the coupling time $T_c =  \inf \{s \geq 0 \ | \ \XSDE_s  = \YSDE_s \}$.
By construction $\XSDE_t = \YSDE_t$ for $t\geq T_c$.  Using Levy's
characterization, $\tilde{B}_t^d = \int_0^t (\Id - 2 e_s
e_s^T) \rmd B_s^d$ is a $d$-dimensional Brownian motion, therefore
$(\XSDE_t)_{t \geq 0}$ and $(\YSDE_t)_{t \geq 0}$ are weak solutions
to \eqref{eq:langevin_2} started at $x$ and $y$ respectively. Then by Lindvall's inequality, for all $t >0$ we have
$\tvnorm{P_t(x,\cdot) - P_t(y,\cdot)} \leq \PP \left( \XSDE_t \not = \YSDE_t \right)$.

Denote by $\Phibf$  the cumulative distribution function of the standard normal distribution. For $a > 0$,  define $\rated_a$ for all $t \geq 0$ by
\begin{equation}
\label{eq:def_rated}
  \rated_a(t) = \sqrt{(4/a)(\rme^{2at}-1)} \eqsp.
\end{equation}

\begin{theorem}
  \label{propo:reflexion_coupling_lip}
  Assume \Cref{assum:regularity_2} and \Cref{assum:potentialU}.
  \begin{enumerate}[label=(\roman*)]
  \item
  \label{item:reflexion_coupling_lip_1}
For any $x,y \in \rset^d$ and $t > 0$, it holds
 \[
 \tvnorm{P_t(x,\cdot) - P_t(y,\cdot)} \leq   1 - 2 \Phibf\{-\norm{x-y}/\rated_m(t)\} \eqsp,
 \]
 where $\rated_m$ is defined in \eqref{eq:def_rated} and $m$ is the strong convexity constant.
\item
\label{item:coro:lip_smigroup_ii}
For any   $\mu,\nu \in \Pens_1(\rset^d)$ and  $t >0$,
 \begin{equation*}
     \tvnorm{\mu P_t- \nu P_t} \leq 2^{1/2}\fracmiddle{ W_1(\mu,\nu)}{ (\uppi^{1/2}\chi_m(t))}  \eqsp.
   \end{equation*}
  \item
  \label{item:reflexion_coupling_lip_2}
For any $x \in \rset^d$ and $t \geq 0$,
   \begin{equation*}
  \tvnorm{\pi-\delta_x P_t}    \leq \fracmiddle{2^{1/2} \defEns{(d/m)^{1/2} +  \norm{x-\xstar}}}{ (\uppi^{1/2}\chi_m(t))} \eqsp.
   \end{equation*}
  \end{enumerate}
\end{theorem}
\begin{proof}
\begin{enumerateList}
\item Denote for $t > 0$, $\mathsf{B}_t^1 = \int_0^t \1_{\{s < T_c\}} e_s^T \rmd B_s^d$.
 We compute a bound for the coupling time. On  $\{t < T_c\}$, by \eqref{eq:def_couplage_par_reflection}, we get
\begin{equation*}
  \rmd \{ \XSDE_t - \YSDE_t \} = -\defEns{ \nabla U (\XSDE_t)- \nabla U(\YSDE_t)} \rmd t + 2 \sqrt{2}  e_t \rmd \mathsf{B}_t^1 \eqsp.
\end{equation*}
Itô's formula on $\{t < T_c\}$ yields
\begin{multline*}
 \rme^{m t }  \norm{\XSDE_t-\YSDE_t} =    \norm{x-y} + m \int_0^t \rme^{m s} \norm{\XSDE_s - \YSDE_s} \rmd s \\- \int_0^{t}  \rme^{m s }  \ps{\nabla U(\XSDE_s)-\nabla U (\YSDE_s)}{e_s} \rmd s +2\sqrt{2} \int_0^t \rme^{m s }  \rmd \mathsf{B}_s^1 \eqsp.
\end{multline*}
Then by \Cref{assum:potentialU}, we obtain on $\{t < T_c\}$, $  \norm{\XSDE_t-\YSDE_t} \leq \mathsf{U}_t $,
where $(\mathsf{U}_t)_{t \in \ooint{0,T_c}}$ is the one-dimensional Ornstein-Uhlenbeck process defined by
\begin{equation*}
  \mathsf{U}_t = \rme^{-m t}\norm{x-y} + 2 \sqrt{2} \int_0^t \rme^{m(s-t)} \rmd \mathsf{B}_s^1 \eqsp.
\end{equation*}
Therefore, for all $x,y \in \rset^d $ and $t \geq 0$, we get $$\PP (T_c > t ) \leq  \PP \parenthese{ \min_{0 \leq s \leq t} \mathsf{U}_s  > 0 } \eqsp.$$
Finally the proof follows from \cite[Formula 2.0.2, page 542]{borodin:salminen:2002}. For completeness, this formula is given in \Cref{sec:distr-hitt-time}.
\item Let $\mu,\nu \in \Pens_1(\rset^d)$ and $\xi \in
  \couplage{\mu}{\nu}$ be an optimal transference plan for $(\mu,\nu)$
  \wrt~$W_1$.  Since for all $s >0$, $1/2-\Phibf(-s) \leq (2
  \uppi)^{-1/2} s$, \ref{item:reflexion_coupling_lip_1} implies that
  for all $x,y \in \rset^d$ and $t >0$,
  \begin{equation*}
    \tvnorm{\mu P_t - \nu P_t} 
\leq 2 \int_{\rset^d \times \rset^d} \frac{\norm{x-y}}{\sqrta{2 \uppi} \chi_m(t) } \,  \rmd \xi( x, y)\eqsp,
  \end{equation*}
  which is the desired result.
\item The proof is a straightforward consequence of \ref{item:coro:lip_smigroup_ii} and
\Cref{theo:convergence-WZ-strongly-convex}-\ref{item:2:theo:convergence-WZ-strongly-convexD}.
\end{enumerateList}
\end{proof}
Since for all $s >0$,  $s
\leq \rme^s -1$, note that
\Cref{propo:reflexion_coupling_lip}-\ref{item:coro:lip_smigroup_ii}
implies that for all $t >0$ and $\mu,\nu \in \Pens_1(\rset^d)$,
\begin{equation}
\label{eq:item:coro:lip_smigroup_ii}
    \tvnorm{\mu P_t- \nu P_t} \leq (4 \uppi t)^{-1/2} W_1(\mu,\nu)  \eqsp.
  \end{equation}
Therefore for all bounded measurable function $f$, $P_tf$ is a Lipschitz function for all $t > 0$ with Lipshitz constant
\begin{equation}
\label{eq:lipshitz-constant-Ptf}
\Vnorm[\Lip]{P_t f} \leq (4 \uppi t)^{-1/2} \osc{f} \eqsp.
\end{equation}

We will now study the contraction of  $\QKer_{\gaStep}^{n,\ell}$ in total variation for non-increasing sequences $(\gaStep_k)_{k \geq 1}$. Strikingly, we are able to derive results which closely parallel \Cref{propo:reflexion_coupling_lip}. The proof is nevertheless completely different because the reflection coupling is no longer applicable in discrete time.  We use a coupling construction inspired by the method of \cite[Section~3.3]{bubley:dyer:jerrum:1998} for Gaussian random walks.  This construction has been used in \cite{eberle:2016} to establish convergence of homogeneous Markov chain in Wasserstein distances using different method of proof. So as not to interrupt the argument, this construction is postponed to \Cref{sec:conv-total-vari_AR}.

For all $n,\ell \geq 1$, $n < \ell$ and $(\gaStep_k)_{k \geq 1}$  a non-increasing sequence denote by
\begin{equation}
  \label{eq:def_xirmse}
\LambdarMSE_{n,\ell}(\gamma) =   \kappa^{-1} \defEns{\prod_{j=n}^{\ell}(1-\kappa \gaStep_j)^{-1} -1} \eqsp, \qquad \LambdarMSE_{\ell}(\gamma) = \LambdarMSE_{1,\ell}(\gamma) \eqsp.
\end{equation}

\begin{theorem}
  \label{theo:convergence_discrete_chain}
  Assume \Cref{assum:regularity_2} and \Cref{assum:potentialU}.
  \begin{enumerate}[label=(\roman*)]
  \item
\label{item:convergence_discrete_chain_1}
 Let $(\gaStep_k)_{k \geq 1}$ be a non-increasing sequence satisfying $\gaStep_1 \leq 2/(m+L)$. Then for all $x,y \in \rset^d$ and $n,\ell \in \nset^*$, $n < \ell$, we have $$\tvnorm{\delta_x \QKer_{\gaStep}^{n,\ell}-\delta_y \QKer_{\gaStep}^{n,\ell}} \leq 1-2 \Phibf\{- \norm{x-y}/\{8\, \LambdarMSE_{n,\ell}(\gamma)\}^{1/2} \} \eqsp.$$
\item
\label{item:convergence_discrete_chain_3}
 Let $(\gaStep_k)_{k \geq 1}$ be a non-increasing sequence satisfying $\gaStep_1 \leq 2/(m+L)$.
Then, for all $\mu,\nu \in \Pens_1(\rset^d)$ and $\ell, n \in \nset^*$, $n < \ell$, we have
 \begin{equation*}
     \tvnorm{\mu \QKer_{\gaStep}^{n,\ell}-\nu \QKer_{\gaStep}^{n,\ell}} \leq \{4 \uppi \LambdarMSE_{n,\ell}(\gamma)\}^{-1/2} W_1(\mu,\nu) \eqsp.
 \end{equation*}
\item
\label{item:convergence_discrete_chain_2}
 Let $\gaStep \in \ocint{0, 2/(m+L)}$. Then for any $x\in \rset^d$ and $n \geq 1$,
  \begin{equation*}
    \tvnorm{\pi_{\gamma} - \delta_x R_{\gamma}^n} \leq \{4 \uppi \kappa(1-(1-\kappa \gamma)^{n/2})   \}^{-1/2} (1-\kappa \gamma)^{n/2} \defEns{\norm{x-\xstar} + (2\kappa^{-1}d)^{1/2}} \eqsp.
  \end{equation*}
\end{enumerate}
\end{theorem}
\begin{proof}
\begin{enumerateList}
\item By \eqref{eq:convex_forte_contra1} for all $x,y$ and $k \geq 1$, we have
\begin{equation*}
 \norm{ x-\gaStep_k \nabla U(x) - y + \gaStep_k \nabla U(y) }\leq (1-\kappa \gaStep_k )^{1/2} \norm{x-y} \eqsp.
\end{equation*}
Let $n,\ell \geq 1$, $n < \ell$,
then applying \Cref{theo:strict_convergence_AR} in \Cref{sec:conv-total-vari_AR}, we get
\begin{equation*}
    \tvnorm{\delta_x \QKer_{\gaStep}^{n,\ell}-\delta_y \QKer_{\gaStep}^{n,\ell}} \leq 1-2 \Phibf\parenthese{- \norm{x-y}/\{8\, \LambdarMSE_{n,\ell}(\gamma)\}^{1/2} } \eqsp,
\end{equation*}
\item Let $f \in \functionspace[b]{\rset^d}$ and $ \ell > n \geq 1$.
For all $x,y \in \rset^d$ by definition of the total variation distance and \ref{item:convergence_discrete_chain_1}, we have
\begin{align*}
\abs{  Q^{n,\ell}_\gamma f(x) - Q^{n,\ell}_\gamma f(y) } &\leq  \osc{f} \tvnorm{\delta_x Q^{n,\ell}_\gamma - \delta_y Q^{n,\ell}_\gamma}  \\
& \leq \osc{f} \defEns{1-2 \Phibf\parenthese{- \norm{x-y}/\{8\, \LambdarMSE_{n,\ell}(\gamma)\}^{1/2} }} \eqsp,
\end{align*}
Using that  for all $s >0$, $1/2-\Phibf(-s) \leq (2 \uppi)^{-1/2} s$ concludes the proof.
\item The proof follows from \ref{item:convergence_discrete_chain_2}, the bound for all $s >0$, $1/2-\Phibf(-s) \leq (2 \uppi)^{-1/2} s$ and \Cref{theo:kind_drift}-\ref{item:kind_drift_2}.
\end{enumerateList}
\end{proof}

  We can combine  \Cref{theo:distance_Euler_diffusion} or
  \Cref{theo:distance_Euler_diffusionD} with \Cref{propo:reflexion_coupling_lip} and \Cref{theo:convergence_discrete_chain} to obtain
  explicit bounds in total
  variation between the Euler-Maruyama discretization and the target
  distribution $\pi$. To that purpose, we  use the following
  decomposition, for all non-increasing sequence $(\gaStep_k)_{k \geq 1}$, initial point $x \in \rset^d$ and $\ell \geq 0$:
  \begin{equation}
\label{eq:decomposition_tv}
    \tvnorm{\pi-\delta_x \QKer_{\gaStep}^\ell}
    \leq \tvnorm{\pi -\delta_x P_{\GaStep_\ell}}
    + \tvnorm{\delta_x  P_{\GaStep_\ell}- \delta_x  \QKer_{\gaStep}^\ell} \eqsp.
  \end{equation}
The first term is dealt with \Cref{propo:reflexion_coupling_lip}-\ref{item:reflexion_coupling_lip_2}.
It remains to bound the second term in
\eqref{eq:decomposition_tv}. Since we will use
\Cref{theo:distance_Euler_diffusion} and
\Cref{theo:distance_Euler_diffusionD},  we have two different results depending on the assumptions on
 $U$.
Define for all $x \in \rset^d$ and $n,p \in \nset$,
\begin{align}
\label{eq:defrateFun_1}
\rateFun^{(1)}_{n,p}(x) =    L^2 \sum_{i=1}^n \gamma_i^2  \left.\prod_{k=i+1}^n(1- \kappa\gamma_k/2) \left[\defEns{\kappa^{-1} + \gamma_i} (2d + d L^2 \gamma_i^2/6) \right. \right.\\
\nonumber
 \left. \left. +L^2 \gamma_i \delta_{i,n,p}(x)   \defEns{\kappa^{-1} + \gamma_{i}}  \right] \right.
\end{align}
\begin{align}
\label{eq:defrateFun_2}
\rateFun_{n,p}^{(2)}(x) =  \sum_{i=1}^n \gamma_i^3 \prod_{k=i+1}^n(1- \kappa\gamma_k/2)
\left[ L^4\delta_{i,n,p}(x) (4 \kappa^{-1}/3+\gamma_{n+1}) \right. \\
\nonumber
\left. +   d\defEns{2L^2+ 4 \kappa^{-1}(d \tilde{L}^2/3 + \gamma_{n+1}L^4/4) + \gamma_{n+1}^2 L^4 /6 } \right] \eqsp,
\end{align}
where
\begin{equation*}
  \delta_{i,n,p}(x) = \rme^{-2m \Gamma_{i-1}}\FunMomentDEuler_{n,p}(x) + (1-\rme^{-2m \Gamma_{i-1}} )(d/m) \eqsp,
\end{equation*}
and $\FunMomentDEuler_{n,p}(x)$ is given by \eqref{eq:FunMomentDEuler}.
  \begin{theorem}
    \label{theo:tv_decreasing}
Assume \Cref{assum:regularity_2} and \Cref{assum:potentialU}.
Let $(\gaStep_k)_{k \geq 1}$  be a non-increasing sequence with $\gaStep_1 \leq 1/(m+L)$.
Then for all $x \in \rset^d$ and $\ell,n \in \nset^*$, $\ell > n$,
\begin{multline}
  \label{eq:bound_tv_first}
  \tvnorm{\delta_x P_{\GaStep_\ell} - \delta_x \QKer_{\gaStep}^\ell} \leq
  (\rateFun_n(x) /( 4\uppi \GaStep_{n+1,\ell}))^{1/2}  \\
+2^{-3/2} L \parenthese{ \sum_{k=n+1}^{\ell} \defEns{(\gaStep_{k}^3 L^2/3) \FunMomentDEuler_{1,k-1}(x)  + d \gaStep_{k}^2}}^{1/2}  \eqsp,
\end{multline}
where $\FunMomentDEuler_{1,n}(x)$ is defined by \eqref{eq:FunMomentDEuler}, $\rateFun_n(x)$ is equal to $\rateFun^{(2)}_{n,0}(x)$ given by \eqref{eq:defrateFun_2}, if  \Cref{assum:reg_plus} holds, and to $\rateFun^{(1)}_{n,0}(x)$ given by \eqref{eq:defrateFun_1} otherwise.
  \end{theorem}
  \begin{proof}
The proof is postponed to \Cref{sec:proof-crefth_tv_decreasing}.
  \end{proof}
Consider the case of decreasing step sizes of the form $\gaStep_k =
\gaStep_1 /k^{\alpha}$ for $k \geq 1$ and $\alpha \in
\ooint{0,1}$. Under \Cref{assum:regularity_2} and
\Cref{assum:potentialU}, setting $n =\ell- \floor{\ell^{\alpha}}$, $\ell \in \nset^*$, we have for $i=2,3$,
\begin{equation}
  \label{eq:bound_tv_non-increasing_proof_1}
  \lim_{n\to \plusinfty} \Gamma_{n,\ell} = 1 \eqsp, \sum_{k=n+1}^{\ell}  \gaStep_{k}^{i}  \leq  \gaStep_{n+1}^i (\ell-n) \leq \gaStep_1^i \floor{\ell^{\alpha}}/(\ell-\floor{\ell^{\alpha}})^{i \alpha} \eqsp.
\end{equation}
In addition, by \Cref{tab:order_convergence_k_alpha}, $\vartheta_n(x) = d\bigO(\ell^{-\alpha})$. Therefore combining this result and \eqref{eq:bound_tv_non-increasing_proof_1} in the bound of \Cref{theo:tv_decreasing}, we get that  $
\tvnorm{\delta_{\xstar} \QKer_{\gaStep}^\ell-\pi} = d^{1/2} \bigO(
\ell^{-\alpha/2})$. In the case $\gaStep_k =
\gaStep_1 /k^{\alpha}$ for $k \geq 1$ and $\alpha =1$, setting $n =\ell- \floor{\ell/2}$, $\ell \in \nset^*$, $\ell >2$, we have for $i=2,3$,
\begin{equation}
  \label{eq:bound_tv_non-increasing_proof_alpha_1}
  \lim_{n\to \plusinfty} \Gamma_{n,\ell} = 1/2 \eqsp, \sum_{k=n+1}^{\ell}  \gaStep_{k}^{i}  \leq  \gaStep_{n+1}^i (\ell-n) \leq \gaStep_1^i/(\ell/2-1) \eqsp.
\end{equation}
In addition, by \Cref{tab:order_convergence_k_alpha},
$\vartheta_n(x) = d\bigO(\ell^{-1})$, for
$\gaStep_1 > 2 \kappa^{-1}$. Therefore combining this result and
\eqref{eq:bound_tv_non-increasing_proof_alpha_1} in the bound of
\Cref{theo:tv_decreasing}, we get that
$ \tvnorm{\delta_{\xstar} \QKer_{\gaStep}^\ell-\pi} = d^{1/2} \bigO(
\ell^{-1/2})$.

Note that these rates for $\gamma_k=\gaStep_1/k^{\alpha}$, $k \in \nset^*$ and $\alpha \in \ocint{0,1}$ improve those obtained in \cite[Proposition
3]{durmus:moulines:2016},  for potentials satisfying
\Cref{assum:regularity_2} but not necessarily convex since \cite[Proposition
3]{durmus:moulines:2016} only requires the additional assumption that $(P_t)_{t \geq 0}$ is geometrically ergodic in total variation.

Assume  \Cref{assum:regularity_2},
\Cref{assum:potentialU} and \Cref{assum:reg_plus} and that $\gaStep_k =
\gaStep_1 /k^{\alpha}$ for $k \geq 1$ and $\alpha \in
\ocint{0,1}$.  setting $n =\ell- \floor{\ell^{\alpha/2}}$, $\ell \in \nset^*$, we have for $i=2,3$,
\begin{equation}
  \label{eq:bound_tv_non-increasing_proof_2}
  \lim_{n\to \plusinfty} \Gamma_{n,\ell} = 1 \eqsp, \sum_{k=n+1}^{\ell}  \gaStep_{k}^{i}  \leq  \gaStep_{n+1}^i (\ell-n) \leq \gaStep_1^i \floorLigne{\ell^{\alpha/2}}/(\ell-\floorLigne{\ell^{\alpha/2}})^{i \alpha} \eqsp.
\end{equation}
In addition (see \Cref{tab:order_convergence_k_alpha}) $\vartheta_n(x) = d^2\bigO(\ell^{-2\alpha})$, with $\gaStep_1 > 4 \kappa^{-1}$ in the case $\alpha=1$. Therefore combining this result and \eqref{eq:bound_tv_non-increasing_proof_2} in the bound of \Cref{theo:tv_decreasing}, we get that  $
\tvnorm{\delta_{\xstar} \QKer_{\gaStep}^\ell-\pi} = d^{1/2} \bigO(
\ell^{-3 \alpha/4})$. These discussions are summarized in
\Cref{tab:order_convergence_k_alpha_tv}.

\begin{table}[h]
\centering
\begin{normalsize}
\begin{tabular}{|c|c|c|}
    \hline
    &$\alpha \in \ooint{0,1}$ & $\alpha =1$ \\
    \hline
 \Cref{theo:distance_Euler_diffusion}& $d^{1/2} \, \bigO( \ell^{-\alpha/2})$& $d^{1/2}  \,  \bigO( \ell^{-1/2})$ for $\gamma_1 > 2\kappa^{-1}$  \\
   \hline
 \Cref{theo:distance_Euler_diffusionD}& $ d^{1/2} \,  \bigO( \ell^{-3\alpha/4})$ & $d^{1/2}  \,  \bigO( \ell^{-3/4})$ for $\gamma_1 > 4\kappa^{-1}$  \\
   \hline
\end{tabular}
\end{normalsize}
\caption{\normalsize{Order of convergence of $\tvnorm{\delta_{\xstar} Q^\ell_\gamma-\pi}$ for $\gamma_k = \gamma_1 /k^{\alpha}$ based on \Cref{theo:tv_decreasing} }}
\label{tab:order_convergence_k_alpha_tv}
\end{table}

When $\gaStep_k = \gaStep \in \ooint{0,1/(m+L)}$ for all
$k \geq 1$, under \Cref{assum:regularity_2} and \Cref{assum:potentialU}, for $\ell > \ceil{\gaStep^{-1}}$ choosing
$n=\ell-\ceil{\gaStep^{-1}}$ implies that (see  \Cref{sec:proof-eqref})
\begin{equation}
\label{eq:borne_tv_1_fixed_step_size}
  \tvnorm{\delta_{x} \RKer^{\ell}_{\gaStep}- \delta_x \PKer_{\ell \gaStep}}
\leq (4 \uppi)^{-1/2}\parentheseDeux{\gamma \constD_1(\gamma,d) + \gamma^3 \constD_2(\gamma) \constD_3(\gamma,d,x) }^{1/2}  + \constD_4(\gamma,d,x) \eqsp,
\end{equation}
where
\begin{align}
\label{eq:def_D_1_2}
&\constD_1(\gamma,d)  = 2 L^2 \kappa^{-1} \parenthese{\kappa^{-1}+ \gamma}\parenthese{2d + L^2 \gamma^2/6}  \eqsp, \constD_2(\gamma) =L^4  \parenthese{\kappa^{-1}+ \gamma} \\
\nonumber
&\constD_3(\gamma,d,x)=  \defEns{(\ell-\ceil{\gaStep^{-1}}) \rme^{-m \gaStep(\ell-\ceil{\gaStep^{-1}}-1)}  \norm[2]{x-\xstar} + 2d(\kappa \gamma m)^{-1} }\\
\nonumber
&\constD_4(\gamma,d,x)= 2^{-3/2} L \left[ d \gamma (1+\gamma)  \right.\\
\nonumber
& \left. + (L^{2}\gaStep^3/3)\defEns{(1+\gaStep^{-1} )(1-\kappa \gamma)^{\ell-\ceil{\gaStep^{-1}}}\norm[2]{x-\xstar} + 2(1+\gamma) \kappa^{-1}  d} \right]^{1/2} \eqsp.
\end{align}
Using this bound and
\Cref{propo:reflexion_coupling_lip}-\ref{item:reflexion_coupling_lip_2},
the number of iterations $\ell_{\varepsilon} >0$ to achieve $\tvnorm{\delta_{\xstar}
  \RKer_{\gaStep_{\varepsilon}}^{\ell_{\varepsilon}} - \pi} \leq
\varepsilon$ is of order 
$d \log(d)\bigO(\abs{ \log(\varepsilon)}
\varepsilon^{-2})$ (the proper choice of the  step size
$\gamma_{\varepsilon}$ is given in \Cref{tab:order_tv_precision}). This result is the same than the one obtained in \cite{durmus:moulines:2016}.

Letting $\ell$ go to infinity in
\eqref{eq:borne_tv_1_fixed_step_size} we get the following result.
\begin{corollary}
\label{theo:bias_tv_gamma_fixed}
Assume \Cref{assum:regularity_2} and \Cref{assum:potentialU}.
Let $\gaStep \in \ocint{0,1/(m+L)}$. Then it holds
\begin{multline*}
  \tvnorm{\pi_{\gaStep}-\pi} \leq 2^{-3/2} L\parentheseDeux{ d \gamma(1+\gamma) + 2 (L^{2}\gaStep^3/3) (1+\gamma) \kappa^{-1}  d }^{1/2}
\\
+(4 \uppi)^{-1/2}\parentheseDeux{\gamma \constD_1(\gamma,d) + 2 d \gamma^2 \constD_2(\gamma) (\kappa m)^{-1} }^{1/2} \eqsp,
\end{multline*}
where $\constD_1(\gamma)$ and $\constD_2(\gamma)$ are given in \eqref{eq:def_D_1_2}.
\end{corollary}
Note that \Cref{theo:bias_tv_gamma_fixed} shows that
$\norm{\pi_{\gamma}-\pi}_{V^{1/2}} \leq C_1 \gaStep^{1/2}$ for some
constant $C_1 \geq 0$.  Under \Cref{assum:regularity_2} and the
assumption and $R_\gamma$ and $(P_t)_{t \geq 0}$ are $V$-uniformly
geometrically ergodic, \cite[Theorem 10]{durmus:moulines:2016}
establishes that
$\norm{\pi_{\gamma}-\pi}_{V^{1/2}} \leq C_2 \gaStep^{1/2}$ for some
explicit constant $C_2 \geq 0$. In the case where $U$ satisfies \Cref{assum:potentialU}, then we can take $V = \norm{\cdot}^2$ and $C_2$ is very similar to $C_1$. In particular both $C_1$ and $C_2$ are of order $d^{1/2}$.

However, if  \Cref{assum:reg_plus} holds, for constant step sizes,
we can improve with respect to the step size $\gamma$, the bounds given by \Cref{theo:bias_tv_gamma_fixed}.
\begin{theorem}
\label{theo:bias_tv_gamma_fixed_imp}
Assume \Cref{assum:regularity_2}, \Cref{assum:potentialU} and \Cref{assum:reg_plus}.
Let $\gaStep \in \ocint{0,1/(m+L)}$. Then it holds
\begin{multline*}
  \tvnorm{\pi_{\gaStep}-\pi} \leq  ( 4 \uppi)^{-1/2}\defEns{
\gamma^2 \constE_1(\gamma,d) + 2d \gamma^2 \constE_2(\gamma)/ (\kappa m)  }^{1/2} \\
+ ( 4 \uppi)^{-1/2} \ceil{\log\parenthese{\gaStep^{-1}}/ \log(2)} \defEns{\gamma^2  \constE_1(\gamma,d)  +  \gamma^2  \constE_2(\gamma) ( 2\kappa^{-1}d+d/m ) }^{1/2} \\
+2^{-3/2} L  \defEns{  2d \gaStep^3L ^2/(3\kappa) +d \gaStep^2 }^{1/2} \eqsp,
\end{multline*}
where $\constE_1(\gamma,d)$ and $\constE_2(\gamma)$ are defined by
\begin{align*}
\constE_1(\gamma,d) &= 2 d \kappa^{-1}  \defEns{2L^2+ 4 \kappa^{-1}(d \tilde{L}^2/3 + \gamma L^4/4) + \gamma^2 L^4 /6 } \\
 \constE_2(\gamma) &= L^4(4 \kappa^{-1}/3 + \gamma) \eqsp.
\end{align*}
\end{theorem}
  \begin{proof}
The proof is postponed to \Cref{sec:proof-theo_tv_fix}.
  \end{proof}

  Note that the bound provided by \Cref{theo:bias_tv_gamma_fixed_imp} is of order $d\bigO(\gaStep \abs{\log(\gaStep)})$, improving the dependency given by \Cref{theo:bias_tv_gamma_fixed} and \cite[Theorem 10]{durmus:moulines:2016},  with respect to the step size $\gamma$, but  \Cref{theo:bias_tv_gamma_fixed_imp} requires that \Cref{assum:reg_plus} holds contrary to \Cref{theo:bias_tv_gamma_fixed} and \cite[Theorem 10]{durmus:moulines:2016}. 
  Furthermore when $\tilde{L} = 0$, this bound given by \Cref{theo:bias_tv_gamma_fixed_imp}  is of order
  $d^{1/2}\bigO(\gaStep \abs{\log(\gaStep)})$ and  is sharp up to a logarithmic factor.
Indeed, assume that $\pi$ is the $d$-dimensional standard Gaussian
distribution. In such case, the ULA sequence $(X_k)_{k \geq 0}$ is the autoregressive process given for all $k \geq 0$ by $X_{k+1}= (1 - \gaStep) X_k + \sqrt{2 \gaStep} Z_{k+1}$.
For $\gaStep \in \ooint{0,1}$, this sequence has a stationary distribution $\pi_\gaStep$, which is  a $d$-dimensional Gaussian distribution with zero-mean and covariance
matrix $\sigma_{\gaStep}^{2}\IdM$, with $\sigma_{\gaStep}^2=(1-\gaStep/2)^{-1}$. Therefore, using \cite[Lemma~4.9]{klartag:2007} (or the Pinsker inequality), we get the following upper bound: $\tvnorm{\pi-\pi_{\gaStep}} \leq C d^{1/2} | \sigma_\gaStep^2 -1| = C  d^{1/2} \gaStep/2$, where $C$ is a universal constant.

  We can also for a precision target $\varepsilon >0$ choose
  $\gaStep_{\varepsilon}>0$ and the number of iterations
  $n_{\varepsilon} >0$ to get $\tvnorm{\delta_x
    \RKer_{\gaStep_{\varepsilon}}^{n_{\varepsilon}} - \pi} \leq
  \varepsilon$.  By
  \Cref{propo:reflexion_coupling_lip}-\ref{item:reflexion_coupling_lip_2},
  \Cref{theo:convergence_discrete_chain}-\ref{item:convergence_discrete_chain_2}
  and \Cref{theo:bias_tv_gamma_fixed_imp}, a sufficient number of
  iterations $\ell_{\varepsilon}$ is of order $d \log^2(d)
  \bigO( \varepsilon^{-1}\log^{2}(\varepsilon))$ for a well chosen
  step size $\gamma_{\varepsilon}$. This result improves the conclusion of \cite{durmus:moulines:2016} and \Cref{theo:bias_tv_gamma_fixed} with respect to the precision parameter $\varepsilon$, which provides an upper bound of the number of iterations of order $d \log(d)
  \bigO( \varepsilon^{-2}\log^{2}(\varepsilon))$. We can also compare our reported upper bound
  with the one obtained for the $d$-dimensional standard
  Gaussian distribution. If the initial distribution is the Dirac mass at zero (the minimum of the potential $U(x)= \norm{x}^2/2$) and $\gaStep \in \ooint{0,1}$,
  the distribution of the ULA sequence after $n$ iterations is zero-mean Gaussian with covariance $(1 - (1 - \gaStep)^{2(n+1)}) / (1- \gaStep/2) \IdM$. If we use \cite[Lemma~4.9]{klartag:2007} again, we get for $\gaStep \in \ooint{0,1}$,
  \begin{equation*}
  \tvnorm{\delta_0 \RKer_{\gaStep}^n - \pi} \leq C d^{1/2} \gaStep  |1 -2 \gaStep^{-1}(1-\gaStep)^{2(n+1)}|
   \eqsp,
  \end{equation*}
  where $C$ is a universal constant. To get an $\varepsilon$ precision we need to choose $\gaStep_\varepsilon= d^{-1/2} \varepsilon/ (2 C)$ and then $n_\varepsilon= \lceil (1/2) \log(\gaStep_\varepsilon/4)/\log(1-\gaStep_\varepsilon) \rceil = d^{1/2}\log(d) \bigO(\varepsilon^{-1} |\log(\varepsilon)|)$. On the other hand since $\tilde{L}=0$, based on the bound given by \Cref{theo:bias_tv_gamma_fixed_imp}, a sufficient number of iterations to get $\tvnorm{\delta_x
    \RKer_{\gaStep_{\varepsilon}}^{n_{\varepsilon}} - \pi} \leq
  \varepsilon$ is  of order $d^{1/2} \log^2(d)
  \bigO( \varepsilon^{-1}\log^{2}(\varepsilon))$.
  It follows that our upper bound for the step size and the optimal number of iterations is again sharp up to a logarithmic factor in the dimension and the precision.
  The discussions on the bounds for constant sequences of step sizes are
  summarized in \Cref{tab:order_tv_pi_pig} and
  \Cref{tab:order_tv_precision}.
\begin{table}[h]
\centering
\begin{normalsize}
\begin{tabular}{|c|c|c|}
    \hline
    & \Cref{assum:regularity_2}, \Cref{assum:potentialU} &\Cref{assum:regularity_2}, \Cref{assum:potentialU} and \Cref{assum:reg_plus} \\
    \hline
$\tvnorm{\pi-\pi_{\gaStep}}$ & $d^{1/2}\bigO( \gaStep ^{1/2})$& $d \bigO( \gaStep
\abs{  \log(\gaStep)})$   \\
   \hline
\end{tabular}
\end{normalsize}
\caption{\normalsize{Order of the bound between $\pi$ and $\pi_{\gaStep}$ in total variation function of the step size $\gaStep >0$ and the dimension $d$.}}
\label{tab:order_tv_pi_pig}
\end{table}

\begin{table}[h]
\centering
\begin{normalsize}
\begin{tabular}{|c|c|c|}
    \hline
    & \Cref{assum:regularity_2}, \Cref{assum:potentialU} &\Cref{assum:regularity_2}, \Cref{assum:potentialU} and \Cref{assum:reg_plus} \\
    \hline
$\gaStep_{\varepsilon}$ & $d^{-1} \bigO(  \varepsilon^2)$& $d^{-1}
  \log^{-1}(d) \bigO( \varepsilon\abs{ \log^{-1}(\varepsilon)})$   \\
\hline
$n_{\varepsilon}$ &$d \log(d)\bigO(\varepsilon^{-2} \abs{\log(\varepsilon)})$& $d \log^2(d) \bigO( \varepsilon^{-1}\log^{2}(\varepsilon))$  \\
   \hline
\end{tabular}
\end{normalsize}
\caption{\normalsize{Order of the step size $\gaStep_{\varepsilon} >0$ and the number of iterations $n_{\varepsilon} \in \nset^*$ to get $\tvnorm{\delta_{\xstar} \RKer_{\gaStep_{\varepsilon}}^{n_{\varepsilon}} - \pi} \leq \varepsilon$ for $\varepsilon >0$.}}
\label{tab:order_tv_precision}
\end{table}


%% file: MSE_tv.tex
\section{Mean square error and concentration for bounded measurable functions}
\sectionmark{Mean square error and concentration for b. m. functions}
\label{SEC:MSE_TV}

Let $(X_k)_{k \geq 0}$ be the Euler discretization of the Langevin
diffusion \eqref{eq:euler-proposal-2} associated with the sequence of
non-increasing step sizes $(\gamma_k)_{k \geq 1}$.  The result of the
previous section allows us to study the approximation of
$\pi(f)$ by the weighted average estimator
$\hat{\pi}^N_n(f)$ defined, for $f : \rset^d \to \rset$,
$N,n \in \nset$, $n \geq 1$ by
\begin{equation}
 \label{eq:def_GammaN_plusn}
\hat{\pi}^N_n(f)= \sum_{k=N+1} ^{N+n} \weight{k} f(X_k) \eqsp, \quad  \weight{k}= \gamma_{k+1} \Gamma_{N+2 , N+n+1}^{-1}\eqsp.
\end{equation}

In all this section, $\PP_x$ and $\PE_x$ denote the probability and the expectation respectively, induced on $((\rset^{d})^{\nset}, \mathcal{B}(\rset^d)^{\nset})$ by the Markov chain $(X_n)_{n \geq 0}$ started at $x \in \rset^d$.
First we derive a bound on the mean-square error, defined as $$\MSE_f^{N,n} = \expeMarkov{x}{\abs{\hat{\pi}_n^N(f)- \pi(f) }^2}\eqsp,$$  for $f: \rset^d\to \rset$, which is either Lipschitz or measurable and bounded. This quantity can be decomposed as the sum of the squared bias and variance:
\begin{equation*}
\MSE_f^{N,n}  = \left\{ \PE_x[ \estimateur{f} ] - \pi(f) \right\}^2 + \VarDeux{x}{ \estimateur{f} }\eqsp.
\end{equation*}

We first obtain a bound for the bias for $f$  Lipschitz.
 For all $k \in \{ N+1, \dots, N+n \}$, denote by $\xi_k$ the optimal transference plan between $\delta_x Q_\gamma^k$ and $\pi$ for $W_2$, \ie~$W_2^2(\delta_x Q_\gamma^k,\pi) = \int_{\rset^d \times \rset^d} \norm[2]{x-y} \rmd\xi_k( x, y)$. Then by the Jensen inequality and because $f$ is Lipschitz, we have:
 \begin{align}
   \nonumber
\defEns{\PE_x[\estimateur{f}]-\pi(f) }^2
   &= \parenthese{\sum_{k=N+1} ^{N+n} \weight{k} \int_{\rset^d \times \rset^d} \{f(z) - f(y)\} \xi_k (\rmd z , \rmd y) }^2 \\
      \nonumber
&\leq \norm{f}_{\Lip}^2  \sum_{k=N+1}^{N+n} \weight{k} \int_{\rset^d \times \rset^d}  \norm{z-y}^2 \xi_k (\rmd z , \rmd y) \\
\label{theo:bound_biasD}
& \leq \norm{f}_{\Lip}^2  \sum_{k=N+1}^{N+n} \weight{k} W_2^2(\delta_x Q_\gamma^k,\pi) \eqsp.
\end{align}

Similarly, if $f$ is bounded,
\begin{equation*}
\parenthese{\PE_x[\estimateur{f}]-\pi(f) }^2
\leq   \osc{f}^2  \sum_{k=N+1}^{N+n} \weight{k} \tvnorm{ \delta_x\QKer_{\gaStep}^k - \pi}^{2} \eqsp;
\end{equation*}
Using the results of \Cref{sec:non-asympt-bounds,sec:quant-bounds-total}, we can deduce different bounds for the bias,
depending on the assumptions on $U$ and the sequence of step sizes
$(\gaStep_k)_{k \geq 1}$. We now derive a bound for the variance.
We get then two different results depending on the class to which the function  $f$ belongs.
In the case of Lipschitz function, we adapt the proof of \cite[Theorem 2]{joulin:ollivier:2010} for homogeneous Markov chain to our inhomogeneous setting.
\begin{theorem}
\label{theo:var}
Assume \Cref{assum:regularity_2} and \Cref{assum:potentialU}.
Let $(\gamma_k)_{k \geq 1}$ be a non-increasing sequence with $\gamma_1 \leq 2/(m+L)$ and $f : \rset^d \to \rset$ be a Lipschitz function.
Then for all $N \geq 0$ and $ n \geq 1$, we get
$\VarDeuxLigne{x}{\hat{\pi}_n^N(f)} \leq  8 \kappa^{-2} \norm{f}_{\Lip}^2 \Gamma_{N+2,N+n+1}^{-1} v_{N,n}(\gamma)$,
where
\begin{equation}
\label{eq:def_u_n_3}
v_{N,n}(\gamma) \eqdef \defEns{1+\Gamma_{N+2,N+n+1}^{-1}(\kappa^{-1}+2/(m+L))} \eqsp.
\end{equation}
\end{theorem}
\begin{proof}
  The proof is postponed to \Cref{sec:proof-creftheo:var}.
\end{proof}
It is noteworthy to observe that the bound for the variance does not depend on the dimension.
We may now discuss the bounds on the MSE (obtained by combining the bounds for the squared bias \eqref{theo:bound_biasD} from \Cref{theo:distance_Euler_diffusion,theo:distance_Euler_diffusionD}, and the variance \Cref{theo:var}) for step sizes given for $k \geq 1$ by $\gamma_k = \gamma_1/ k^{\alpha}$ where $\alpha \in \ccint{0,1}$ and $\gamma_1 < 1/(m+L)$. Details of these calculations are postponed to \cite[\Cref{sapp:bound_MSE_strongly_convex,sapp:bound_MSE_strongly_convexD}]{durmus:moulines:2015:supplement}.  The order of the bounds (up to numerical constants) of the MSE are summarized in \Cref{tab:bound_MSE_strongly_convex_fixe_gamma} as a function of  $\gamma_1$, $n$ and $N$. Then, we can conclude  that in the infinite horizon setting, it is optimal to take $\alpha =1/2$ under \Cref{assum:regularity_2} and \Cref{assum:potentialU}, and $\alpha =1/3$ under \Cref{assum:regularity_2}, \Cref{assum:potentialU} and \Cref{assum:reg_plus}. Note that  \cite{lamberton:pages:2002} shows also that the optimal value for $\alpha$ is $1/3$ by studying the asymptotic behaviour of $\hat{\pi}^0_n(f)$ as $n \to \plusinfty$ for smooth functions $f :\rset^d \to \rset$. 
\begin{table}[h]
\centering
\begin{normalsize}
\begin{tabular}{|c|c|c|}
\hline
& Bound for the MSE\\
    \hline
    $\alpha =0$ & $\gamma_1 + (\gamma_1 n)^{-1}\defEns{1+\exp(-\kappa \gamma_1 N /2)}$ \\
\hline
$\alpha \in \ooint{0,1/2}$ &$\gamma_1n^{-\alpha} + (\gamma_1 n^{1-\alpha})^{-1}\defEns{1+\exp(-\kappa \gamma_1 N^{1-\alpha} /(2(1-\alpha)))}$\\
\hline
$\alpha =1/2$& $\gamma_1 \log(n) n^{-1/2} + (\gamma_1 n^{1/2})^{-1}\defEns{1+\exp(-\kappa \gamma_1 N^{1/2} /4)}$\\
\hline
 $\alpha \in \ooint{1/2,1}$& $n^{\alpha-1}\parentheseDeux{\gamma_1 + \gamma_1^{-1}\defEns{1+ \exp(-\kappa \gamma_1  N^{1-\alpha} /(2(1-\alpha)))}}$
\\
\hline
 $\alpha =1$& $\bigO(\log(n)^{-1})$ for $\gamma_1 > 2 \kappa^{-1}$
\\
\hline
\end{tabular}
\caption{\normalsize{Bound for the MSE for $\gamma_k = \gamma_1 k^{-\alpha}$ for fixed $\gamma_1$ and $N$ under \Cref{assum:regularity_2} and \Cref{assum:potentialU}}}
\label{tab:bound_MSE_strongly_convex_fixe_gamma}
\end{normalsize}
\end{table}
\begin{table}[h]
\centering
\begin{normalsize}
\begin{tabular}{|c|c|c|}
\hline
& Bound for the MSE\\
    \hline
    $\alpha =0$ & $\gamma_1^2 + (\gamma_1 n)^{-1}\{1+\exp(-\kappa \gamma_1 N /2)\}$ \\
\hline
$\alpha \in \ooint{0,1/3}$ &$\gamma_1^2n^{-2\alpha} + (\gamma_1 n^{1-\alpha})^{-1}\{1+\exp(-\kappa \gamma_1 N^{1-\alpha} /(2(1-\alpha)))\}$\\
\hline
$\alpha =1/3$& $\gamma_1^2 \log(n) n^{-2/3} + (\gamma_1 n^{2/3})^{-1}\{1+\exp(-\kappa \gamma_1 N^{1/2} /4)\}$\\
\hline
 $\alpha \in \ooint{1/3,1}$& $n^{\alpha-1}\parentheseDeux{\gamma_1^2 + \gamma_1^{-1}\{1+\exp(-\kappa \gamma_1  N^{1-\alpha} /(2(1-\alpha)))\}}$
\\
\hline
 $\alpha =1$& $\bigO(\log(n)^{-1})$ for $\gamma_1 > 4 \kappa^{-1}$
\\
\hline
\end{tabular}
\caption{\normalsize{Bound for the MSE for $\gamma_k = \gamma_1 k^{-\alpha}$ for fixed $\gamma_1$ and $N$ under \Cref{assum:regularity_2}, \Cref{assum:potentialU} and \Cref{assum:reg_plus}}}
\label{tab:bound_MSE_strongly_convex_fixe_gammaD}
\end{normalsize}
\end{table}

In the case $\gamma_k = \gamma$ for all $k \in \nset^*$ and  the total number of iterations $n+N$ is held fixed (fixed horizon
setting), we optimize the value of the
step size $\gamma$ but also of the burn-in period $N$ to get an  upper bound of order $n^{-1/2}$ under \Cref{assum:regularity_2} and \Cref{assum:potentialU}, and $n^{-2/3}$ under \Cref{assum:regularity_2}, \Cref{assum:potentialU} and \Cref{assum:reg_plus}.

In the case where $f$ is measurable and bounded, we have the following result.

\begin{theorem}
\label{theo:var_tv}
Assume \Cref{assum:regularity_2} and \Cref{assum:potentialU}.
Let $(\gamma_k)_{k \geq 1}$ be a non-increasing sequence with $\gamma_1 \leq 2/(m+L)$ and $f : \rset^d \to \rset$ be a measurable and bounded function.
Then for all $N \geq 0$, $ n \geq 1$, $x \in \rset^d$, we get
\[
\VarDeuxLigne{x}{\hat{\pi}_n^N(f)} \leq  \osc{f}^2 \{2 \gaStep_1 \GaStep_{N+2,N+n+1}^{-1}+u_{N,n}^{(4)}(\gaStep)\}
\]
\begin{multline}
  \label{eq:u_lap_tv}
 u_{N,n}^{(4)}(\gaStep)=    \sum_{k=N}^{N+n-1} \gaStep_{k+1} \defEns{ \sum_{i=k+2}^{N+n} \frac{\weight{i}}{(\uppi \LambdarMSE_{k+2,i}(\gamma))^{1/2}} }^2 \\
+  \kappa^{-1} \defEns{ \sum_{i=N+1}^{N+n}\frac{ \weight{i} }{(4 \uppi \LambdarMSE_{N+1,i}(\gamma))^{1/2}} }^2 \eqsp,
\end{multline}
for $n_1,n_2 \in \nset$, $\LambdarMSE_{n_1,n_2}(\gamma)$ is given by \eqref{eq:def_xirmse}.
\end{theorem}
\begin{proof}
  The proof is postponed to \Cref{sec:proof_theo_var_tv}.
\end{proof}
To illustrate the result \Cref{theo:var_tv}, we first illustrate numerically the
behaviour $(u_{N,n}^{(4)})_{n \geq 1}$ for $\kappa =1$ $N=0$, and four
different non-increasing sequences of step sizes $(\gamma_k)_{k \geq
  1}$, $\gamma_k = (1+k)^{-\alpha}$ for $\alpha =1/4,1/2,3/4$ and $\gamma_k = 1/2$ for $k \geq 1$. These
results are gathered in \Cref{fig:u_n_4}, where it can be observed
that $(\Gamma_n u_{0,n}^{(4)}(\gamma))_{n \geq 1}$ converges to a limit  as $n \to \plusinfty$.  In \Cref{sec:bounds-u_0-n4gamma}, we show that there exist $C_1,C_2 >0$ independent of $(\gamma_k)_{k \geq 1}$, such that $C_1 \Gamma_n^{-1}  \leq u_{0,n}^{(4)}(\gamma) \leq C_2 \Gamma_n^{-1}$, for non-increasing sequence  $(\gamma_k)_{k \geq 1}$ satisfying $\lim_{k \to \plusinfty} \gamma_k = 0$ and $\lim_{k \to \plusinfty} \Gamma_k = \plusinfty$. 
Therefore, the consequences of \Cref{theo:var_tv} are similar to those of \Cref{theo:var} and are omitted.

\begin{figure}
\begin{tabular}{cc}
  \includegraphics[width=0.5\textwidth]{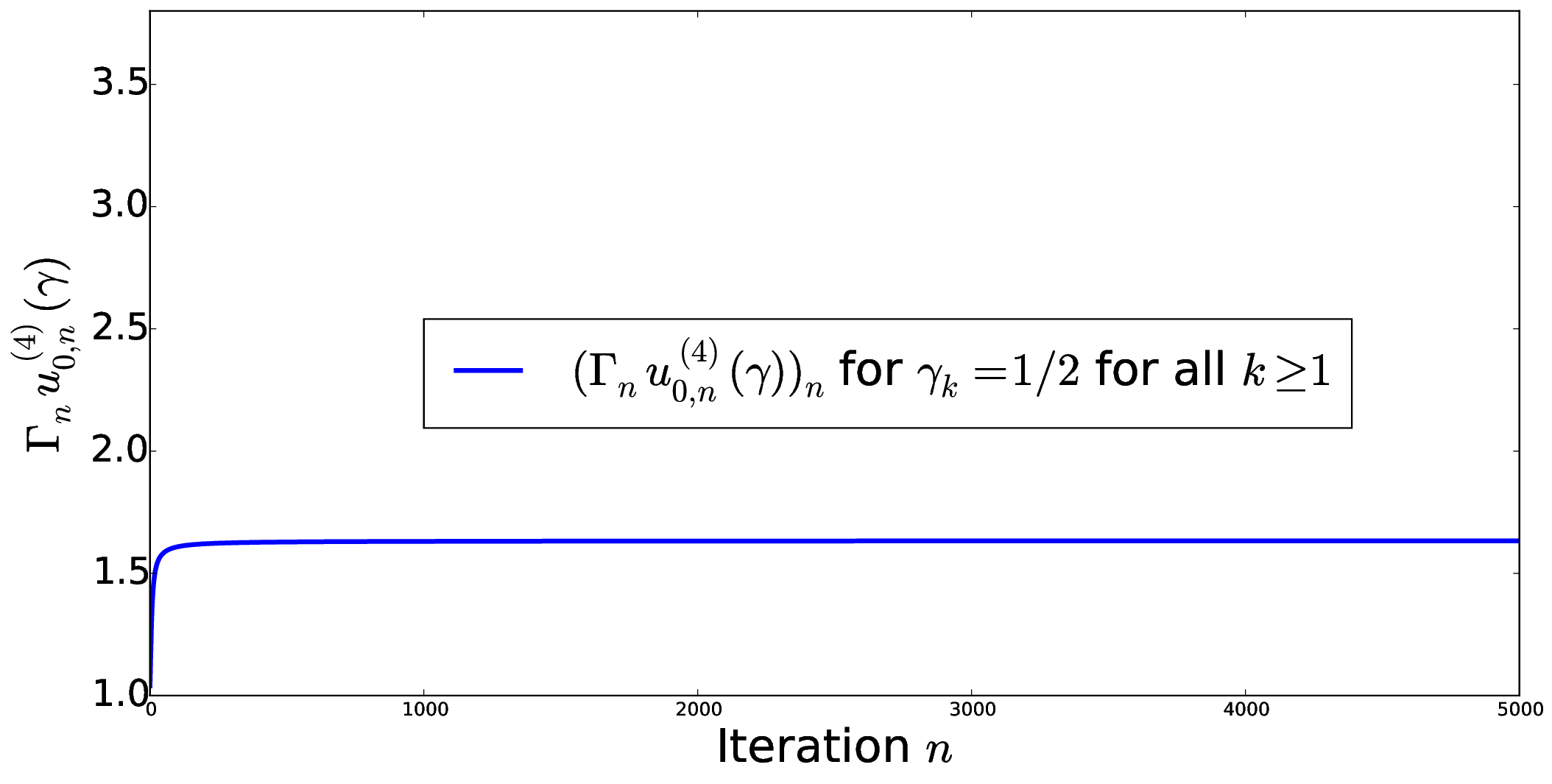} & \includegraphics[width=0.5\textwidth]{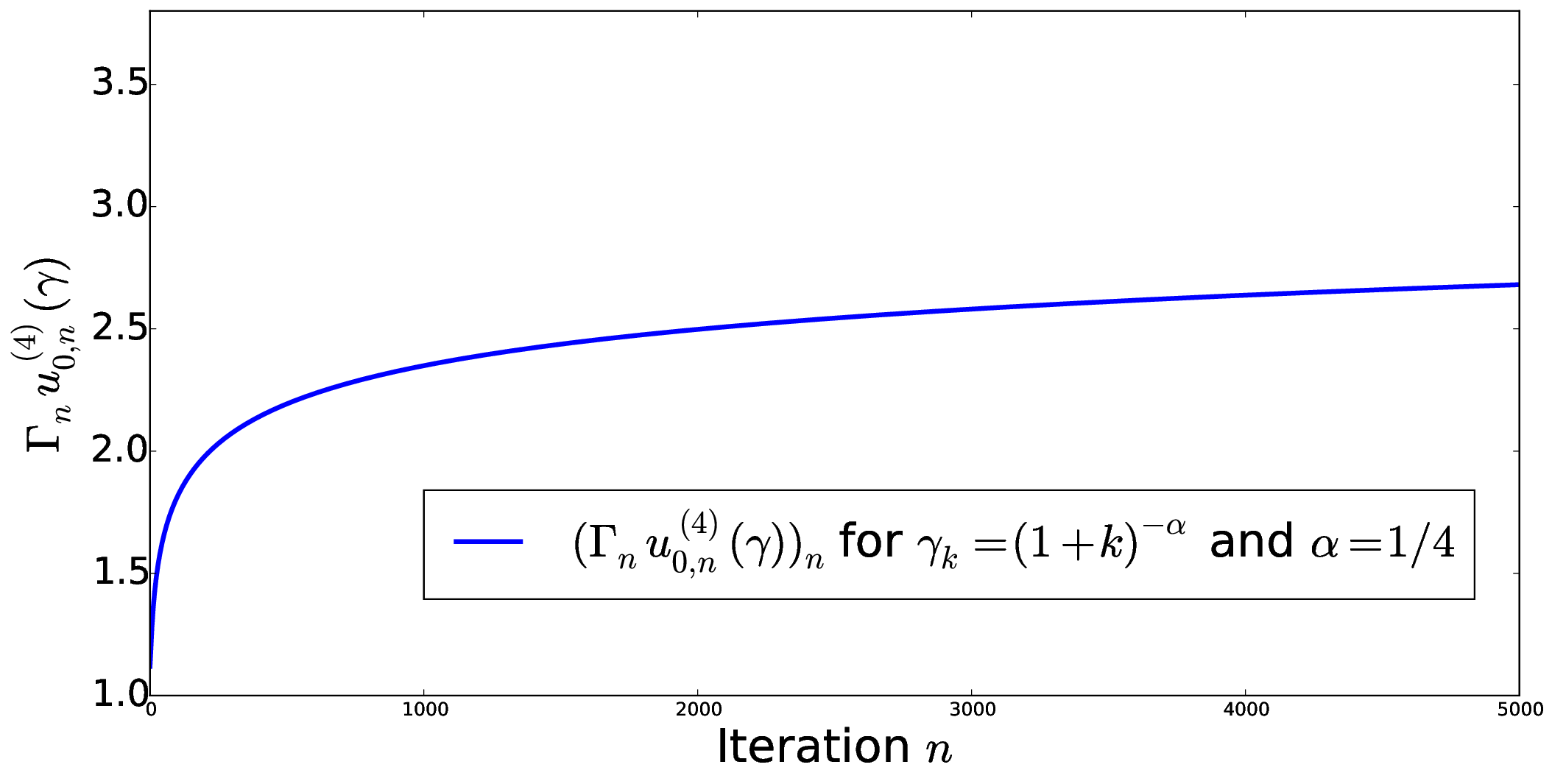} \\
  \includegraphics[width=0.5\textwidth]{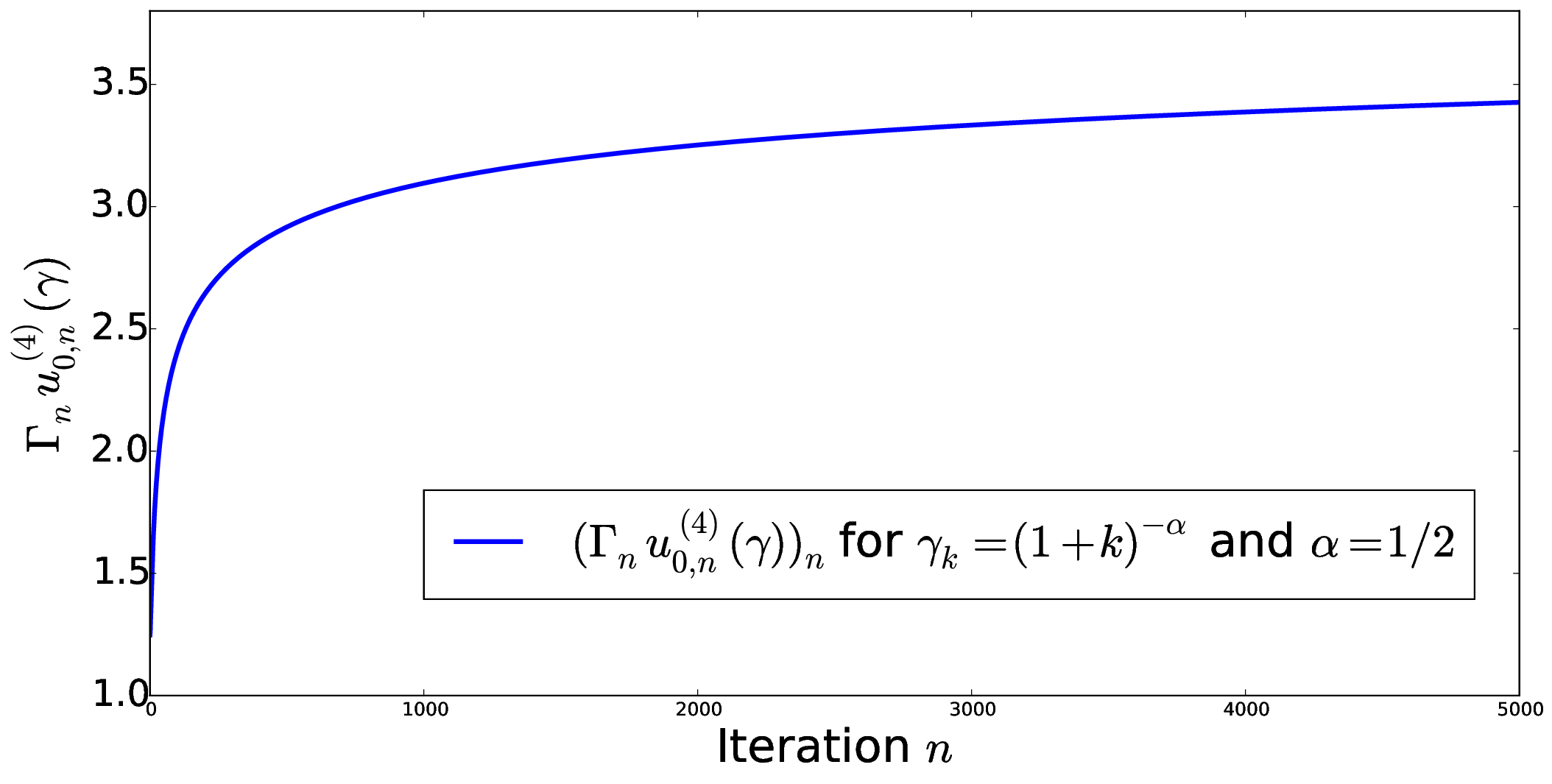} & \includegraphics[width=0.5\textwidth]{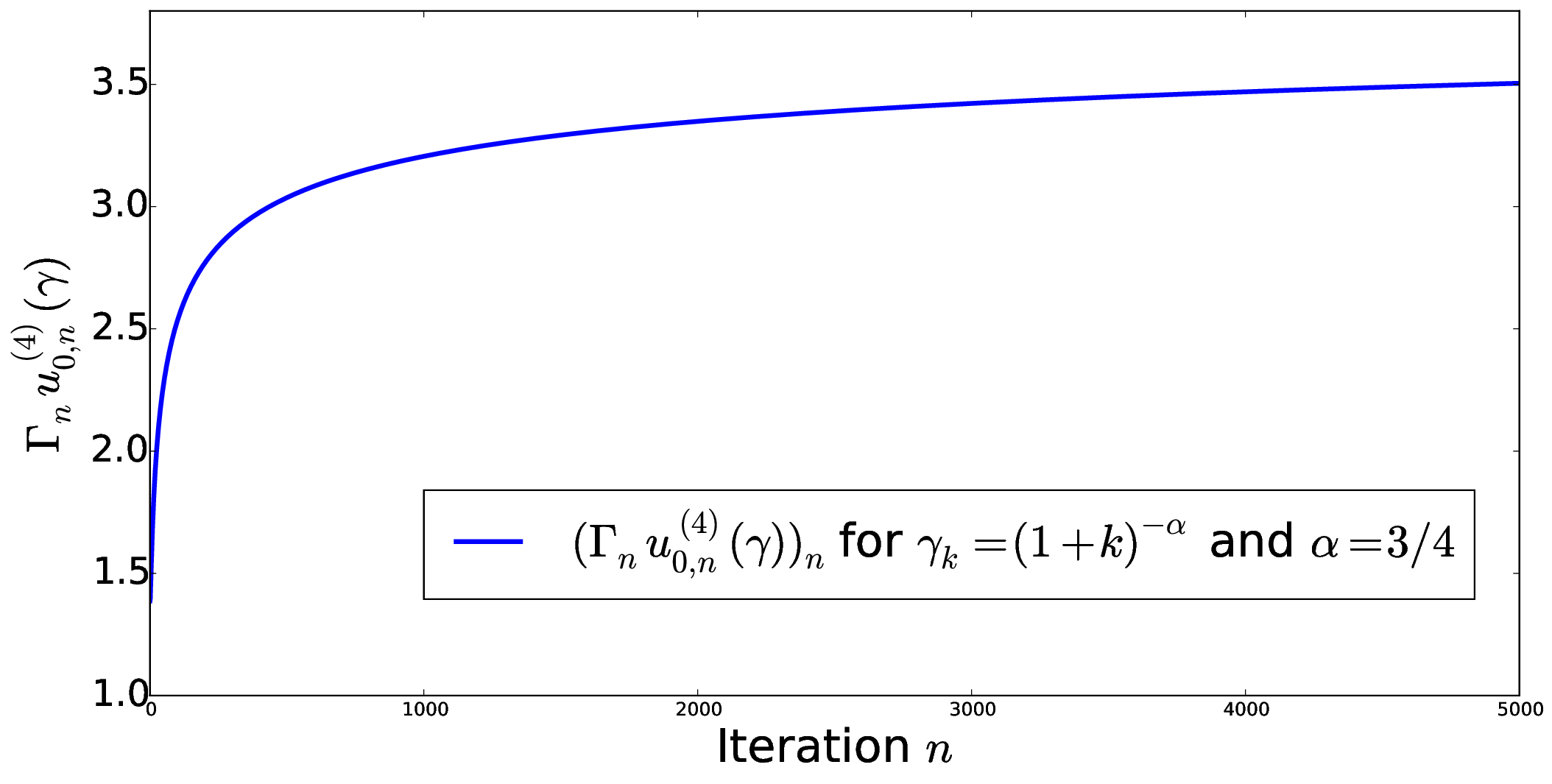}
\end{tabular}
\caption[caption]{Plots of $(u_{0,n}^{(4)})_{n \geq 1}\Gamma_{n}$ for four sequences of step sizes $(\gamma_k)_{k \geq
  1}$, $\gamma_k = (1+k)^{-\alpha}$ for $\alpha =0,1/4,1/2,3/4$}
\label{fig:u_n_4}
\end{figure}

We now establish an exponential deviation inequality for
$\estimateur{f} - \PE_x[ \estimateur{f} ]$ given by
\eqref{eq:def_GammaN_plusn} for a bounded measurable function $f$.

\begin{theorem}
\label{theo:concentration_gauss}
Assume \Cref{assum:regularity_2} and  \Cref{assum:potentialU}. Let $(\gamma_k)_{k \geq 1}$ be a non-increasing sequence with $\gamma_1 \leq 2/(m+L)$. Then for all $N \geq 0$, $n \geq 1$, $r>0$ and Lipschitz functions $f : \rset^d \to \rset$:
\begin{equation*}
 \probaMarkov{x}{\estimateur{f} \geq  \PE_x[ \estimateur{f} ]+r }\leq  \exp\left(  - \frac{r^2 \kappa ^2 \Gamma_{N+2, N+n+1}}{16 \norm{f}_{\Lip}^2 v_{N,n}(\gamma)}\right) \eqsp,
\end{equation*}
where $v_{N,n}(\gamma)$ is defined by \eqref{eq:def_u_n_3}.
\end{theorem}

\begin{proof}
The proof is postponed to \Cref{sec:proof-crefth_conc_gauss}.
\end{proof}

If we apply this result to the sequence $(\gamma_k)_{k \geq 1}$ defined for all $k \geq 1$ by $\gamma_k = \gamma_1 k^{-\alpha}$, for $\alpha \in \ccint{0,1}$, we end up with a concentration of order $\exp(-Cr^2 \gamma_1 n^{1-\alpha})$ for $\alpha \in \coint{0,1}$, for some constant $C \geq 0$ independent of $\gamma_1$ and $n$.

\begin{theorem}
\label{theo:concentration_gauss_tv}
Assume \Cref{assum:regularity_2} and  \Cref{assum:potentialU}. Let $(\gamma_k)_{k \geq 1}$ be a non-increasing sequence with $\gamma_1 \leq 2/(m+L)$. Let $(X_n)_{n \geq 0}$ be given by \eqref{eq:euler-proposal-2} and started at $x \in \rset^d$. Then for all $N \geq 0$, $n \geq 1$, $r>0$,  and functions $f\in \functionspace[b]{\rset^d}$:
\begin{equation*}
 \probaMarkov{x}{\estimateur{f} \geq  \PE_x[ \estimateur{f} ]+r }
\leq  \rme^{  - \defEns{r - \osc{f}(\Gamma_{N+2,N+n+1})^{-1}}^2 /\{2 \osc{f}^2 u_{N,n}^{(5)}(\gaStep) \}}\eqsp,
\end{equation*}
where
\begin{equation*}
u_{N,n}^{(5)}(\gaStep) =
\sum_{k=N}^{N+n-1} \gaStep_{k+1} \defEns{ \sum_{i=k+2}^{N+n} \frac{\weight{i}}{ (\uppi \LambdarMSE_{k+2,i})^{1/2}} }^2 \\
+  \kappa^{-1} \defEns{ \sum_{i=N+1}^{N+n} \frac{\weight{i}}{ (\uppi \LambdarMSE_{N+1,i})^{1/2}} }^2 \eqsp.
\end{equation*}
\end{theorem}
\begin{proof}
  The proof is postponed to \Cref{sec:proof-concentration_tv}.
\end{proof}

Note that $u_{N,n}^{(5)}(\gaStep)$ is up to numerical constants similar to $u_{N,n}^{(4)}(\gaStep)$ given in \eqref{eq:u_lap_tv}. Therefore, using the same calculations as in \Cref{sec:bounds-u_0-n4gamma}, there exist $C_1,C_2 >0$ such that $C_1 \Gamma_n^{-1}  \leq u_{0,n}^{(5)}(\gamma) \leq C_2 \Gamma_n^{-1}$, for $\gamma_k = \gamma_1 / k^{-\alpha}$, $\alpha \in \ccint{0,1}$.  Then, if we apply \Cref{theo:concentration_gauss_tv} to the sequence $(\gamma_k)_{k \geq 1}$ defined for all $k \geq 1$ by $\gamma_k = \gamma_1 k^{-\alpha}$, for $\alpha \in \ccint{0,1}$, we end up with a concentration of order $\exp(-Cr^2 \gamma_1 n^{1-\alpha})$ for $\alpha \in \coint{0,1}$, for some constant $C \geq 0$ independent of $\gamma_1$ and $n$.


%% file: NumericalExp.tex
\section{Numerical experiments}
\label{sec:numerics}
Consider a binary regression set-up in which the binary observations
(responses) $\{Y_i\}_{i=1}^{p}$ are conditionally independent Bernoulli
random variables with parameters $
\{\vrho(\bfbeta^T X_i)\}_{i=1}^p$,
where $\vrho$ is the
logistic function defined for $z \in \rset$ by $\vrho(z) =
\rme^z/(1+\rme^z)$ and $\{X_i\}_{i=1}^p$ and $\bfbeta$ are $d$ dimensional vectors of known
covariates and unknown regression coefficients, respectively.  The prior distribution for the parameter
$\bfbeta$ is a zero-mean Gaussian distribution with covariance matrix
$\Sigma_{\bfbeta}$. %
The density of the posterior distribution of $\bfbeta$ is up to a proportionality constant given by
\[
\pi_{\bfbeta}(\bfbeta | \{(X_i,Y_i)\}_{i=1}^p) \propto \exp \parenthese{\sum_{i=1}^p \defEns{ Y_i \bfbeta^T X_i - \log(1+\rme^{\bfbeta^T X_i})} - 2^{-1} \bfbeta^T \Sigma^{-1}_{\bfbeta} \bfbeta } \eqsp.
\]
Bayesian inference for the logistic regression model has long been recognized as a numerically involved problem. Several algorithms have been proposed, trying to mimick the data-augmentation (DA) approach  of \cite{albert:chib:1993} for probit regression; see \cite{holmes:held:2006}, \cite{fruehwirth:fruehwirth:2010} and \cite{gramacy:polson:2012}. Recently, a very promising  DA algorithm has been proposed in \cite{polson:scott:windle:2013}, using the Polya-Gamma distribution in the DA part. This algorithm has been shown to be uniformly ergodic for the total variation by \cite[Proposition~1]{choi:hobert:2013}, which provides an explicit expression for the ergodicity constant. This constant is exponentially small in the dimension of the parameter space and the number of samples. Moreover, the complexity of the augmentation step is cubic in the dimension, which prevents from using this algorithm when the dimension of the regressor is large.

We apply ULA to sample from the posterior distribution
$\pi_{\bfbeta}(\cdot | \{(X_i,Y_i)\}_{i=1}^p)$. The gradient of
its log-density may be expressed as
\[
\nabla \log\{\pi_{\bfbeta}(\bfbeta | \{X_i,Y_i\}_{i=1}^p) \} = \sum_{i=1}^p \defEns{Y_i X_i - \frac{X_i}{1+ \rme^{-\bfbeta^T X_i}}}
- \Sigma_{\bfbeta}^{-1} \bfbeta\eqsp,
\]
Therefore $-\log \pi_{\bfbeta}(\cdot | \{X_i,Y_i\}_{i=1}^p)$
is strongly convex \Cref{assum:potentialU} with $m
=\lambda_{\max}^{-1}(\Sigma_{\bfbeta})$ and satisfies \Cref{assum:regularity_2}
with $L = (1/4)\sum_{i=1}^p X_i^{\operatorname{T}} X_i+
\lambda^{-1}_{\min}(\Sigma_{\bfbeta}) $, where $\lambda_{\min}(\Sigma_{\bfbeta})$
and $\lambda_{\max}(\Sigma_{\bfbeta})$ denote the minimal and maximal eigenvalues
of $\Sigma_{\bfbeta}$, respectively.  We first
compare the histograms produced by ULA and the P\`olya-Gamma Gibbs
sampling from \cite{polson:scott:windle:2013}. For that purpose, we take
$d=5$, $p=100$, generate synthetic data $ (Y_i)_{1 \leq i \leq p}$ and
$(X_i)_{1 \leq i \leq p}$, and set $\Sigma_{\bfbeta}^{-1} = (dp)^{-1}  (\sum_{i=1}^p
X_i^{\operatorname{T}} X_i)\operatorname{I}_d$. We produce $10^8$
samples from the P\'olya-Gamma sampler using the \textsf{R} package
\textsf{BayesLogit} \cite{windle:polson:scott:BayesLogit}. Next, we
make $10^3$ runs of the Euler approximation scheme with $n=10^6$
effective iterations, with a constant sequence $(\gamma_k)_{k \geq
  1}$, $\gamma_k =  10 (\kappa
n^{1/2})^{-1}$ for all $k \geq 0$ and a burn-in period
$N=n^{1/2}$. The histogram of the P\'olya-Gamma Gibbs sampler for first component, the corresponding  mean of the obtained histograms for ULA and the $0.95$ quantiles are displayed  in \Cref{fig:hist_Euler_constant}.
 The same procedure is
also applied with the decreasing step size sequence $(\gamma_k)_{k \geq
  1}$ defined by $\gamma_k = \gamma_1 k^{-1/2}$, with $\gamma_1 = 10 (\kappa
\log(n)^{1/2})^{-1}$ and for the burn in period
$N = \log(n)$, see also \Cref{fig:hist_Euler_constant}.
\begin{figure}[!h]
\begin{tabular}{cc}
\includegraphics[width=0.45\textwidth,height=3.7cm]{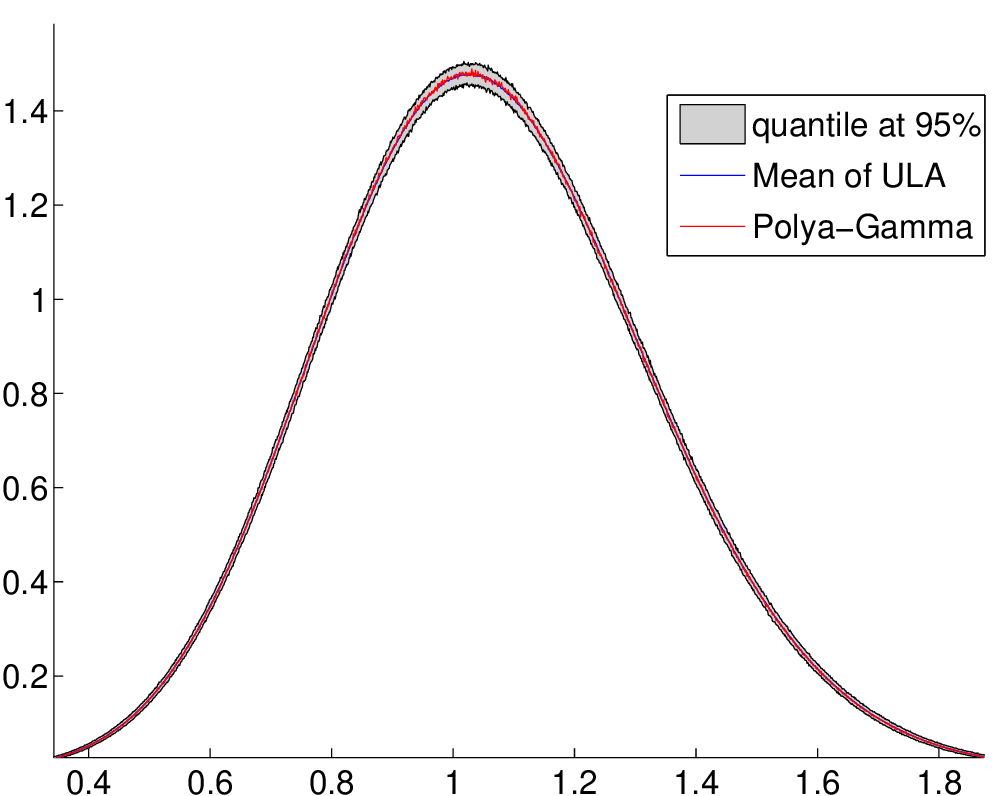} &  \includegraphics[width=0.45\textwidth]{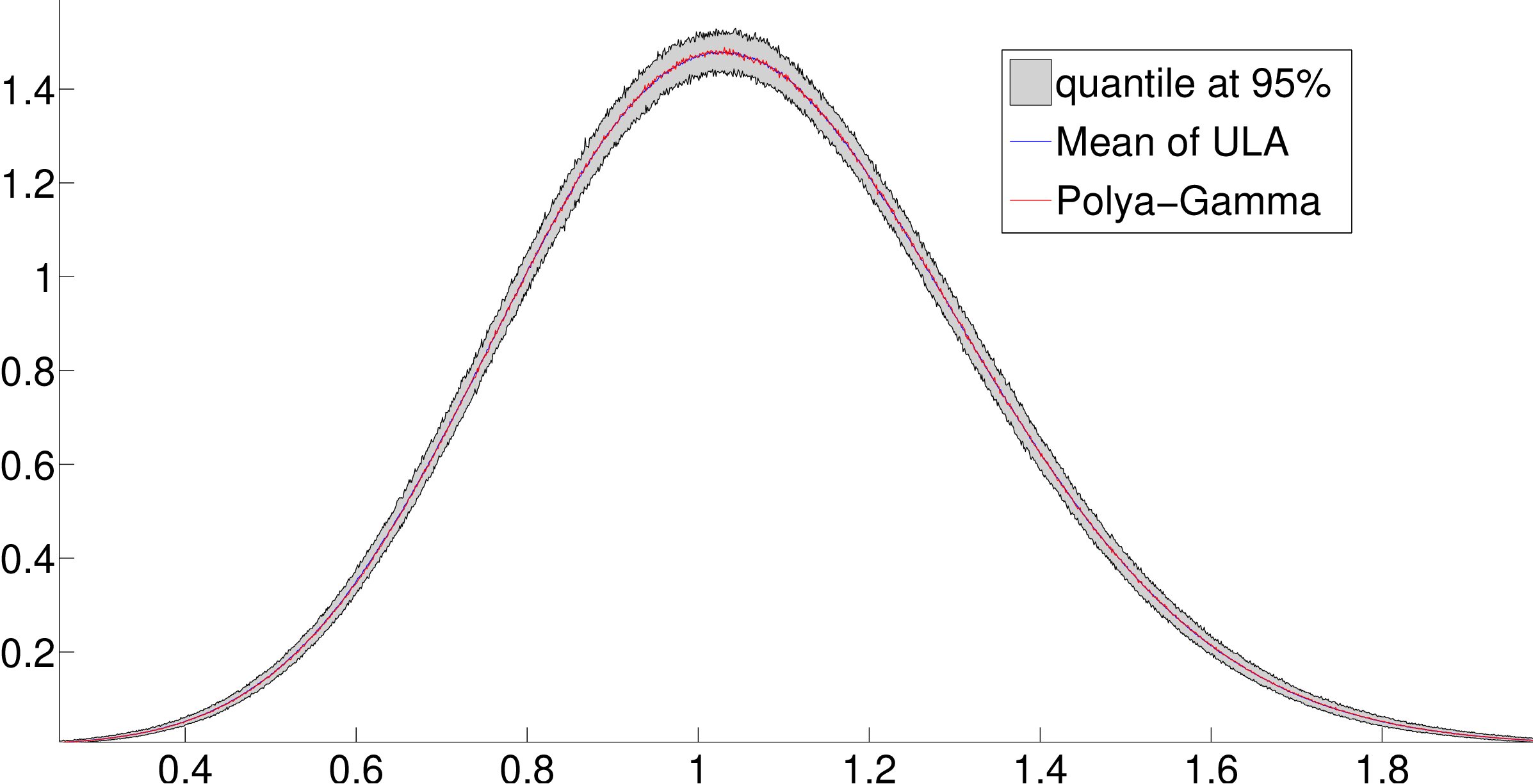}
\end{tabular}
\caption{Empirical distribution comparison between the Polya-Gamma Gibbs Sampler and ULA. Left panel:  constant step size $\gamma_k = \gamma_1$ for all $k \geq 1$; right panel: decreasing step size $\gamma_k = \gamma_1 k^{-1/2}$ for all $k \geq 1$}
\label{fig:hist_Euler_constant}
\end{figure}
In addition, we also compare MALA and ULA
on five real data sets, which are summarized in
\Cref{tab:benchmark}. Note that for the Australian credit data set,
the ordinal covariates have been stratified by dummy
variables. Furthermore, we normalized the data sets and consider the
Zellner prior setting $\Sigma^{-1} = (\uppi^2d /3)\Sigma_X^{-1}$ where
$\Sigma_X= p^{-1} \sum_{i=1}^p X_i X_i^T $ ; see
\cite{sabane:held:2011}, \cite{hanson:branscum:wesley:2014} and the
references therein. Also, we apply a pre-conditioned version of MALA and ULA,
targeting the probability density $\tilde{\pi}_{\bfbeta}(\cdot) \propto \pi_{\bfbeta}(\Sigma_X^{1/2}
\cdot )$.
Then, we obtain samples from $\pi_{\bfbeta}$  by post-multiplying the obtained draws by
$\Sigma_X^{1/2}$.
We compare MALA and ULA for each
data sets by estimating for each component $i \in \{1,\ldots,d\}$ the
marginal accuracy between their $d$ marginal empirical distributions
and the $d$ marginal posterior distributions, where the marginal accuracy
between two probability measure $\mu,\nu$ on $(\rset,\mathcal{B}(\rset))$ is defined by
\begin{equation*}
  \operatorname{MA}(\mu,\nu) = 1-(1/2)\tvnorm{\mu-\nu} \eqsp.
\end{equation*}
This quantity has already been considered in
\cite{faes:ormerod:wand:2011} and \cite{chopin:ridgway:2015} to
compare approximate samplers. To estimate the $d$ marginal
posterior distributions, we run $2 \cdot 10^7$ iterations of the
Polya-Gamma Gibbs sampler.  Then $100$ runs of MALA and ULA ($10^6$
iterations per run) have been performed. For MALA, the step size is
chosen so that the acceptance probability at stationarity is
approximately equal to $0.5$ for all the data sets.  For ULA, we
choose the same constant step size than MALA. We display the boxplots
of the mean of the estimated marginal accuracy across all the
dimensions in \Cref{fig:box_plots_accuracy}. These results all imply
that ULA is an alternative to the Polya-Gibbs sampler and the MALA
algorithm.

\begin{table}
\centering
\begin{tabular}{|c|c|c|}
\hline
\backslashbox{Data set}{Dimensions} & Observations $p$  & Covariates $d$ \\
\hline
German credit \footnotemark[1]& 1000 & 25  \\
\hline
Heart disease \footnotemark[2] & 270 & 14 \\
\hline
Australian credit\footnotemark[3]& 690& 35 \\
\hline
Pima indian diabetes\footnotemark[4]& 768& 9 \\
\hline
Musk\footnotemark[5]& 476 & 167 \\
\hline
\end{tabular}
\caption{\label{tab:benchmark}Dimension of the data sets}
\end{table}
\footnotetext[1]{\url{http://archive.ics.uci.edu/ml/datasets/Statlog+(German+Credit+Data)}}
\footnotetext[2]{\url{http://archive.ics.uci.edu/ml/datasets/Statlog+(Heart)}}
\footnotetext[3]{\url{http://archive.ics.uci.edu/ml/datasets/Statlog+(Australian+Credit+Approval)}}
\footnotetext[4]{\url{http://archive.ics.uci.edu/ml/datasets/Pima+Indians+Diabetes}}
\footnotetext[5]{\url{https://archive.ics.uci.edu/ml/datasets/Musk+(Version+1)}}


\begin{figure}
\begin{tabular}{cc}
\includegraphics[width=0.45\textwidth]{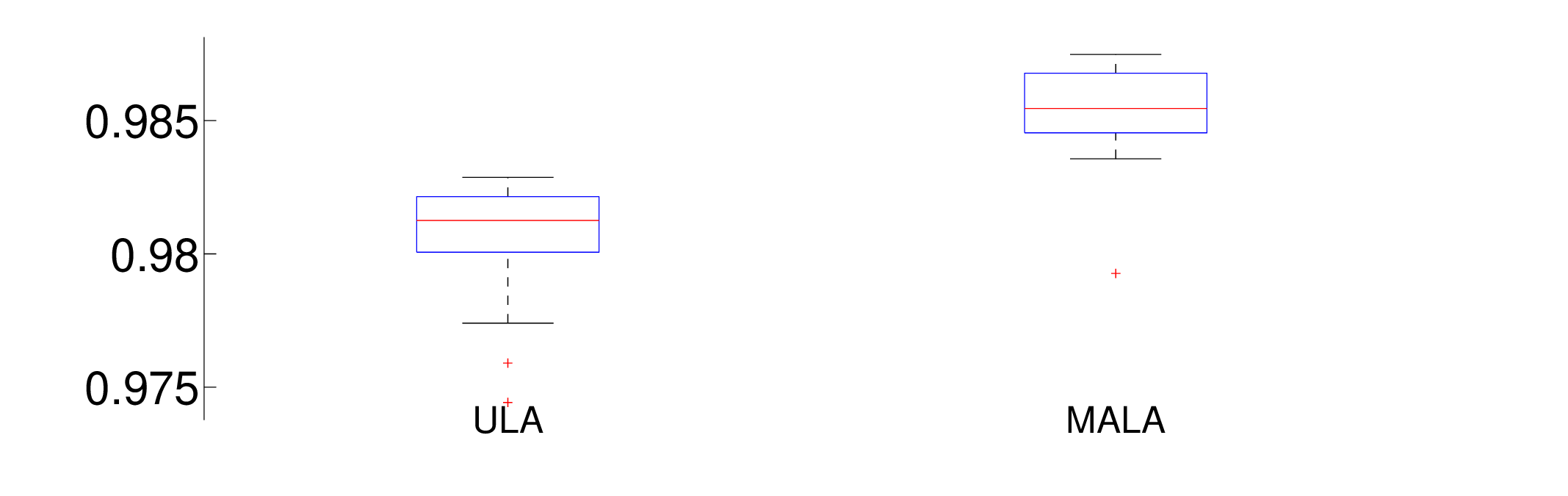} & \includegraphics[width=0.45\textwidth]{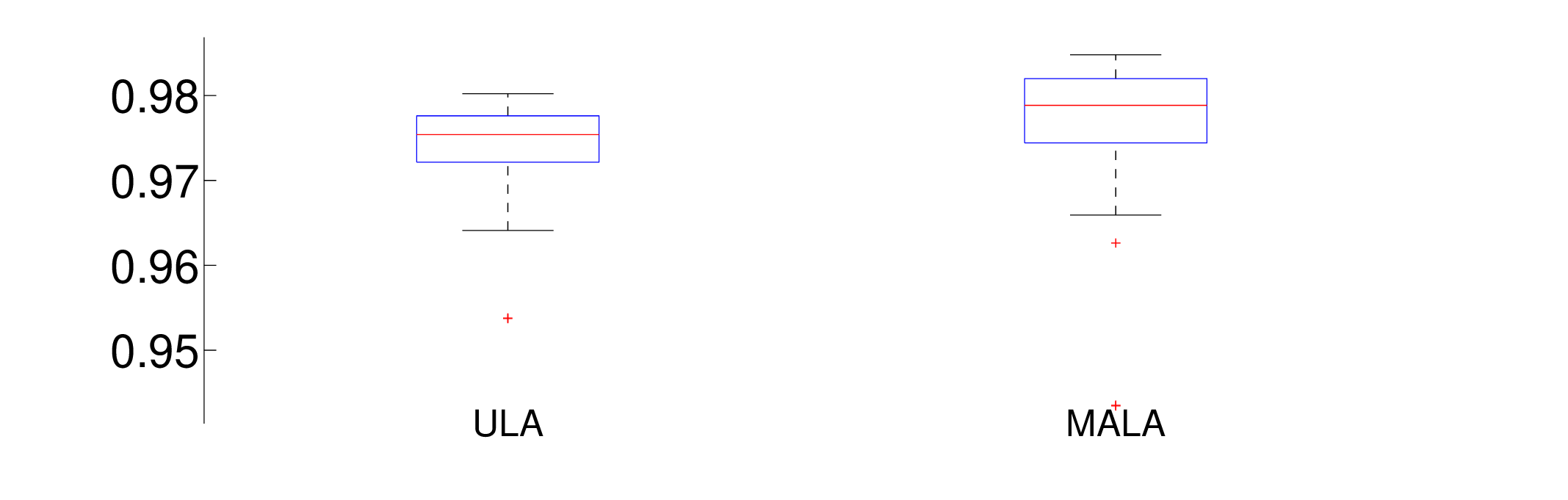} \\
\includegraphics[width=0.45\textwidth]{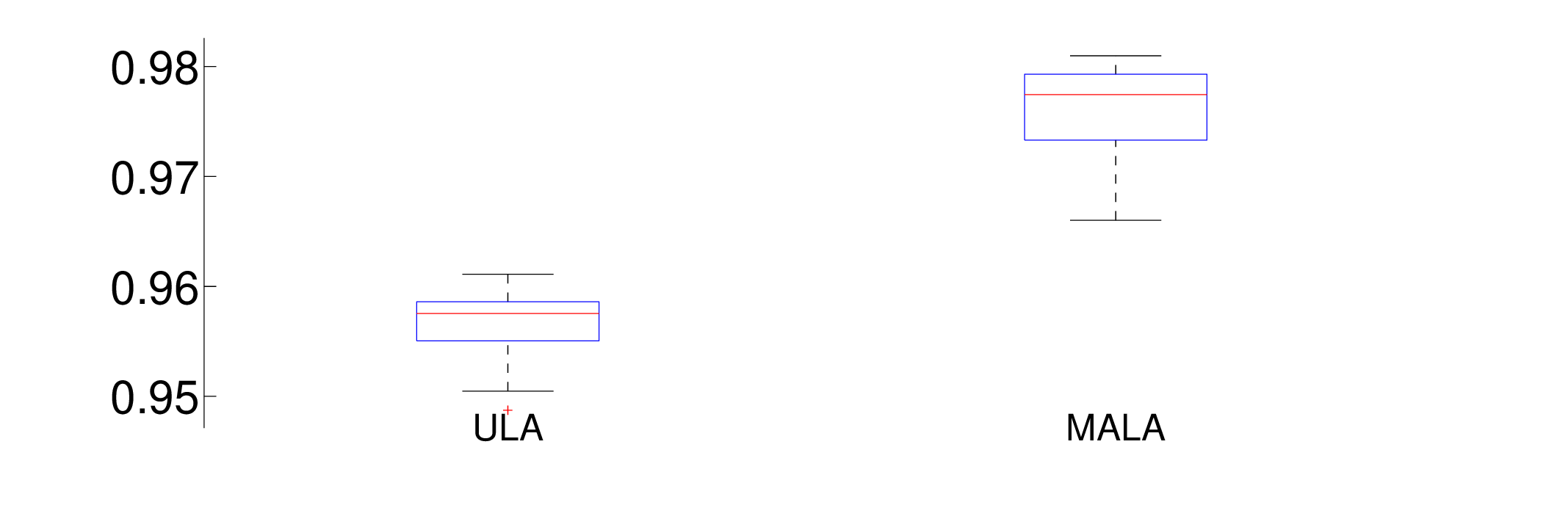}& \includegraphics[width=0.45\textwidth]{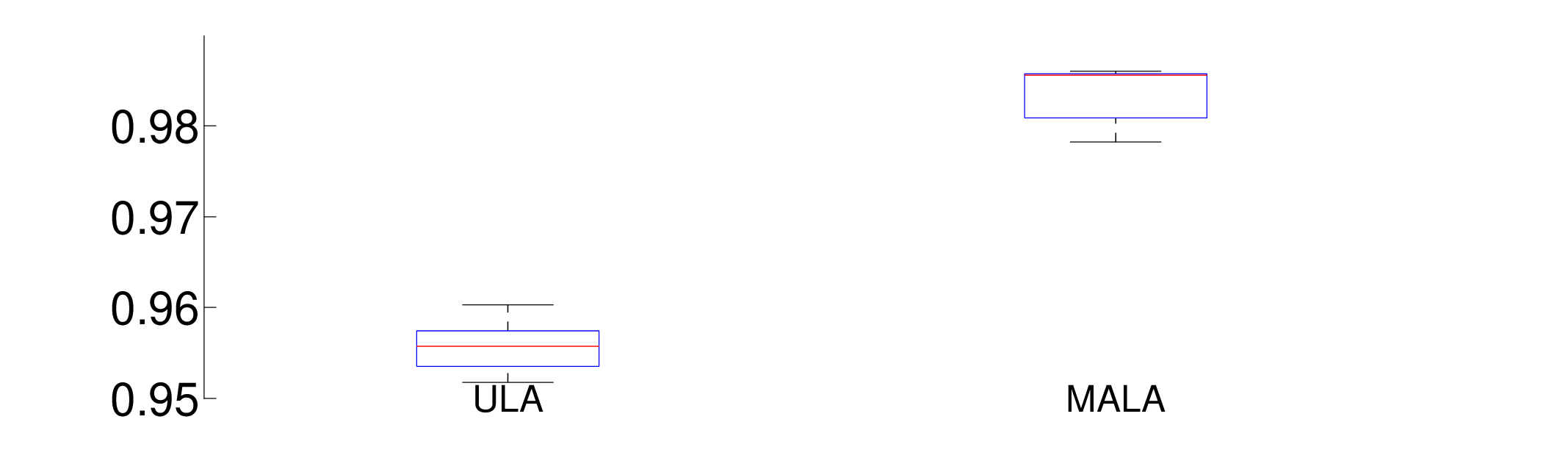}
\end{tabular}
\centering
\includegraphics[width=0.45\textwidth]{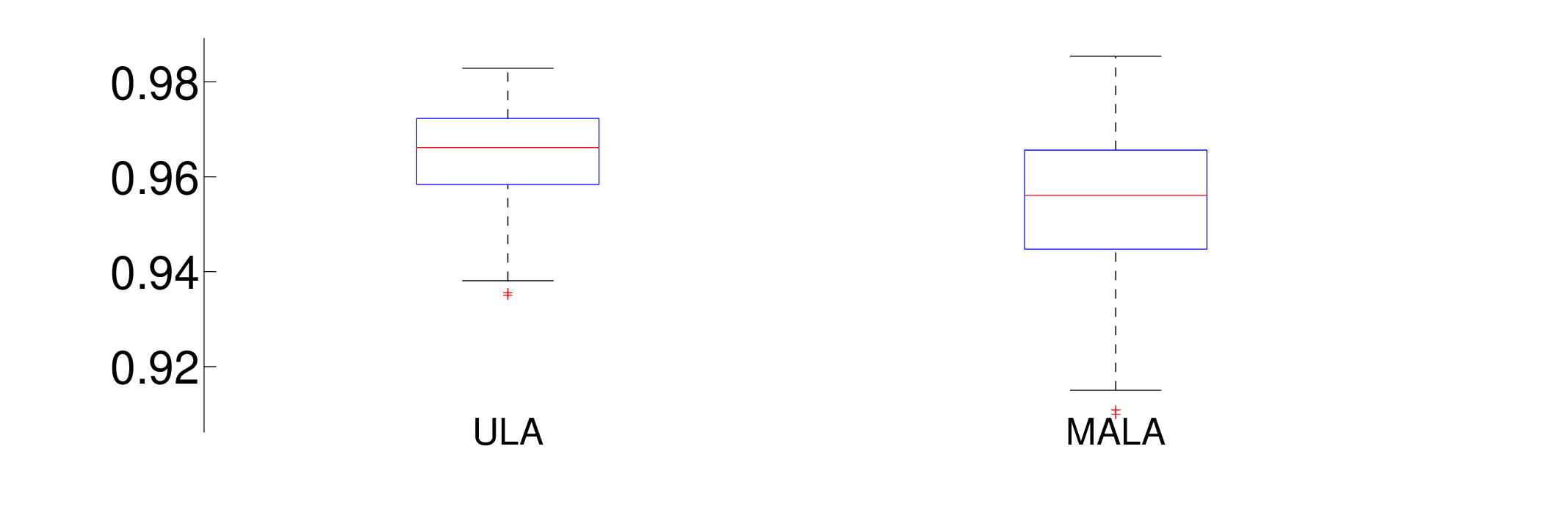}
\caption[caption]{Marginal accuracy across all the dimensions.\\\hspace{\textwidth} Upper left: German credit data set. Upper right: Australian credit data set. Lower left: Heart disease data set. Lower right: Pima Indian diabetes data set. At the bottom: Musk data set}
\label{fig:box_plots_accuracy}
\end{figure}

%% file: convergence_fun_autoreg_D.tex
\section{Contraction  in total variation for  functional autoregressive models}
\sectionmark{Contraction results}
\label{sec:conv-total-vari_AR}
In this section, we consider
functional autoregressive models defined for $k \geq 0$ by 
\begin{equation}
\label{eq:def:functio_AR}
  \Xr_{k+1} = \funreg_{k+1}(\Xr_k) + \sigma_{k+1} \Zr_{k+1} \eqsp,
\end{equation}
where $\sequenceg{\Zr}[k][ 1]$ is a sequence of \iid~$d$ dimensional
standard Gaussian random variables, $\sequenceg{\sigma}[k][1]$ is a
sequence of positive real numbers and $\sequenceg{\funreg}[k][ 1]$  is
a sequence of measurable functions from $\rset^d$ to $\rset^d$ which
satisfies the following assumption:
\begin{assumptionAR}
  \label{assum:strict_contraction_AR}
For all $k \geq 1$,   $\funreg_k$ is $\kappar_k$-Lipschitz.
\end{assumptionAR}
The sequence $\sequence{\Xr}[k][\nset]$ is an inhomogeneous
Markov chain with  Markov kernels
$\sequenceg{\Pr}[k][1]$ on $(\rset^d, \borelSet(\rset^d))$ given for all $x
\in \rset^d$ and $\eventA \in \rset^d$ by
\begin{equation}
\label{eq:def_markov_kernel_AR}
  \Pr_k(x,\eventA) = \frac{1}{( 2 \pi \sigmakD)^{d/2}}\int_{\eventA}\exp\parenthese{-\norm[2]{y-\funreg_k(x)}/(2\sigmakD)} \rmd y\eqsp.
\end{equation}
We denote for all $n \geq 1$ by $\Qr^n$ the marginal distribution of $\Xr_n$ given by
\begin{equation}
  \label{eq:def_marg_AR}
  \Qr^n = \Pr_1 \cdots \Pr_n \eqsp.
\end{equation}
In this section we compute an upper bound of $\tvnorm{\delta_x \Qr^n - \delta_y \Qr^n}$ which does not depend on the dimension $d$. Define for $x,y \in \rset^d$ 
\begin{equation}
\label{eq:def_Er}
 \Er_k(x,y) =
\funreg_k(y)-\funreg_k(x) \eqsp,  \er_k(x,y) =
\begin{cases}
\Er_k(x,y)/\norm{\Er_k(x,y)} & \text{ if } \Er_k(x,y) \not = 0 \\
0 & \text{otherwise} \eqsp,
\end{cases}
\end{equation}
For all  $x,y,z \in \rset^d$,  $x \not = y$, define
\begin{align}
\label{eq:def_F_k}
    \transfrr_k(x,y,z) &= \funreg_k(y) +  \parenthese{\Id-2\er_k(x,y)\er_k(x,y)^{\transp}}z\\
\label{def:alphar}
\alphar_k(x,y,z)&= \frac{\phibfs[k](  \norm{\Er_k(x,y)}-\ps{\er_k(x,y)}{z})}
{\phibfs[k]( \ps{\er_k(x,y)}{z})} \eqsp,
\end{align}
where  $\phibfs[k]$ is the probability density of a zero-mean gaussian variable with variance $\sigma_k^2$.
Let $\Zr_1$ be  a standard $d$-dimensional Gaussian random variable.
Set $\Xr_1= \funreg_k(x) + \sigma_k \Zr_1$ and
\[
\Yr_1 =
\begin{cases}
\funreg_k(y) +\sigma_k \Zr_1 & \text{ if }\Er_k(x,y) =0 \\
B_1  \, \Xr_1  + (1-B_1)  \, \transfrr_k(x,y,\Zr_1) & \text{ if } \Er_k(x,y) \not = 0 \eqsp,
\end{cases}
\]
where $B_1$ is a
Bernoulli random variable independent of $\Zr_1$ with
success probability
\[
p_k(x,y,z)= 1 \wedge \ \alphar_k(x,y,z) \eqsp.
\]
The construction above defines for all $(x,y) \in \rset^d \times \rset^d$ the Markov kernel $\Kr_k$ on $(\rset^d \times \rset^d , \borelSet(\rset^d) \otimes \borelSet(\rset^d))$ given for all $(x,y) \in \rset^d \times \rset^d$ and $\eventA \in \borelSet(\rset^d) \otimes \borelSet(\rset^d)$ by
\begin{align}
  \label{eq:def_coupling_kernel_AR}
&
\Kr_k((x,y) , \eventA) = \frac{\1_{\Deltar}(\funreg_k(x),\funreg_k(y))} {(2 \uppi \sigmakD)^{d/2}}\int_{\rset^d} \1_{\eventA}(\tildex,\tildex) \rme^{-\norm[2]{\transar_k(\tildex,x)}/(2\sigmakD)} \rmd \tildex \\
\nonumber
&+ \frac{\1_{\Deltar^{\complem}}(\funreg_k(x),\funreg_k(y))}{ (2 \uppi \sigmakD)^{d/2}} \left[ \int_{\rset^d} \1_{\eventA}(\tildex,\tildex)
p_k(x,y,\transar_k(\tildex,x)) \rme^{-\norm[2]{\transar_k(\tildex,x)}/(2\sigmakD)} \rmd \tildex \right.
\\
\nonumber
&\left. +  \int_{\rset^d} \1_{\eventA}(\tildex,\transfrr_k(x,y,\transar_k(\tildex,x)))
\defEns{1-p_k(x,y,\transar_k(\tildex,x))} \rme^{-\norm[2]{\transar_k(\tildex,x)}/(2\sigmakD)} \rmd \tildex \right] \eqsp,
\end{align}
where for all $\tildex \in \rset^d$, $\transar_k(\tildex,x) = \tildex-\funreg_k(x)$ and $ \Deltar = \defEnsE{(\tildex,\tildey) \in \rset^d \times \rset^d }{\tildex = \tildey} $.
It is shown in \cite[Section 3.3]{bubley:dyer:jerrum:1998} that for all $x,y \in \rset^d$ and $k \geq 1$, $\Kr_k((x,y),\cdot)$ is a transference plan of $\Pr_k(x,\cdot)$ and $\Pr_k(y,\cdot)$. For completeness, the proof is given in \Cref{sec:coupling}. Furthermore, we have for all $x,y \in \rset^d$ and $k \geq 1$
  \begin{equation}
    \label{eq:proof_AR_1}
    \Kr_k((x,y), \Deltar) = 2\Phibf\parenthese{-\frac{\norm{\Er_k(x,y)}}{2 \sigma_k}} \eqsp.
  \end{equation}

For all initial distribution $\mu_0$ on $(\rset^d\times \rset^d,
\borelSet(\rset^d) \otimes \borelSet(\rset^d))$, $\PPtilde_{\mu_0}$
and $\PEtilde_{\mu_0}$ denote the probability and the expectation
respectively, associated with the sequence of Markov kernels
$\sequenceg{\Kr}[k][1]$ defined in \eqref{eq:def_coupling_kernel_AR}
and $\mu_0$ on the canonical space $((\rset^d\times
\rset^d)^{\nset},(\borelSet(\rset^d) \otimes
\borelSet(\rset^d))^{\otimes \nset})$,
$\sequenceD{(\Xr_i,\Yr_i)}[i][\nset]$ denotes the canonical process
and $\sequence{\filtrationTilde}[i][\nset]$ the corresponding
filtration.  Then  if $(\Xr_0,\Yr_0) =
(x,y) \in \rset^d \times \rset^d$, for all $k \geq 1$ $(\Xr_k,\Yr_k)$
is a coupling of $\delta_x \Qr^k$ and $\delta_y \Qr^k$. Using Lindvall's inequality, bounding
$ \tvnorm{\delta_x \Qr^{n}- \delta_y \Qr^{n}}$ amounts to evaluate $\PPtilde_{(x,y)}(\Xr_n \not =
\Yr_n)$. 
\begin{theorem}
  \label{theo:strict_convergence_AR}
  Assume \Cref{assum:strict_contraction_AR}. Then for all $x,y \in \rset^d$ and $n \geq 1$,
  \begin{equation*}
    \tvnorm{\delta_x \Qr^{n}- \delta_y \Qr^{n}} \leq \1_{\Deltar^{\complem}}((x,y)) \defEns{1-2\Phibf\parenthese{-\frac{\norm{x-y} }{2\Xi_n^{1/2}}}} \eqsp,
  \end{equation*}
where $\sequenceg{\Xi}[i][1]$ is defined for all $k \geq 1$ by $\Xi_{k} = \sum_{i=1}^{k} \{ \sigma_{i}^2  /  \prod_{j=1}^{i} \kappar_{j}^{2}\} $.
\end{theorem}
We preface the proof by a technical Lemma.

\begin{lemma}
  \label{lem:proof_AR_2}
  For all $\varsigma,\ar >0$ and $t \in \rset_+$, the following identity holds
  \begin{multline*}
    \int_{\rset}\phibfvs(y) \defEns{1- 1 \wedge \frac{\phibfvs(t-y)}{\phibfvs(y)}}\defEns{1-2 \Phibf\parenthese{-\frac{\abs{2y-t}}{2 \ar}}
}\rmd y \\
= 1-2\Phibf\parenthese{-\frac{t}{2 (\varsigma^2+\ar^2)^{1/2}}} \eqsp.
  \end{multline*}
\end{lemma}
\begin{proof}
  Let $\varsigma,\ar >0$ and $t \in \rset_+$.
Let us denote by $I$ the integral on the left hand side in the expression above.
Then,
\begin{align}
\nonumber
&I
\nonumber
 = \int_{-\infty}^{t/2} \defEns{\phibfvs(y)-\phibfvs(t-y)}\defEns{1-2 \Phibf\parenthese{\frac{2y-t}{2 \ar}}
}\rmd y\\
\label{eq:proof_AR_gauss}
& = \int_{-\infty}^{t/2} \phibfvs(y)\defEns{1-2 \Phibf\parenthese{\frac{2y-t}{2 \ar}}
}\rmd y \\
\nonumber
&\phantom{\defEns{1-2 \Phibf\parenthese{\frac{2y-t}{2 \ar}}}}
  - \int_{-\infty}^{-t/2} \phibfvs(y)\defEns{1-2 \Phibf\parenthese{\frac{t+2y}{2 \ar}}
}\rmd y \eqsp,
\end{align}
 Now to
simplify the proof, we give a probabilistic interpretation of this two
integrals. Let $\Xrd$ and $\Yrd$ be two real Gaussian random variables with
zero mean and variance $\ar^2$ and $\varsigma^2$ respectively.
Since for all $u \in \rset_+$, $1-2 \Phibf(-u/(2 \ar)) = \PP[\abs{\Xrd} \leq u/2]$, we have  by \eqref{eq:proof_AR_gauss}
\begin{multline*}
  I = \proba{\Yrd \leq t/2 , \Xrd + \Yrd \leq t/2, \Yrd - \Xrd \leq t/2}\\
-
\proba{\Yrd \geq t/2 , \Xrd + \Yrd \geq t/2, \Yrd - \Xrd \geq t/2} \eqsp.
\end{multline*}
Using that  $\Yrd$ and $-\Yrd$ have the same law in the second term, we get $I= I_1+I_2$ where
\begin{align}
\nonumber
I_1
&=\proba{\Yrd \leq t/2 , \Xrd + \Yrd \leq t/2, \Yrd - \Xrd \leq t/2, \Xrd \geq 0}\\
\nonumber
& \phantom{aaaaa}-
\proba{\Yrd \leq -t/2 , \Xrd - \Yrd \geq t/2, \Yrd + \Xrd \leq -t/2 ,\Xrd \geq 0} \\
\label{eq:proof_AR_gauss_2}
& = \proba{\abs{\Xrd + \Yrd} \leq t/2, \Xrd \geq 0} \eqsp,
\end{align}
and
\begin{multline*}
  I_2= \proba{\Yrd \leq t/2 , \Xrd + \Yrd \leq t/2, \Yrd - \Xrd \leq t/2, \Xrd \leq 0}\\
-
\proba{\Yrd \leq -t/2 , \Xrd - \Yrd \geq t/2, \Yrd + \Xrd \leq -t/2, \Xrd \leq 0}   \eqsp.
\end{multline*}
Using again that  $\Yrd$ and $-\Yrd$ have the same law  in the two terms we have
\begin{align}
\nonumber
I_2
& =\proba{\Yrd \geq -t/2 , \Xrd - \Yrd \leq t/2, \Yrd + \Xrd \geq -t/2, \Xrd \leq 0}
\\
\nonumber
&\phantom{aaaaa}-\proba{\Yrd \geq t/2 , \Xrd + \Yrd \geq t/2,  \Xrd -\Yrd \leq -t/2, \Xrd \leq 0} \\
\label{eq:proof_AR_gauss_3}
 &=\proba{\abs{\Xrd + \Yrd} \leq t/2, \Xrd \leq 0}  \eqsp.
\end{align}
Combining  \eqref{eq:proof_AR_gauss_2}, \eqref{eq:proof_AR_gauss_3}, we get $I = \PP(\vert \Xrd + \Yrd\vert \leq t /2)$.
The proof follows from the fact that   $\Xrd + \Yrd$ is a real
Gaussian random variable with mean zero and variance
$\ar^2+\varsigma^2$, since $\Xrd $ and $\Yrd$ are independent.

\end{proof}

\begin{proof}[Proof of \Cref{theo:strict_convergence_AR}]

 Since for all $k \geq 1$,
  $(\Xr_k,\Yr_k)$ is a coupling of $\delta_x \Qr^k$ and $\delta_y
  \Qr^k$,  $\tvnorm{\delta_x \Qr^k-\delta_y \Qr^k} \leq \PPtilde_{(x,y)}(\Xr_k \not = \Yr_k )$.

Define for all $k_1,k_2 \in \nset^*$,  $k_1 \leq k_2$, $\Xi_{k_1,k_2} = \sum_{i=k_1}^{k_2}\{ \sigma_{i}^2  / \prod_{j=k_1}^{i} \kappar_{j}^{2}\}$.
Let $n \geq 1$. We show by backward induction that for all $k \in \defEns{0,\cdots,n-1}$,
\begin{equation}
\label{eq:proof_convergence_AR_induction_strict}
  \PPtilde_{(x,y)}(\Xr_n \not = \Yr_n )  \leq \expeMarkovTilde{(x,y)}{
\1_{\Deltar^{\complem}}(\Xr_{k},\Yr_{k}) \parentheseDeux{1-2\Phibf\defEns{-\frac{\norm{\Xr_{k}-\Yr_{k}} }{2 \parenthese{\Xi_{k+1,n}}^{1/2}}}}
} \eqsp,
\end{equation}
Note that the inequality for $k=0$ will conclude the proof.

Since $\Xr_n \not = \Yr_n $ implies that $\Xr_{n-1} \not = \Yr_{n-1} $, the Markov property and \eqref{eq:proof_AR_1} imply
\begin{multline*}
\PPtilde_{(x,y)}(\Xr_n \not = \Yr_n ) = \expeMarkovTilde{(x,y)}{
\1_{\Deltar^{\complem}}(\Xr_{n-1},\Yr_{n-1}) \expeMarkovTilde{(\Xr_{n-1},\Yr_{n-1})}{\1_{\Deltar^{\complem}}(\Xr_{1},\Yr_{1})}}
\\
 \leq \expeMarkovTilde{(x,y)}{
\1_{\Deltar^{\complem}}(\Xr_{n-1},\Yr_{n-1}) \parentheseDeux{1-2\Phibf\defEns{-\frac{\norm{\Er_{n-1}(\Xr_{n-1},\Yr_{n-1})}}{2 \sigma_{n}}}}}
\end{multline*}
Using \Cref{assum:strict_contraction_AR} and
\eqref{eq:def_Er}, $\norm{\Er_{n}(\Xr_{n-1},\Yr_{n-1})} \leq \kappar_{n}
\norm{\Xr_{n-1} - \Yr_{n-1}}$, showing \eqref{eq:proof_convergence_AR_induction_strict} holds for
$k=n-1$.

Assume that \eqref{eq:proof_convergence_AR_induction_strict} holds for $k \in \{1,\ldots,n-1\}$.
On $\defEns{\Xr_{k} \not =
  \Yr_{k}}$, we have
\begin{equation*}
  \norm{\Xr_{k} - \Yr_{k}} = \abs{-\norm{\Er_{k}(\Xr_{k-1},\Yr_{k-1})} + 2 \sigma_{k}
  \er_{k}( \Xr_{k-1} , \Yr_{k-1})^{\transp} \Zr_{k}} \eqsp,
\end{equation*}
which implies
\begin{align*}
 & \1_{\Deltar^{\operatorname{c}}}(\Xr_{k},\Yr_{k})
\parentheseDeux{1-2\Phibf\defEns{-\frac{\norm{\Xr_{k}-\Yr_{k}}}{2 \Xi_{k+1,n}^{1/2}}}}
\\
\nonumber
&= \1_{\Deltar^{\operatorname{c}}}(\Xr_{k},\Yr_{k})
\parentheseDeux{1-2\Phibf\defEns{-\frac{\abs{
 2  \sigma_{k}
  \er_{k}( \Xr_{k-1} , \Yr_{k-1})^{\transp} \Zr_{k}-\norm{\Er_{k}(\Xr_{k-1},\Yr_{k-1})}
}}{2 \Xi_{k+1,n}^{1/2}}}} \eqsp.
\end{align*}
Since $\Zr_{k}$ is independent of $\filtrationTilde_{k-1}$, $ \sigma_{k} \er_{k}( \Xr_{k-1} , \Yr_{k-1})^{\transp} \Zr_{k}$ is a real  Gaussian random variable with zero mean and  variance $\sigma_k^2$, therefore by \Cref{lem:proof_AR_2}, we get
\begin{multline*}
\expeMarkovTildeD{(x,y)}{
\1_{\Deltar^{\complem}}(\Xr_{k},\Yr_{k})
\parentheseDeux{1-2\Phibf\defEns{-\frac{\norm{\Xr_{k}-\Yr_{k}} }{2 \Xi_{k+1,n}^{1/2}}}}
}{\filtrationTilde_{k-1}}
\\
\leq \1_{\Deltar^{\complem}}(\Xr_{k-1},\Yr_{k-1}) \parentheseDeux{1- 2 \Phibf\defEns{-\frac{\norm{\Er_k(\Xr_{k-1},\Yr_{k-1})}}{2 \parenthese{ \sigma_{k}^2 +  \Xi_{k+1,n}}^{1/2}}}} \eqsp.
\end{multline*}
Using by \Cref{assum:strict_contraction_AR} that $\norm{\Er_k(\Xr_{k-1},\Yr_{k-1})} \leq \kappar_k\norm{\Xr_{k-1}-\Yr_{k-1}}$  concludes the induction.
\end{proof}



%% file: proof_arxiv.tex
\input{proof_sup}

\input{proof_supp_autoreg}

%% file: proof_sup.tex
\section{Proofs of \Cref{sec:non-asympt-bounds}}
\label{sec:postponed-proofs}

In this section are gathered the postponed proofs of  \Cref{sec:non-asympt-bounds}.
If  \Cref{assum:regularity_2} holds, then  \cite[Theorem 2.1.12, Theorem 2.1.9]{nesterov:2004} show that for all $x,y \in \rset^d$:
\begin{equation}
\label{eq:convex_forte_contra1}
\ps{\nabla U(y) - \nabla U(x)}{y-x} \geq \frac{\kappa}{2}\norm[2]{y-x} + \frac{1}{m+L} \norm[2]{\nabla U(y) - \nabla U(x)} \eqsp,  \\
\end{equation}
where
\begin{equation*}
\label{eq:definition-kappa}
\kappa = \frac{2 m L }{m+L} \eqsp.
\end{equation*}

\subsection{Proof of  \Cref{theo:convergence-WZ-strongly-convex}}
\label{proof:theo:convergence-WZ-strongly-convex}

\begin{enumerate}[label=(\roman*), wide=0pt, labelindent=\parindent]
\item
The generator $\generator$ associated with $(P_t)_{t \geq 0}$ is given,
for all $f \in C^2(\rset^d)$ and $y \in \rset^d$, by:
\begin{equation}
\label{eq:def_generator}
  \generator f(y) = -\ps{\nabla U(y)}{\nabla f(y)} + \Delta f(y) \eqsp.
\end{equation}
Denote for all $y \in \rset^d$ by $\VStar(y) = \norm{y - x^\star}^2$.
Let $x \in \rset^d$ and $(Y_t)_{t \geq 0}$ be a solution of
\eqref{eq:langevin_2} started at $x$.  Under \Cref{assum:regularity_2}
$\sup_{t \in \ccint{0,T}} \PE[\norm[2]{Y_t^{}}] < \plusinfty$ for
all $T \geq 0$. Therefore, the process $$\parenthese{\VStar(Y_t)
  -\VStar(x) - \int_0 ^t \generator \VStar(Y_s) \rmd s}_{ t \geq 0}$$
is a $(\mathcal{F}_t)_{ t \geq 0}$-martingale.  Denote for all $t \geq 0$ and $x \in \rset^d$ by $v(t,x) = P_t \VStar(x)$. Then we have,
$\partial v(t,x)/\partial t = P_t \generator \VStar(x)$.

Since $\nabla U(x ^\star) = 0$ and by
\Cref{assum:potentialU}, $\ps{\nabla U(x) - \nabla
  U(\xstar)}{x-\xstar} \geq m \norm[2]{x-\xstar}$, we have
\begin{equation}
\label{eq:generator}
\generator \VStar(x) = 2 \parenthese{ - \ps{\nabla U(x)- \nabla U(x^\star)}{x-x^{\star}} +d } \leq 2  \parenthese{ -m \VStar(x) +d } \eqsp.
\end{equation}
Therefore, we get
\[
\frac{\partial v(t,x)}{\partial t} = P_t \generator \VStar(x)  \leq  -2m P_t \VStar(x) + 2d = -2m v(t,x) + 2d \eqsp,
\]
and the proof follows from the Gr\"onwall inequality.
\item Set $\VStar(x) = \norm[2]{x-x^\star}$. By  \Cref{theo:convergence-WZ-strongly-convex}-\ref{item:moment_diffusion_2},  using that $\pi P_t = \pi$ for all $t > 0$ and that the function $z \mapsto z \wedge c$ is concave  for all $c > 0$, we get  using the Jensen inequality
\begin{align*}
\pi(\VStar \wedge c)
= \pi P_t (\VStar \wedge c)& \leq \pi (P_t \VStar \wedge c) \\
&\leq \int \pi(\rmd x)\, c \wedge \left\{ \VStar(x) \rme^{-2 m t} + \frac{d}{m} (1-\rme^{-2mt})  \right\} \end{align*}
Using Lebesgue's dominated convergence theorem and taking the limit as $t \to \plusinfty$, we get $\pi(\VStar \wedge c) \leq d/m$.
Using the monotone convergence theorem and taking the limit as $c \to \plusinfty$ concludes the proof.

\item
Let $x,y \in \rset^d$.
Consider the following SDE in $\rset^d \times \rset^d$:
\begin{equation}
\label{eq:double_langevin}
\begin{cases}
\rmd Y_t &=  - \nabla U (Y_t) \rmd t + \sqrt{2} \rmd B_t \eqsp,\\
\rmd \tilde{Y}_t &=  - \nabla U (\tilde{Y}_t) \rmd t + \sqrt{2} \rmd B_t \eqsp,
\end{cases}
\end{equation}
where $(Y_0,\tilde{Y}_0)=(x,y)$.
Since $\nabla U$ is Lipschitz, then by  \cite[Theorem 2.5, Theorem 2.9, Chapter 5]{karatzas:shreve:1991}, this SDE has a unique strong solution $(Y_t,\tilde{Y}_t)_{t \geq 0}$  associated with $(B_t)_{t \geq 0}$. Moreover since $(Y_t ,\tilde{Y}_t)_{t \geq 0}$ is a solution of \eqref{eq:double_langevin},
\[
\norm[2]{Y_t - \tilde{Y}_t} =  \norm[2]{Y_0 - \tilde{Y}_0}- 2\int_0^t  \ps{\nabla U (Y_s) - \nabla U(\tilde{Y}_s)}{Y_s - \tilde{Y}_s} \rmd s \eqsp,
\]
which implies using \Cref{assum:potentialU} and Gr\"onwall's inequality that
\begin{equation*}
\norm[2]{Y_t - \tilde{Y}_t} \leq \norm[2]{Y_0 - \tilde{Y}_0} - 2m \int_0^t \norm[2]{Y_s - \tilde{Y}_s} \rmd s \leq \norm[2]{Y_0 - \tilde{Y}_0} \rme^{-2mt}\eqsp.
\end{equation*}
Since for all $t \geq 0$, the law of $(Y_t,\tilde{Y}_t)$ is a coupling between $\delta_x P_t$ and $\delta_y P_t$, by definition of $W_{2}$, $W_{2}(\delta_x P_t , \delta_y P_t ) \leq \PE[\Vert  Y_t - \tilde{Y}_t  \Vert ^2]^{1/2}$, which  concludes the proof.
\item The proof is a direct consequence of \ref{item:moment_diffusion_item} and \ref{item:1:theo:convergence-WZ-strongly-convexD}
\end{enumerate}

\subsection{Proof of \Cref{theo:kind_drift}}
\label{proof:theo:kind_drift}
\begin{enumerateList}
\item
Note that the proof is trivial if $\ell <n$. Therefore we only need to consider the case $\ell \geq n$.
For any $\gamma \in \ooint{0,2/(m+L)}$, we have for all $x \in \rset^d$:
\[
\int_{\rset^d} \norm[2]{y-x^\star}   R_\gamma (x,\rmd y) = \norm{x-\gamma \nabla U(x) -x^\star}^2 + 2 \gamma d \eqsp.
\]
Using that $\nabla U (x^\star ) = 0$,   we get using the previous identity and \eqref{eq:convex_forte_contra1}:
\begin{align*}
\int_{\rset^d}& \norm[2]{y-x^\star}   R_\gamma (x,\rmd y)\\
& \leq \left(1-\kappa \gamma \right) \norm{x-x^\star}^2 + \gamma\left(\gamma - \frac{2}{m+L} \right)\norm{\nabla U(x) - \nabla U(x^\star)}^2 + 2 \gamma d \\
& \leq \left(1-\kappa \gamma \right) \norm{x-x^\star}^2  + 2 \gamma d \eqsp,
\end{align*}
where we have used for the last inequality that $\gamma \leq 2/(m+L)$.
Then  by definition \eqref{eq:iterate_kernel} of $Q^{n,\ell}_\gamma$ for $\ell,n \geq 1$, $\ell\geq n$,  the proof follows from a straightforward induction.
\item
By \ref{item:kind_drift_1}, we have for all $x \in \rset^d$ and $n \geq 1$,
\begin{align}
\nonumber
\int_{\rset^d} \norm[2]{y-\xstar}R_{\gamma}^n(x,\rmd y)& \leq (1-\kappa \gamma)^{n} \norm[2]{x-\xstar} + 2d \sum_{k=1}^{n} \gamma (1-\kappa \gamma)^{n-k} \\
\label{eq:proof:drift_euler}
&=  (1-\kappa \gamma)^{n} \norm[2]{x-\xstar} + 2\kappa^{-1}d (1-(1-\kappa \gamma)^{n}) \eqsp.
\end{align}
Since any compact set of $\rset^d$ is accessible and small for $R_{\gamma}$, then \cite[Theorem 15.0.1]{meyn:tweedie:2009} implies that $R_{\gamma}$ has a unique stationary distribution $\pi_{\gamma}$. Using \eqref{eq:proof:drift_euler}, the proof is along the same lines as \Cref{theo:convergence-WZ-strongly-convex}-\ref{item:moment_diffusion_item}.
\end{enumerateList}

\subsection{Proof of \Cref{theo:convergence_p_Euler}}
\label{sec:proof:convergence_p_Euler}

\begin{enumerateList}
\item  Let $(Z_k)_{k \geq 1}$ be a sequence of \iid~$d$-dimensional
Gaussian random variables. For $n \in \nset$, define the process
 $(X_{k}^{n,1},X_{k}^{n,2})_{k \geq 0}$ as follows: $(X_0^{n,1},X_0^{n,2})= (x,y)$  and  for $k \geq 0$,
\begin{equation}
\label{eq:definition_couplage_chaine}
X_{k+1}^{n,j} = X_{k}^{n,j}- \gamma_{k+n} \nabla U(X_{k}^{n,j}) + \sqrt{2 \gamma_{k+n}} Z_{k+1}  \quad \text{ $j=1,2$} \eqsp.
\end{equation}
Note that $X_{\ell}^{n,1}$ and $X_{\ell}^{n,2}$ are distributed according to $\delta_x Q_\gamma^{n,\ell}$ and $\delta_y Q_\gamma^{n,\ell}$ respectively. Therefore by definition of the Wasserstein distance of order $2$,  we get for any $\ell \geq n \geq 1$.
$W_{2}^{2}(\delta_x Q_\gamma^{n,\ell},\delta_y Q_\gamma^{n,\ell}) \leq \PE[\normLigne{X^{n,1}_{\ell} - X^{n,2}_{\ell}}^{2} ]$ and
\eqref{eq:convex_forte_contra1} implies for $k \geq n-1$,
\begin{align*}
\norm{X_{k+1}^{n,1} - X_{k+1}^{n,2}}^2 &= \norm{X_{k}^{n,1} - X_{k}^{n,2}}^2 + \gamma_{n+k}^2\norm{\nabla U(X_{k}^{n,1})- \nabla U(X_{k}^{n,2})}^2 \\
&\qquad - 2 \gamma_{n+k}\ps{X_{k}^{n,1} - X_{k}^{n,2}}{\nabla U(X_{k}^{n,1}) - \nabla U(X_{k}^{n,2})} \\
& \leq \left(1-\kappa \gamma_{n+k}  \right) \norm{X_{k}^{n,1}-X_{k}^{n,2}}^2 \eqsp.
\end{align*}
Therefore by  a straightforward induction we get for all $\ell \geq n$,
\[
\norm{X^{n,1}_{\ell} - X^{n,2}_{\ell}}^2 \leq \prod_{k=n}^{\ell} (1- \kappa \gamma_k) \norm{X^{n,1}_{0}-X^{n,2}_{0}}^2\eqsp.
\]
\item Let $\mu \in \Pens_{2}(\rset^d)$.  For all $n \geq 0$, $\mu R_{\gamma}^n \in \Pens_{2}(\rset^d)$. Then, by \Cref{theo:convergence_p_Euler}-\ref{item:contraction_Euler-i} for $\gamma \leq 2/(m+L)$, $R_{\gamma}$
is a strict contraction in $(\Pens_{2}(\rset^d),W_{2})$ and there is a unique fixed point  $\pi_\gamma$ which is the unique invariant distribution.
\end{enumerateList}

\subsection{Proof of \Cref{theo:distance_Euler_diffusion} }
\label{sec:proof:theo:distance_Euler_diffusion}
We preface the proof by a technical Lemma.
\begin{lemma}
\label{lem:moment_diffusion}
Let $(Y_t)_{t \geq 0}$ be the solution of \eqref{eq:langevin_2}
started at $x \in \rset^d$. For all $t \geq 0$ and $x \in \rset^d$,
\begin{equation*}
\expe{\norm[2]{Y_t-x}} \leq  d t(2+L^2 t^2/3) + (3/2) t^2 L^2 \norm[2]{x-x^\star}  \eqsp.
\end{equation*}
\end{lemma}
\begin{proof}
Let $\generator$ be the generator associated with $(P_t)_{t \geq 0}$ defined by \eqref{eq:generator}.
  Denote for all $x,y \in \rset^d$, $\tilde{V}_x(y) = \norm[2]{y-x}$.  Note that the process $(\tilde{V}_x(Y_t^{}) -\tilde{V}_x(x) - \int_0 ^t \generator \tilde{V}_x(Y^{}_s) \rmd s)_{ t \geq 0}$,
is a $(\mathcal{F}_t)_{ t \geq 0}$-martingale. Denote for all $t \geq
0$ and $x \in \rset^d$ by $\tilde{v}(t,x) = P_t \tilde{V}_x(x)$. Then we get,
\begin{equation}
\label{eq:Dynkin0_2}
\frac{\partial \tilde{v}(t,x)}{\partial t} = P_t \generator \tilde{V}_x(x) \eqsp.
\end{equation}
By \Cref{assum:potentialU}, we have for all $y \in \rset^d$, $\ps{\nabla U(y) - \nabla U(x)}{y-x} \geq m\norm[2]{x-y}$, which implies
\begin{equation*}
\generator \tilde{V}_x(y) = 2 \parenthese{ - \ps{\nabla U(y)}{y-x} +d }
\leq 2  \parenthese{ -m \tilde{V}_x(y) +d - \ps{\nabla U(x) } {y-x}} \eqsp.
\end{equation*}
Using \eqref{eq:Dynkin0_2}, this inequality and that $\tilde{V}_x$ is positive, we get
\begin{equation}
\label{eq:preparation_Grown}
\frac{\partial \tilde{v}(t,x)}{\partial t} = P_t \generator \tilde{V}_x(x)   \leq 2\parenthese{ d -
  \int_{ \rset^d}  \ps{\nabla U(x) } {y-x} P_t( x , \rmd y)}  \eqsp.
\end{equation}
By the Cauchy-Schwarz inequality, $\nabla U( x ^\star) = 0$, \eqref{eq:langevin_2} and the Jensen inequality, we have,
\begin{align*}
&\abs{\expe{ \ps{\nabla U(x) } {Y_t^{}-x}}} \leq \norm{\nabla U(x) } \norm{\expe{Y_t^{}-x}}\\
& \quad \leq \norm{\nabla U(x)} \norm{\expe{\int_0^t \left\{ \nabla U(Y_s^{}) - \nabla U(x^\star) \right\} \rmd s}} \\
&\quad \leq \sqrt{t} \norm{\nabla U(x)- \nabla U(x^\star) } \parenthese{ \int_0^t\PE\left[ \norm[2]{\nabla U(Y_s^{}) - \nabla U(x^\star) }\right]\rmd s }^{1/2}  \eqsp.
\end{align*}
Furthermore, by \Cref{assum:regularity_2} and \Cref{theo:convergence-WZ-strongly-convex}-\ref{item:moment_diffusion_2}, we have
\begin{align}
\nonumber
&\abs{\int_{ \rset^d}  \ps{\nabla U(x) } {y-x} P_t( x , \rmd y)} \leq  \sqrt{t} L^2 \norm{x-x^\star } \parenthese{\int_0^t\mathbb{E}\left[ \norm[2]{Y_s^{}-x^\star }\right]\rmd s}^{1/2}  \\
\nonumber
 &\qquad \qquad \leq  \sqrt{t}L^2 \norm{x-x^\star } \left(  \frac{1-\rme^{-2mt}}{2m} \norm[2]{x-x^\star} + \frac{2tm + \rme^{-2mt}-1}{2m} \frac{d}{m} \right)^{1/2} \\
\nonumber
&\qquad \qquad \leq  L^2 \norm{x-x^\star }  ( t \norm{x- x^{\star}} + t^{3/2} d^{1/2})\eqsp, \phantom{qfsdq}
\end{align}
where we used for the last line that by the Taylor theorem with remainder term, for all $s \geq 0 $, $(1-\rme^{-2ms})/(2m) \leq s$ and $(2ms+\rme^{-2ms}-1)/(2m) \leq m s^2$, and the inequality $\sqrt{a+b} \leq \sqrt{a} + \sqrt{b}$. Plugging this upper bound in \eqref{eq:preparation_Grown}, and since $ 2\norm{x-x^\star } t^{3/2} d^{1/2} \leq  t \norm[2]{x-x^\star} + t^2 d $, we get
\[
\frac{\partial \tilde{v}(t,x)}{\partial t} \leq   2d + 3 L^2 t \norm[2]{x-x^\star } + L^2t^{2} d
\]
Since $\tilde{v}(0,x) = 0$, the proof is completed by integrating this result.
\end{proof}

To show \Cref{theo:distance_Euler_diffusion} and \Cref{theo:distance_Euler_diffusionD}, since $\pi$ is invariant for
$P_t$ for all $t \geq 0$, it suffices to get some bounds on $W_2(\delta_x
Q^n_\gamma, \nu_0 P_{\Gamma_n})$, with $\nu_0 \in \Pens_2(\rset^d)$
and take $\nu_0 = \pi$. To do so, we construct a
coupling between the diffusion and the linear interpolation of the Euler discretization. An obvious candidate is the synchronous coupling $(Y_t,\overline{Y}_t)_{t \geq 0}$ defined for all $n \geq 0$ and  $t \in \coint{\Gamma_n,\Gamma_{n+1}}$ by
\begin{equation}
\label{eq:definition_couplage_2}
\begin{cases}
 Y_t = Y_{\Gamma_n}  - \int_{\Gamma_n}^t\nabla U( Y_s) \rmd s + \sqrt{2} (B_t - B_{\Gamma_n})  \\
 \Ybar_t= \Ybar_{\Gamma_n}- \nabla U(\Ybar_{\Gamma_n}) (t-\Gamma_n) + \sqrt{2}( B_t - B_{\Gamma_n}) \eqsp,
\end{cases}
\end{equation}
with $Y_0$ is distributed according to $\nu_0$, $\Ybar_0=x$ and $(\Gamma_n)_{n \geq 1}$ is given in \eqref{eq:def_Gamma}.
Therefore since for all $n \geq 0$, $W_2^2(\delta_x  Q_\gamma^n , \nu_0 P_{\Gamma_n}) \leq \PE[\normLigne{Y_{\Gamma_n} - \Ybar_{\Gamma_n}}^2]$,   taking $\nu_0=\pi$,
we derive an explicit bound on the Wasserstein distance
between the sequence of distributions $(\delta_x Q^k_\gamma)_{k \geq 0}$
and the stationary measure $\pi$ of the Langevin diffusion \eqref{eq:langevin_2}.

Let $(\filtration_t')_{t \geq 0}$ be the
filtration associated with $(B_t)_{t \geq 0}$ and $(Y_0,\overline{Y}_0)$.
\begin{lemma}
\label{lem:distance_Euler_diffusion_part1}
Assume \Cref{assum:regularity_2} and \Cref{assum:potentialU}. Let $(\gamma_k)_{k \geq 1}$
be a non-increasing sequence with $\gamma_1 \leq 1/(m+L)$.
Let $\zeta_0 \in \Pens_2(\rset^d \times \rset^d)$, $(Y_t,\overline{Y}_t)_{t \geq 0}$ such that $(Y_0, \overline{Y}_0)$ is distributed according to $\zeta_0$ and
given by \eqref{eq:definition_couplage_2}.  Then almost surely for all $n \geq 0$ and $\epsilon >0$,
\begin{align}
\label{eq:as_ineq1}
\norm{Y_{\Gamma_{n+1}}- \overline{Y}_{\Gamma_{n+1}}}^2 &\leq  \defEns{1- \gamma_{n+1}\left(\kappa-2\epsilon\right)} \norm{Y_{\Gamma_{n}}- \overline{Y}_{\Gamma_{n}}}^2
\\
\nonumber
& + (2 \gamma_{n+1}+(2 \epsilon)^{-1}) \int_{\Gamma_n}^{\Gamma_{n+1}}\norm{ \nabla U(Y_s) - \nabla U(Y_{\Gamma_n})}^2 \rmd s  \eqsp, \\
\label{eq:CPE_ineq1}
\CPE{\norm{Y_{\Gamma_{n+1}}- \overline{Y}_{\Gamma_{n+1}}}^2}{{\filtration^{'}_{\Gamma_n}}}
&\leq \defEns{1- \gamma_{n+1}\left(\kappa-2\epsilon\right)} \norm{Y_{\Gamma_{n}}- \overline{Y}_{\Gamma_{n}}}^2 \\
\nonumber
&  +  L^2\gamma_{n+1}^2 (1/(4\epsilon) +  \gamma_{n+1})\left(2 d +L^2\gamma_{n+1} \norm[2]{Y_{\Gamma_n}-x^\star} + dL^2 \gamma_{n+1}^2/6 \right) \eqsp.
\end{align}
\end{lemma}
\begin{proof}
We first show \eqref{eq:as_ineq1}. Set $\Theta_n= Y_{\Gamma_{n}}- \overline{Y}_{\Gamma_{n}}$.
By definition we have:
\begin{multline}
\label{eq:02/01}
\norm{\Theta_{n+1}}^2 =
 \norm{\Theta_n}^2+ \norm{ \int_{\Gamma_n}^{\Gamma_{n+1}}\defEns{ \nabla U(Y_s) - \nabla U(\overline{Y}_{\Gamma_{n}})} \rmd s}^2 \\- 2 \gamma_{n+1}\ps{\Theta_n}{\nabla U(Y_{\Gamma_n}) - \nabla U(\overline{Y}_{\Gamma_{n}})  } - 2 \int_{\Gamma_n}^{\Gamma_{n+1}}  \ps{\Theta_n}{\defEns{\nabla U(Y_s) - \nabla U(Y_{\Gamma_n})}} \rmd s   \eqsp.
\end{multline}
Young's inequality and  Jensen's inequality  imply
\begin{multline*}
 \norm{ \int_{\Gamma_n}^{\Gamma_{n+1}}\defEns{ \nabla U(Y_s) - \nabla U(\overline{Y}_{\Gamma_{n}}) } \rmd s}^2
\leq 2 \gamma_{n+1}^2 \norm{\nabla U ( Y_{\Gamma_n})- \nabla U(\overline{Y}_{\Gamma_{n}})}^2\\ +  2\gamma_{n+1} \int_{\Gamma_n}^{\Gamma_{n+1}}\norm{ \nabla U(Y_s) - \nabla U(Y_{\Gamma_n})}^2 \rmd s \eqsp.
\end{multline*}
Using \eqref{eq:convex_forte_contra1},  $\gamma_1 \leq 1/(m+L)$ and $(\gamma_k)_{k \geq 1}$ is non-increasing, \eqref{eq:02/01} becomes
\begin{align}
\nonumber
  \norm{\Theta_{n+1}}^2 \leq
 \defEns{1-\gamma_{n+1} \kappa } \norm{\Theta_n}^2+  2\gamma_{n+1} \int_{\Gamma_n}^{\Gamma_{n+1}}\norm{ \nabla U(Y_s) - \nabla U(Y_{\Gamma_n})}^2 \rmd s  \\
  \label{eq:002/01}
- 2 \int_{\Gamma_n}^{\Gamma_{n+1}}  \ps{\Theta_n}{\defEns{\nabla U(Y_s) - \nabla U(Y_{\Gamma_n})}} \rmd s   \eqsp.
\end{align}
Using the  inequality $|\ps{a}{b}| \leq \epsilon \|a\|^2 + (4 \epsilon)^{-1} \|b\|^2$  concludes the proof of \eqref{eq:as_ineq1}.

 We now prove \eqref{eq:CPE_ineq1}. Note that \eqref{eq:as_ineq1} implies that
\begin{multline}
\label{eq:distance_diffusion_first_part_fourth_estimate}
\CPE{\norm{\Theta_{n+1}}^2}{\filtration^{'}_{\Gamma_n}} \leq
 \defEns{1-\gamma_{n+1} (\kappa-2\epsilon) } \norm{\Theta_n}^2 \\
 +(2 \gamma_{n+1}+(2 \epsilon)^{-1}) \int_{\Gamma_n}^{\Gamma_{n+1}} \CPE{\norm{ \nabla U(Y_s) - \nabla U(Y_{\Gamma_n})  }^2}{\filtration^{'}_{\Gamma_n}} \rmd s \eqsp.
\end{multline}
By \Cref{assum:regularity_2}, the Markov property of $(Y_t)_{t \geq 0}$ and \Cref{lem:moment_diffusion}, we have
\begin{multline*}
\nonumber
\int_{\Gamma_n}^{\Gamma_{n+1}} \CPE{\norm{  \nabla U(Y_s) - \nabla U(Y_{\Gamma_n})  }^2}{\filtration^{'}_{\Gamma_n}} \rmd s 
\\
 \leq L^2 \left( d \gamma_{n+1}^2 + dL^2 \gamma_{n+1}^4 /12+ (1/2)L^2 \gamma_{n+1}^3 \norm[2]{Y_{\Gamma_n}-x^\star}\right) \eqsp.
\end{multline*}
The proof is then concluded plugging this bound   in \eqref{eq:distance_diffusion_first_part_fourth_estimate} .
\end{proof}

\begin{proof}[Proof of \Cref{theo:distance_Euler_diffusion}]
  Let $x \in \rset^d$, $n \geq 1$ and $\zeta_0=\pi\otimes \delta_x$.
Let $(Y_t, \overlineY_t)_{t \geq 0}$ with $(Y_0,\overlineY_0)$ distributed according to $\zeta_0$ and defined by \eqref{eq:definition_couplage_2}. By definition of $W_2$ and since for all $t \geq 0$, $\pi$ is invariant for $P_t$, $W^2_2( \mu_0 Q^n , \pi) \leq \expe{\norm{Y_{\Gamma_n} - \overline{Y}_{\Gamma_{n}}}^2}$.
 \Cref{lem:distance_Euler_diffusion_part1} with $\epsilon = \kappa/4$,  \Cref{theo:convergence-WZ-strongly-convex}-\ref{item:moment_diffusion_2} imply, using a straightforward induction, that for all $n \geq 0$
\begin{equation}
\label{eq:distance_induction}
\expe{\norm[2]{Y_{\Gamma_n} - \overline{Y}_{\Gamma_n}}} \leq u_n^{(1)}(\gamma)
\int_{\rset^d} \norm[2]{y-x} \pi(\rmd y)  + A_n (\gamma) \eqsp,
\end{equation}
 where $(u_n^{(1)}(\gamma))_{n \geq 1}$ is given by \eqref{eq:def_u1}, and
\begin{multline}
\label{eq:def_A_n_proof_1}
A_n (\gamma) \eqdef  L^2 \sum_{i=1}^n \gamma_i^2 \defEns{\kappa^{-1} + \gamma_i} (2d + d L^2 \gamma_i^2/6) \prod_{k=i+1}^n(1- \kappa\gamma_k/2)
\\
+L^4 \sum_{i=1}^n \tilde{\delta}_i \gamma_i^3 \defEns{\kappa^{-1} + \gamma_{i}}  \prod_{k=i+1}^n(1- \kappa \gamma_k/2)
\end{multline}
with
\begin{equation*}
\label{eq:defdeltai}
\tilde{\delta}_i = \rme^{-2m \Gamma_{i-1}} \expe{\norm[2]{Y_0 - x^\star}} + (1-\rme^{-2m \Gamma_{i-1}} )(d/m) \leq d/m \eqsp.
\end{equation*}
Since $Y_0$ is distributed according to $\pi$, \Cref{theo:convergence-WZ-strongly-convex}-\ref{item:moment_diffusion_item} shows that for all $i \in \{ 1, \cdots, n \}$,
\begin{equation}
\label{eq:boundtilde_delta_1}
  \tilde{\delta}_i \leq d/m \eqsp.
\end{equation}
In addition since for all $y \in \rset^d$, $\norm[2]{x-y} \leq 2(\norm[2]{x-x^\star} +\norm[2]{x^\star-y})$, using \Cref{theo:convergence-WZ-strongly-convex}-\ref{item:moment_diffusion_item}, we get $\int_{\rset^d} \norm[2]{y-x} \pi(\rmd y) \leq \norm[2]{x-\xstar} +d/m$. Plugging this result, \eqref{eq:boundtilde_delta_1} and \eqref{eq:def_A_n_proof_1} in \eqref{eq:distance_induction}  completes the proof.
\end{proof}

\subsection{Proof of \Cref{coro:distance_Euler_target}}
\label{sec:proof-lem_rec_2}

We preface the proof by a technical lemma.
\begin{lemma}
\label{lem:suite_recurrence_2}
Let $(\gamma_k)_{k \geq 1}$ be a sequence of non-increasing real numbers, $\tildem >0$ and $\gamma_1 < \tildem^{-1}$. Then for all $n \geq 0$, $j \geq 1$
and $\ell \in \defEns{1,\dots,n+1}$,
\begin{equation*}
\sum_{i=1}^{n+1} \prod_{k=i+1}^{n+1} \left(1- \tildem \gamma_k  \right)\gamma_i^j
  \leq \prod_{k=\ell}^{n+1}(1  - \tildem \gamma_k)
  \sum_{i=1}^{\ell -1} \gamma_i^j + \frac{\gamma_{\ell}^{j-1}}{
    \tildem}  \eqsp.
\end{equation*}
\end{lemma}

\begin{proof}
  Let $ \ell \in \defEns{1, \dots, n+1}$. Since $(\gamma_k)_{k \geq 1}$ is non-increasing and $\gamma_1 < \varpi^{-1}$,
\begin{align}
\nonumber
\sum_{i=1}^{n+1} \prod_{k=i+1}^{n+1} \left(1- \tildem \gamma_k \right)\gamma_i^j &=\sum_{i=1}^{\ell-1} \prod_{k=i+1}^{n+1} \left(1- \tildem \gamma_k \right)\gamma_i^j  +\sum_{i=\ell}^{n+1} \prod_{k=i+1}^{n+1} \left(1- \tildem \gamma_k \right)\gamma_i^j \\
\nonumber
& \leq \prod_{k=\ell}^{n+1}(1  - \tildem \gamma_k) \sum_{i=1}^{\ell -1}  \gamma_i^j +\gamma_{\ell}^{j-1} \sum_{i=\ell}^{n+1} \prod_{k=i+1}^{n+1} \left(1- \tildem \gamma_k \right)\gamma_i\\
\nonumber
& \leq \prod_{k=\ell}^{n+1}(1  - \tildem \gamma_k)  \sum_{i=1}^{\ell -1} \gamma_i^j + \frac{\gamma_{\ell}^{j-1}}{ \tildem} \eqsp.
\end{align}
\end{proof}

\begin{proof}[Proof of \Cref{coro:distance_Euler_target}]
By \Cref{theo:distance_Euler_diffusion}, it suffices to show that $u_n^{(1)}(\gamma)$ and $u_n^{(2)}(\gamma)$, defined by \eqref{eq:def_u1} and \eqref{eq:def_u2} respectively, goes to $0$ as $n \to \plusinfty$. Using the bound $1+t \leq \rme^t$ for $t \in \rset$, and $\lim_{n \to \plusinfty} \Gamma_n = \plusinfty$, we have $\lim_{n \to \plusinfty} u_n^{(1)}(\gamma) =0$. Since   $(\gamma_k)_{k\geq 0}$ is non-increasing, note that to show that $\lim_{n \to \plusinfty} u_n^{(2)}(\gamma) =0$, it suffices to prove
$
\lim_{n\to \plusinfty}\sum_{i=1}^{n} \prod_{k=i+1}^{n} (1- \kappa \gamma_k/2  )\gamma^2_i = 0
$.
But since $(\gamma_k)_{k \geq 1}$ is non-increasing, there exists $c \geq 0$ such that  $c \Gamma_n \leq n-1$ and by \Cref{lem:suite_recurrence_2} applied with $\ell = \floor{c \Gamma_n}$ the integer part of $c \Gamma_n$:
\begin{equation}
\label{eq:coro_distance}
\sum_{i=1}^{n} \prod_{k=i+1}^{n} \left(1-\kappa \gamma_k/2 \right) \gamma^2_i \leq
2\kappa^{-1} \gamma_{\floor{c \Gamma_n}}
+ \exp \left(- 2^{-1}\kappa (\Gamma_n -  \Gamma_{\floor{ c \Gamma_n}}) \right) \sum_{i=1}^{\floor{c \Gamma_n}-1} \gamma_i \eqsp.
\end{equation}
Since $\lim_{k \to \plusinfty} \gamma_k = 0$, by the Ces\'aro theorem, we have $\lim_{n \to \plusinfty} n^{-1} \Gamma_n = 0$.
 Then using that $\lim_{n \to \plusinfty}\Gamma_n = \plusinfty$, we get $\lim_{n \to \plusinfty}  \Gamma_{\floor{c \Gamma_n}}/\Gamma_n = 0$,
and the conclusion follows from combining in \eqref{eq:coro_distance}, this limit, $\lim_{k \to \plusinfty} \gamma_k = 0$, $\lim_{n \to \plusinfty} \Gamma_n = \plusinfty$ and $\sum_{i=1}^{\floor{c \Gamma_n}-1} \gamma_i \leq c \gamma_1 \Gamma_n$.
\end{proof}

\subsection{Proofs of \Cref{theo:distance_Euler_diffusionD}}
\label{sec:proof:distance_Euler_diffusionD}

\begin{lemma}
\label{lem:distance_Euler_diffusion_part1D}
Assume \Cref{assum:regularity_2}, \Cref{assum:potentialU} and \Cref{assum:reg_plus}. Let $(\gamma_k)_{k \geq 1}$
be a non-increasing sequence with $\gamma_1 \leq 1/(m+L)$.
Let  $\zeta_0 \in \Pens_2(\rset^d \times \rset^d)$ and $(Y_t,\overline{Y}_t)_{t \geq 0}$ be defined by \eqref{eq:definition_couplage_2} such that $(Y_0, \overline{Y}_0)$ is distributed according to $\zeta_0$. Then for all $n \geq 0$ and $\epsilon >0$, almost surely
\begin{multline*}
\CPE{\norm[2]{Y_{\Gamma_{n+1}}- \overline{Y}_{\Gamma_{n+1}}}}{{\filtration_{\Gamma_n}}}
\leq \defEns{1- \gamma_{n+1}\left(\kappa-2\epsilon\right)} \norm{Y_{\Gamma_{n}}- \overline{Y}_{\Gamma_{n}}}^2  \\
 \phantom{aaaaaaaaaaaaaa}
+ \gamma_{n+1}^3\left\lbrace d\parentheseDeux{2L^2 + \gamma_{n+1}^2 L^4 /6+ \epsilon^{-1}(d \tilde{L}^2/3 + \gamma_{n+1}L^4/4)  }\right.\\
 \phantom{aaaaaaaaaaaaaaaaaaaaaaaaa} \left. +
L^4(\epsilon^{-1}/3+\gamma_{n+1}) \norm[2]{Y_{\Gamma_n}-\xstar} \right\rbrace \eqsp.
\end{multline*}
\end{lemma}
\begin{proof}
Let $n \geq 0$ and $\epsilon >0$, and set $\Theta_n= Y_{\Gamma_{n}}- \overline{Y}_{\Gamma_{n}}$.
Using Itô's formula, we have for all $s \in \coint{\Gamma_n,\Gamma_{n+1}}$,
\begin{align}
\nonumber
  \nabla U(Y_s) - \nabla U(Y_{\Gamma_n}) = \int_{\Gamma_n}^s \defEns{\nabla^2 U(Y_u) \nabla U(Y_u) + \vec{\Delta} (\nabla U) (Y_u)} \rmd u \\
\label{eq:ito-D}
+ \sqrt{2}\int_{\Gamma_n}^s \nabla^2 U(Y_u) \rmd B_u \eqsp.
\end{align}
Since $\Theta_n$ is $\filtration_{\Gamma_n}$-measurable and $(\int_{0}^s \nabla^2 U(Y_u) \rmd B_u)_{s \in \ccint{0,\Gamma_{n+1}}}$ is a $(\mcf_{s})_{s \in \ccint{0,\Gamma_{n+1}}}$-martingale under \Cref{assum:regularity_2}, by \eqref{eq:ito-D} we have:
\begin{multline*}
  \abs{\CPE{\ps{\Theta_n}{\nabla U(Y_s) - \nabla U(Y_{\Gamma_n})}} {\filtration_{\Gamma_n}} } \\
 =  \abs{\ps{\Theta_n}{\CPE{\int_{\Gamma_n}^s \defEns{\nabla^2 U(Y_u) \nabla U(Y_u) + \vec{\Delta} (\nabla U) (Y_u)} \rmd u}{\filtration_{\Gamma_n}}}}
\end{multline*}
Combining this equality and  $|\ps{a}{b}| \leq \epsilon \|a\|^2 + (4 \epsilon)^{-1} \|b\|^2$ in  \eqref{eq:002/01}  we have
\begin{align}
\nonumber
&\CPE{\norm{\Theta_{n+1}}^2}{\filtration_{\Gamma_n}} \leq \defEns{1-\gamma_{n+1} (\kappa-2\epsilon) } \norm{\Theta_n}^2 + (2\epsilon)^{-1}A \\
\label{eq:distance_diffusion_part1_first_estimateD}
& \qquad \qquad \qquad \qquad + 2\gamma_{n+1} \CPE{\int_{\Gamma_n}^{\Gamma_{n+1}}\norm{ \nabla U(Y_s) - \nabla U(Y_{\Gamma_n})}^2 \rmd s}{\filtration_{\Gamma_n}}
 \eqsp,
\end{align}
where
\begin{equation*}
  A =  \int_{\Gamma_n}^{\Gamma_{n+1}} \norm[2]{\CPE{\int_{\Gamma_n}^s \nabla^2 U(Y_u) \nabla U(Y_u) + \vec{\Delta} (\nabla U) (Y_u)\rmd u}{\filtration_{\Gamma_n}} } \rmd s \eqsp.
\end{equation*}
We now separately bound the two last terms of the right hand side.
By \Cref{assum:regularity_2}, the Markov property of $(Y_t)_{t \geq 0}$ and \Cref{lem:moment_diffusion}, we have
\begin{multline}
\label{eq:04/022}
\int_{\Gamma_n}^{\Gamma_{n+1}} \CPE{\norm{  \nabla U(Y_s) - \nabla U(Y_{\Gamma_n})  }^2}{\filtration_{\Gamma_n}} \rmd s 
\\
 \leq L^2 \left( d \gamma_{n+1}^2 + dL^2 \gamma_{n+1}^4 /12+ (1/2)L^2 \gamma_{n+1}^3 \norm[2]{Y_{\Gamma_n}-x^\star}\right) \eqsp.
\end{multline}
We now bound $A$. We get using Jensen's inequality,  Fubini's theorem, $\nabla U(\xstar) = 0$ and \eqref{eq:borne_Lap_vec}
\begin{align}
\nonumber
A
 &\leq 2 \int_{\Gamma_n}^{\Gamma_{n+1}} (s - \Gamma_n)\int_{\Gamma_n}^s \CPE{\norm[2]{ \nabla^2 U(Y_u) \nabla U(Y_u) }}{\filtration_{\Gamma_n}}  \rmd  u \, \rmd s \\
\nonumber
 &\phantom{aaaaaaaaa}+  2 \int_{\Gamma_n}^{\Gamma_{n+1}} (s - \Gamma_n)\int_{\Gamma_n}^s \CPE{\norm[2]{\vec{\Delta} (\nabla U) (Y_u)}}{\filtration_{\Gamma_n}}    \rmd u\, \rmd s\\
\label{eq:04/023}
& \leq 2 \int_{\Gamma_n}^{\Gamma_{n+1}} (s - \Gamma_n) L ^4 \int_{\Gamma_n}^s \CPE{\norm[2]{ Y_u- \xstar} } {\filtration_{\Gamma_n}}\rmd u \, \rmd s  + 2 \gamma_{n+1}^3 d^2 \tilde{L}^2/3  \eqsp.
\end{align}
By \Cref{lem:moment_diffusion}-\ref{item:moment_diffusion_2}, the Markov property and for all $t \geq 0$, $1-\rme^{-t} \leq t$, we have for all $s \in \ccint{\Gamma_n,\Gamma_{n+1}}$,
\begin{equation*}
  \int_{\Gamma_n}^s \CPE{\norm[2]{ Y_u- \xstar} } {\filtration_{\Gamma_n}}\rmd u \leq
(2m)^{-1}(1-\rme^{-2m(s-\Gamma_n)})\norm[2]{Y_{\Gamma_n}-\xstar} + d(s-\Gamma_n)^2 \eqsp.
\end{equation*}
Using this inequality in \eqref{eq:04/023} and for all $t \geq 0$, $1-\rme^{-t} \leq t$ , we get
\begin{equation*}
A  \leq (2 L^4 \gamma_{n+1}^3/3)\norm[2]{Y_{\Gamma_n}-\xstar} + L^4 d \gamma_{n+1}^4/2 +  2\gamma_{n+1}^3 d^2 \tilde{L}^2/3 \eqsp.
\end{equation*}
Combining this bound and \eqref{eq:04/022}   in \eqref{eq:distance_diffusion_part1_first_estimateD} concludes the proof.
\end{proof}

\begin{proof}[Proof of \Cref{theo:distance_Euler_diffusionD}]
  The proof of is along the same lines as \Cref{theo:distance_Euler_diffusion}, using \Cref{lem:distance_Euler_diffusion_part1D} in place of \Cref{lem:distance_Euler_diffusion_part1}.
\end{proof}

\section{Proofs of \Cref{sec:quant-bounds-total}}
In this section are gathered the postponed proofs of \Cref{sec:quant-bounds-total}.

\subsection{Proof of \Cref{theo:tv_decreasing}}
\label{sec:proof-crefth_tv_decreasing}
Applying \Cref{lem:distance_Euler_diffusion_part1} or
   \Cref{lem:distance_Euler_diffusion_part1D}, we get that for all $x \in \rset^d$
\begin{equation}
\label{lem:distance_wasserstein_semi_group_RKER}
  W_1(\delta_x \QKer_{\gaStep}^n , \delta_x P_{\GaStep_n} ) \leq
\defEns{  \rateFun_n(x) }^{1/2} \eqsp,  \rateFun_n(x) =
\begin{cases}
\rateFun^{(1)}_{n,0}(x) & \text{(\Cref{assum:regularity_2}, \Cref{assum:potentialU})} \eqsp,\\
\rateFun^{(2)}_{n,0}(x) & \text{(\Cref{assum:regularity_2}, \Cref{assum:potentialU}, \Cref{assum:reg_plus})} \eqsp,
\end{cases}
\end{equation}
By the triangle inequality, we get
    \begin{equation}
      \label{eq:proof:theo_tv_first_decomp}
\tvnormEq{\delta_x P_{ \GaStep_\ell} - \delta_x \QKer_{\gaStep}^\ell} \leq
\tvnormEq{\defEns{\delta_x P_{ \GaStep_{n}} - \delta_x \QKer_{\gaStep}^{n} } P_{ \GaStep_{n+1,\ell}}} + \tvnormEq{\delta_x \QKer^{n}_{\gaStep}\defEns{\QKer_{\gaStep}^{n+1,\ell}-P_{\GaStep_{n+1,\ell}}} } \eqsp.
    \end{equation}
    Using \eqref{eq:item:coro:lip_smigroup_ii}
    and \eqref{lem:distance_wasserstein_semi_group_RKER}, we have
\begin{equation}
  \label{eq:proof:theo_tv_first_decomp_1}
  \tvnormEq{\defEns{\delta_x P_{ \GaStep_{n}} - \delta_x \QKer_{\gaStep}^{1,n} } P_{ \GaStep_{n+1,\ell}}} \leq (\rateFun_n(x) /( 4 \uppi \GaStep_{n+1,\ell}))^{1/2}  \eqsp.
\end{equation}
For the second term, by \cite[Equation 15]{durmus:moulines:2016} (note
that we have a different convention for the total variation distance)
and the Pinsker inequality, we have
\begin{multline*}
  \tvnormEq{\delta_x \QKer^{1,n}_{\gaStep}\defEns{\QKer_{\gaStep}^{n+1,\ell}-P_{\GaStep_{n+1,\ell}}} }^2  \\
\leq
2^{-3} L^2 \sum_{k=n+1}^\ell \defEns{(\gaStep_{k}^3/3) \int_{\rset^d} \norm[2]{\nabla U(z) - \nabla U(\xstar)} \QKer_{\gaStep}^{k-1}(x,\rmd z) +d \gaStep_{k}^2} \eqsp.
\end{multline*}
By \Cref{assum:regularity_2} and \Cref{theo:kind_drift}, we get
\begin{equation*}
   \tvnormEq{\delta_x \QKer^{1,n}_{\gaStep}\defEns{\QKer_{\gaStep}^{n+1,\ell}-P_{\GaStep_{n+1,\ell}}} }^2
\leq 2^{-3} L^2 \sum_{k=n+1}^\ell \defEns{(\gaStep_{k}^3L^2 /3)\FunMomentDEuler_{1,k-1}(x) +d \gaStep_{k}^2}
\eqsp.
\end{equation*}
Combining the last inequality and \eqref{eq:proof:theo_tv_first_decomp_1} in \eqref{eq:proof:theo_tv_first_decomp} concludes the proof.

\subsection{Proof of \eqref{eq:borne_tv_1_fixed_step_size}}
\label{sec:proof-eqref}

Consider the constant sequence $\gaStep_k = \gamma$ for all $k \in \nset^*$ with  $\gamma \in \ocint{0,1/(m+L)}$.  By \eqref{eq:defrateFun_1}, we have for all $n \in \nset^*$ and $x \in \rset^d$
\begin{equation*}
\rateFun^{(1)}_{n,0}(x) \leq \gamma \constD_1(\gamma,d) + \gamma^3 \constD_2(\gamma) \sum_{i=1}^n (1-\kappa \gamma/2)^{n-i}  \delta_{i,n,0}(x)  \eqsp,
\end{equation*}
where
\begin{equation*}
  \constD_1(\gamma,d) = 2 L^2 \kappa^{-1} \parenthese{\kappa^{-1}+ \gamma}\parenthese{2d + L^2 \gamma^2/6} \eqsp, \quad
  \constD_2(\gamma) =L^4  \parenthese{\kappa^{-1}+ \gamma} \eqsp.
\end{equation*}
In addition, using that $\kappa \geq 2m$ and for all $t \geq 0$, $1-t \leq \rme^{-t}$,
\begin{align}
\nonumber
  \sum_{i=1}^n (1-\kappa \gamma/2)^{n-i}  \delta_{i,n,0}(x) & =
  \sum_{i=1}^n \left[(1-\kappa \gamma/2)^{n-i}\defEnsG{\rme^{-2m \gaStep(i-1)} \norm[2]{x-\xstar}} \right.\\
\nonumber
& \qquad \qquad
\left. \defEnsD{ + \parenthese{1-\rme^{-2m \gaStep(i-1)}}(d/m)} \right] \\
\label{eq:bound_ratefun_1_const_gamma_2}
&  \leq n \rme^{-m \gaStep(n-1)}  \norm[2]{x-\xstar} + 2d(\kappa \gamma m)^{-1} \eqsp.
\end{align}
Therefore for all $n \geq 1$ and $x \in \rset^d$ we get
\begin{equation}
  \label{eq:bound_ratefun_1_const_gamma}
\rateFun^{(1)}_{n,0}(x) \leq \gamma \constD_1(\gamma) + \gamma^3 \constD_2(\gamma) \defEns{n \rme^{-m \gaStep(n-1)}  \norm[2]{x-\xstar} + 2d(\kappa \gamma m)^{-1} } \eqsp.
\end{equation}
Let now $\ell \in \nset^*$, $\ell \geq \ceil{\gaStep^{-1}}+1$ and $n = \ell - \ceil{\gaStep^{-1}}$. Then,
\begin{align}
\nonumber
& \sum_{k=n+1}^{\ell} \defEns{(\gaStep_{k}^3 L^2/3) \FunMomentDEuler_{1,k-1}(x)  + d \gaStep_{k}^2}\\
\nonumber
  &
 \leq
(L^{2}\gaStep^3/3)\defEns{(1-\kappa \gamma)^{n}(\ell-n)\norm[2]{x-\xstar} + 2 \kappa^{-1} \gaStep d (\ell-n)}
+
d \gamma^{2}(\ell-n)  \\
\nonumber
 &
 \leq  (L^{2}\gaStep^3/3)\defEns{(1+\gaStep^{-1} )(1-\kappa \gamma)^{\ell-\ceil{\gaStep^{-1}}}\norm[2]{x-\xstar} + 2(1+\gamma)\kappa^{-1}  d}
+ d \gamma(1+\gamma)  \eqsp.
\end{align}
Combining this inequality and \eqref{eq:bound_ratefun_1_const_gamma}  in the bound given by \Cref{theo:tv_decreasing} shows \eqref{eq:borne_tv_1_fixed_step_size}.

\subsection{Proof of \Cref{theo:bias_tv_gamma_fixed_imp}}
\label{sec:proof-theo_tv_fix}
We preface the proof by a preliminary lemma.
Define
for all $\gaStep >0$, the function $\nFun : \rset_+^* \to \nset$ by
\begin{equation}
\label{eq:def_nFun}
  \nFun(\gaStep) = \ceil{\log\parenthese{\gaStep^{-1}}/ \log(2)} \eqsp.
\end{equation}
  \begin{lemma}
    \label{theo:tv_fix}
Assume \Cref{assum:regularity_2}, \Cref{assum:potentialU} and \Cref{assum:reg_plus}. Let $\gaStep \in \ooint{0,1/(m+L)}$.
Then for all $x \in \rset^d$ and $\ell \in \nset^*$, $\ell > 2^{\nFun(\gaStep)}$,
\begin{multline*}
  \tvnorm{\delta_x P_{\ell \gaStep} - \delta_x \RKer_{\gaStep}^{\ell}} \leq
(\rateFun^{(2)}_{\ell-2^{\nFun(\gaStep)},0}(x) /( \uppi 2^{\nFun(\gaStep)+2} \gaStep))^{1/2} \\+
2^{-3/2} L  \defEns{(\gaStep^3L ^2/3)\FunMomentDEuler_{1,\ell-1}(x) +d \gaStep^2}^{1/2}
+ \sum_{k=1}^{\nFun(\gaStep)} (\rateFun_{2^{k-1},\ell-2^k}^{(2)}(x) /( \uppi 2^{k+1} \gaStep))^{1/2}   \eqsp.
\end{multline*}
where $\FunMomentDEuler_{1,\ell-1}(x)$ is defined by \eqref{eq:FunMomentDEuler} and for all $n_1,n_2 \in \nset$,  $\rateFun^{(2)}_{n_1,n_2}$ is given by \eqref{eq:defrateFun_2}.
  \end{lemma}
\begin{proof}
Let $ \gaStep \in \ooint{0,1/(m+L)}$ and $\ell \in \nset^*$.
For ease of notation, let $\nFunD = \nFun(\gaStep)$, and assume  that $\ell > 2^{\nFunD}$.
    Consider the following decomposition
\begin{multline}
\label{eq:decomp_proof_tv_fix_1}
\tvnormEq{\delta_x P_{\ell \gaStep} - \delta_x \RKer_{\gaStep}^\ell} \leq
\tvnormEq{\defEns{\delta_x P_{(\ell-\DnFunD) \gaStep} - \delta_x \RKer_{\gaStep}^{\ell-2^{\nFunD}}} P_{\DnFunD \gaStep}} \\
+ \tvnormEq{\delta_x\RKer_{\gaStep}^{\ell-1}\defEns{P_{\gaStep}-\RKer_{\gaStep}}}
+ \sum_{k=1}^{\nFunD} \tvnormEq{\delta_x \RKer_{\gaStep}^{\ell-2^{k}}\defEns{  P_{2^{k-1}\gaStep}-\RKer_{\gaStep}^{2^{k-1}}} P_{2^{k-1} \gaStep} }   \eqsp.
\end{multline}
We bound each term in the right hand side.
First by \eqref{eq:item:coro:lip_smigroup_ii} and \Cref{lem:distance_wasserstein_semi_group_RKER}, we have
\begin{equation}
  \label{eq:first_bound_proof_tv_fix}
  \tvnormEq{\defEns{\delta_x P_{(\ell-\DnFunD) \gaStep} - \delta_x \RKer_{\gaStep}^{\ell-2^{\nFunD}}} P_{\DnFunD \gaStep}} \leq
(\rateFun_{\ell-\DnFunD,0}^{(2)}(x) /( \uppi 2^{n+2} \gaStep))^{1/2}  \eqsp,
\end{equation}
where  $\rateFun^{(2)}_{n,0}(x)$ is given by \eqref{eq:defrateFun_2}. Similarly but using in addition \Cref{theo:kind_drift}, we have for all $k \in \defEns{1,\cdots,\nFunD}$,
\begin{equation}
\label{eq:second_bound_proof_tv_fix}
\tvnormEq{ \delta_x \RKer_{\gaStep}^{\ell-2^{k}}\defEns{  P_{2^{k-1}\gaStep}-\RKer_{\gaStep}^{2^{k-1}}} P_{2^{k-1} \gaStep}}
\leq (\rateFun_{2^{k-1},\ell-2^k}^{(2)}(x) /( \uppi 2^{k+1} \gaStep))^{1/2}  \eqsp,
\end{equation}
where $\rateFun_{2^{k-1},\ell-2^k}^{(2)}(x)$ is  given by \eqref{eq:defrateFun_2}.
For the last term, by \cite[Equation 11]{dalalyan:2014} and the Pinsker inequality, we have
\begin{equation*}
 \tvnormEq{\delta_x \RKer_{\gaStep}^{\ell-1}\defEns{P_{\gaStep}-\RKer_{\gaStep}}}^2  \leq
2^{-3} L^2  \defEns{(\gaStep^3/3) \int_{\rset^d} \norm[2]{\nabla U(z)} \RKer_{\gaStep}^{\ell-1}(x,\rmd z) +d \gaStep^2} \eqsp.
\end{equation*}
By \Cref{assum:regularity_2} and \Cref{theo:kind_drift}, we get
\begin{equation}
  \label{eq:third_bound_proof_tv_fix}
 \tvnormEq{\delta_x \RKer_{\gaStep}^{\ell-1}\defEns{\RKer_{\gaStep}-P_{\gaStep}}}^2  \leq
2^{-3} L^2  \defEns{(\gaStep^3L^2/3)\FunMomentDEuler_{1,\ell-1}(x)  +d \gaStep^2} \eqsp.
\end{equation}
Combining \eqref{eq:first_bound_proof_tv_fix}, \eqref{eq:second_bound_proof_tv_fix} and \eqref{eq:third_bound_proof_tv_fix} in \eqref{eq:decomp_proof_tv_fix_1} concludes the proof.
\end{proof}

\begin{proof}[Proof of  \Cref{theo:bias_tv_gamma_fixed_imp}]
First for all $n \geq 1$ and $x \in \rset^d$, we have
\begin{equation*}
\label{eq:defrateFun_2_s}
\rateFun_{n,0}^{(2)}(x) \leq
 \gamma^2 \constE_1(\gamma,d) + \gamma^3 \constE_2(\gamma)   \sum_{i=1}^n  \prod_{k=i+1}^n(1- \kappa\gamma_k/2) \delta_{i,n,0}(x) \eqsp,
\end{equation*}
where
\begin{equation*}
\constE_1(\gamma,d) = 2 d \kappa^{-1}  \defEns{2L^2+ 4 \kappa^{-1}(d \tilde{L}^2/3 + \gamma L^4/4) + \gamma^2 L^4 /6 }  \eqsp,
 \constE_2(\gamma) = L^4(4 \kappa^{-1}/3 + \gamma) \eqsp.
\end{equation*}

By \eqref{eq:bound_ratefun_1_const_gamma_2}, we get for all $n \geq 1$ and $x \in \rset^d$,
\begin{equation}
\label{eq:bound_ratefun_2_const_gamma_1}
\rateFun_{n,0}^{(2)}(x) \leq \gamma^2 \constE_1(\gamma,d) + \gamma^3 \constE_2(\gamma)  \defEns{n \rme^{-m \gaStep(n-1)}  \norm[2]{x-\xstar} + 2d(\kappa \gamma m)^{-1}  }  \eqsp.
\end{equation}

On the other hand, for all $x \in \rset^d$, $\ell,n \in \nset$, $n \geq 1$, $\ell >n$ we have using that $\kappa \geq 2m$ and for all $t \geq 0$, $1-t \leq \rme^{-t}$,
\begin{align}
\nonumber
  \rateFun_{n,\ell}^{(2)}(x) &\leq  \gamma^3 n  \constE_1(\gamma)  +  \gamma^3 n \constE_2(\gamma) \defEns{\rme^{-m \gamma(n-1)} \FunMomentDEuler_{n,\ell}(x) +d/m } \\
\nonumber
& \leq  \gamma^3 n  \constE_1(\gamma) +  \gamma^3 n \constE_2(\gamma) \defEns{\rme^{-m \gamma(n-1)}\parenthese{(1-\kappa \gamma)^{\ell-n}\norm[2]{x-\xstar}+2\kappa^{-1}d} +d/m }\\
\label{eq:bound_ratefun_2_const_gamma_2}
& \leq \gamma^3 n  \constE_1(\gamma)  +  \gamma^3 n \constE_2(\gamma) \defEns{\rme^{-m \gamma(\ell-1)}\norm[2]{x-\xstar}+ 2\kappa^{-1}d+d/m } \eqsp.
\end{align}

Finally, for all $\ell \in \nset^*$ and $x \in\rset^d$, we have
\begin{equation}
\label{eq:bound_remainder_gamma_fixed_2}
(\gaStep^3L ^2/3)\FunMomentDEuler_{1,\ell-1}(x) +d \gaStep^2
\leq
  (\gaStep^3L ^2/3)\defEns{(1-\kappa \gamma)^{\ell-1}\norm[2]{x-\xstar} + 2d \kappa^{-1}} +d \gaStep^2  \eqsp.
\end{equation}

Combining \eqref{eq:bound_ratefun_2_const_gamma_1}, \eqref{eq:bound_ratefun_2_const_gamma_2} and \eqref{eq:bound_remainder_gamma_fixed_2} in the bound given by \Cref{theo:tv_fix}, and using that $\gamma^{-1} \leq 2^{\nFun(\gaStep)} \leq 2 \gamma^{-1}$ we have for all $\ell \in \nset^*$, $\ell > 2^{\nFun(\gaStep)}$,
  \begin{align*}
&      \tvnorm{\delta_x P_{\ell \gaStep} - \delta_x \RKer_{\gaStep}^{\ell}} \leq
2^{-3/2} L  \parentheseDeux{  (\gaStep^3L ^2/3)\defEns{(1-\kappa \gamma)^{\ell-1}\norm[2]{x-\xstar} + 2d \kappa^{-1}} +d \gaStep^2 }^{1/2} \\
&+( 4 \uppi)^{-1/2}\parentheseDeux{
\gamma^2 \constE_1(\gamma) + \gamma^3 \constE_2(\gamma)  \defEns{(\ell - \gamma^{-1}) \rme^{-m \gaStep(\ell-2\gamma^{-1}-1)}  \norm[2]{x-\xstar} + 2d(\kappa \gamma m)^{-1}  }}^{1/2} \\
&+ \sum_{k=1}^{\nFun(\gaStep)} \parentheseDeux{\frac{\gamma^3 2^{k-1}  \constE_1(\gamma)  +  \gamma^3 2^{k-1} \constE_2(\gamma) \defEns{\rme^{-m \gamma(\ell-2^k-1)}\norm[2]{x-\xstar}+ 2\kappa^{-1}d+d/m }} { \uppi 2^{k+1} \gaStep}}^{1/2} \\
& \phantom{    \tvnorm{\delta_x P_{\ell \gaStep} - \delta_x \RKer_{\gaStep}^{\ell}}}\leq 2^{-3/2} L  \defEns{  (\gaStep^3L ^2/3)\defEns{(1-\kappa \gamma)^{\ell-1}\norm[2]{x-\xstar} + 2d \kappa^{-1}} +d \gaStep^2 }^{1/2} \\
&+( 4 \uppi)^{-1/2}\parentheseDeux{
\gamma^2 \constE_1(\gamma) + \gamma^3 \constE_2(\gamma)  \defEns{(\ell - \gamma^{-1}) \rme^{-m \gaStep(\ell-2\gamma^{-1}-1)}  \norm[2]{x-\xstar} + 2d(\kappa \gamma m)^{-1}  }}^{1/2} \\
&+ ( 4 \uppi)^{-1/2} \nFun(\gaStep) \parentheseDeux{\gamma^2  \constE_1(\gamma)  +  \gamma^2  \constE_2(\gamma) \defEns{\rme^{-m \gamma(\ell-2 \gamma^{-1}-1)}\norm[2]{x-\xstar}+ 2\kappa^{-1}d+d/m } }^{1/2} \eqsp.
  \end{align*}
Letting $\ell$ go to infinity, using \Cref{propo:reflexion_coupling_lip}-\ref{item:reflexion_coupling_lip_2} and \Cref{theo:convergence_discrete_chain}-\ref{item:convergence_discrete_chain_2}, we get the desired conclusion.
\end{proof}


\section{Proof of \Cref{SEC:MSE_TV}}
In this section are gathered the postponed proofs of \Cref{SEC:MSE_TV}.

\subsection{Proof of \Cref{theo:var} and \Cref{theo:var_tv}}
Our main tool in the proof of \Cref{theo:var} and \Cref{theo:var_tv}  is the
Gaussian Poincar{\'e} inequality \cite[Theorem~3.20]{boucheron:lugosi:massart:2013}
which
can be applied to $R_\gamma(y,\cdot)$ defined by \eqref{eq:definition-Rgamma},  noticing that  $R_\gamma(y,\cdot)$ is a Gaussian distribution  with
mean $y - \gamma \nabla U(y)$ and covariance matrix $2 \gamma \operatorname{I}_d$:
for all Lipschitz function $g : \rset^d \to \rset$
\begin{equation}
\label{lem:var_1}
 R_\gamma \defEns{ g(\cdot) - R_\gamma g(y)}^2(y)   \leq 2 \gamma  \norm{g}_{\Lip}^2 \eqsp.
\end{equation}
To go further, we decompose $\hat{\pi}_n^N(f)-\PE_x[\hat{\pi}_n^N(f)]$, for $f : \rset^d\to \rset$, Lipschitz or measurable and bounded, as the sum of martingale increments,
\wrt~$(\mcg_n)_{n \geq 0}$, the natural filtration associated with Euler approximation $(X_n)_{n \geq 0}$, and we get
\begin{multline}
\label{eq:decomposition-variance-new}
\VarDeux{x}{\hat{\pi}_n^N(f)}= \sum_{k=N}^{N+n-1} \PE_x\left[ \parenthese{ \CPE[x]{\hat{\pi}_n^N(f)}{\mcg_{k+1}} - \CPE[x]{\hat{\pi}_n^N(f)}{\mcg_k}}^2\right] \\
+
\expeMarkov{x}{\parenthese{\CPE[x]{\hat{\pi}_n^N(f)}{\mcg_N} - \PE_x[\hat{\pi}_n^N(f)]}^2}\eqsp.
\end{multline}
Since $\hat{\pi}_n^N(f)$ is an additive functional, the martingale increment
$\CPE[x]{\hat{\pi}_n^N(f)}{\mcg_{k+1}}- \CPE[x]{\hat{\pi}_n^N(f)}{\mcg_k}$ has a simple expression.
For $k = N+n-1,\dots,N+1$, define backward in time the function
\begin{equation}
\label{eq:def_mart_f-new}
\martInc^N_{n,k}: x_k \mapsto \omega_{k,n}^N f(x_{k}) +R_{\gamma_{k+1}} \martInc^N_{n,k+1}( x_k ) \eqsp,
\end{equation}
where $\martInc_{n,N+n}^N : x_{N+n} \mapsto \martInc_{n,N+n}^N (x_{N+n}) =  \weight{N+n} f(x_{N+n})$. Denote finally
\begin{equation}
\label{eq:def_F_f-new}
\martIncF^{N}_{n}: x_N \mapsto R_{\gamma_{N+1}} \martInc^N_{n,N+1}(x_N)\eqsp.
\end{equation}
Note that for $k \in \{N,\dots,N+n-1\}$, by the Markov property,
\begin{equation}
\label{eq:relation_mart_var_cond-new}
\martInc^N_{n,k+1}(X_{k+1}) - R_{\gamma_{k+1}}\martInc^N_{n,k+1}(X_{k})= \CPE[x]{\hat{\pi}_n^N(f)}{\mcg_{k+1}} - \CPE[x]{\hat{\pi}_n^N(f)}{\mcg_k} \eqsp,
\end{equation}
and $\martIncF^N_{n}(X_N)= \CPE[x]{\hat{\pi}_n^N(f)}{\mcg_N}$.
With these notations, \eqref{eq:decomposition-variance-new} may be equivalently expressed as
\begin{align}
\nonumber
\VarDeux{x}{\hat{\pi}_n^N(f)}
 = \sum_{k=N}^{N+n-1} \PE_x\left[ R_{\gamma_{k+1}}\defEns{\martInc^N_{n,k+1}(\cdot)
 - R_{\gamma_{k+1}}\martInc^N_{n,k+1}(X_{k}) }^2 (X_k) \right]\\
\label{eq:decomposition-variance-alter-new}
  + \VarDeux{x}{\martIncF_n^N(X_{N})}  \eqsp.
\end{align}
Now for $k=N+n-1,\dots,N$, we will use the Gaussian Poincar{\'e} inequality \eqref{lem:var_1}
to the sequence of function $\martInc_{n,k+1}^N$
to prove that $x \mapsto R_{\gamma_{k+1}}\{\martInc^N_{n,k+1}(\cdot) - R_{\gamma_{k+1}}\martInc^N_{n,k+1}(x)\}^2(x)$ is uniformly bounded.
It is required to bound the Lipschitz constant of $\martInc_{n,k}^N$ .

\subsubsection{Proof of  \Cref{theo:var}}
\label{sec:proof-creftheo:var}
We preface the proof by two lemmas.
\begin{lemma}
\label{coro:induction_function_mart-new}
Assume \Cref{assum:regularity_2} and \Cref{assum:potentialU}. Let $(\gamma_k)_{k \geq 1}$ be a non-increasing sequence
with $\gamma_1 \leq 2/(m+L)$. Let $N \geq 0$ and $n \geq 1$.
 Then for all $y \in \rset^d$,  Lipschitz function $f$ and $k\in \defEns{N,\dots,N+n-1}$,
\begin{equation*}
 R_{\gamma_{k+1}}\defEns{\martInc^N_{n,k+1}(\cdot) - R_{\gamma_{k+1}}\martInc^N_{n,k+1}(y) }^2 (y)
\leq 8   \gamma_{k+1}\norm[2]{f}_{\Lip} (\kappa\Gamma_{N+2,N+n+1})^{-2} \eqsp,
 \end{equation*}
 where $\martInc^N_{n,k+1}$ is given by \eqref{eq:def_mart_f-new}.
\end{lemma}

\begin{proof}
By \eqref{eq:def_mart_f-new}, $\normLigne{\martInc^N_{n,k}}_{\Lip} \leq \sum_{i=k+1}^{N+n} \weight{i} \normLigne{Q_\gamma^{k+2,i}f}_{\Lip}$.
Using  \Cref{propo:contraction_Euler-new}, the bound $(1-t)^{1/2} \leq 1 - t/2$ for $t \in \ccint{0,1}$ and the definition of $\weight{i}$ given by \eqref{eq:def_GammaN_plusn}, we have
\begin{equation*}
\norm{\martInc^N_{n,k}}_{\Lip} \leq \norm{f}_{\Lip} \sum_{i=k+1}^{N+n} \weight{i} \prod_{j=k+2}^i (1- \kappa \gamma_j/2) \leq 2 \norm{f}_{\Lip}( \kappa \Gamma_{N+2,N+n+1} )^{-1} \eqsp.
\end{equation*}
Finally, the proof follows from \eqref{lem:var_1}.
\end{proof}

Also to control the last term in right hand side of \eqref{eq:decomposition-variance-alter-new},
we need to control the variance of $\martIncF_n^N(X_N)$ under $\delta_x Q^N_\gamma$. But similarly to the sequence of functions $\martInc_{n,k}^N$,
$\martIncF_n^N$ is Lipschitz by \Cref{propo:contraction_Euler-new} by definition, see \eqref{eq:def_F_f-new}.
Therefore it suffices to find some bound for the variance of $g$ under $\delta_y Q_\gamma^{n,p}$ , for $g: \rset^d \to \rset$ a Lipschitz function, $y \in \rset^d$ and $\gamma >0$, which is done using the following result.

\begin{lemma}
\label{lem:var_2}
Assume \Cref{assum:regularity_2} and \Cref{assum:potentialU}. Let  $(\gamma_k)_{k \geq 1}$ be a non-increasing sequence with $\gamma_1 \leq  2/(m+L)$. Let $g : \rset^d \to \rset$ be a Lipschitz function. Then for all $n,p \geq 1$, $n \leq p$ and $y \in \rset^d$
\[
0 \leq \int_{\rset^d } Q^{n,p}_\gamma(y, \rmd z) \defEns{ g(z) - Q^{n,p}_\gamma g(y)}^2 \leq  2 \kappa^{-1} \norm{g}^2_{\Lip} \eqsp,
\]
where $Q_\gamma ^{n,p}$ is given by \eqref{eq:iterate_kernel}.
\end{lemma}
\begin{proof}
By decomposing $g(X_p)- \CPE[y]{g(X_p)}{\mcg_n}= \sum_{k=n+1}^p \{ \CPE[y]{g(X_p)}{\mcg_k} - \CPE[y]{g(X_p)}{\mcg_{k-1}} \}$,
and using $\CPE[y]{g(X_p)}{\mcg_k}=  Q_\gamma^{k+1,p} g(X_k)$, we get
\begin{align*}
\CPVar[y]{g(X_p)}{\mcg_n}
&= \sum_{k=n+1}^p \CPE[y]{\CPE[y]{\left( \CPE[y]{g(X_p)}{\mcg_k} - \CPE[y]{g(X_p)}{\mcg_{k-1}} \right)^2}{\mcg_{k-1}} }{\mcg_n} \\
&= \sum_{k=n+1}^p \CPE[y]{ R_{\gamma_k} \defEns{ Q_\gamma^{k+1,p} g(\cdot) - R_{\gamma_k} Q_\gamma^{k+1,p} g(X_{k-1})}^2(X_{k-1})}{\mcg_n} \eqsp.
\end{align*}
\Cref{lem:var_1} implies $\CPVar[y]{g(X_p)}{\mcg_n} \leq 2 \sum_{k=n+1}^p \gamma_k \normLigne[2]{Q_\gamma^{k+1,p} g}_{\Lip}$. The proof follows from
\Cref{propo:contraction_Euler-new} and \Cref{lem:suite_recurrence_2}, using the bound $(1-t)^{1/2} \leq 1 - t/2$ for $t \in \ccint{0,1}$.
\end{proof}

\begin{corollary}
\label{coro:function_F}
Assume \Cref{assum:regularity_2} and \Cref{assum:potentialU}. Let $(\gamma_k)_{k \geq 1}$ be a non-increasing sequence
with $\gamma_1 \leq 2/(m+L)$.
Then for all Lipschitz function $f$ and $x \in \rset^d$, $\VarDeuxLigne{x}{\martIncF^N_n(X_N)}
\leq  8 \kappa^{-3} \norm{f}_{\Lip}^2  \Gamma_{N+2,N+n+1} ^{-2}$,
where $\martIncF^N_n$ is given by \eqref{eq:def_F_f-new}.
\end{corollary}
\begin{proof}
By \eqref{eq:def_F_f-new}  and \Cref{propo:contraction_Euler-new}, $\martIncF_n^N$ is Lipschitz function with $\normLigne{\martIncF^N_n}_{\Lip} \leq  \sum_{i=N+1}^{N+n} \weight{i} \normLigne{Q_\gamma^{N+1,i}f}_{\Lip}$.
Using  \Cref{propo:contraction_Euler-new}, the bound $(1-t)^{1/2} \leq 1 - t/2$ for $t \in \ccint{0,1}$ and the definition of $\weight{i}$ given by \eqref{eq:def_GammaN_plusn}, we have
\begin{equation*}
\norm{\martIncF^N_n}_{\Lip} \leq \norm{f}_{\Lip} \sum_{i=N+1}^{N+n} \weight{i} \prod_{j=N+2}^i (1- \kappa \gamma_j/2) \leq 2 \norm{f}_{\Lip}( \kappa \Gamma_{N+2,N+n+1} )^{-1} \eqsp.
\end{equation*}
The proof follows from \Cref{lem:var_2}.
\end{proof}
Plugging the bounds given by \Cref{coro:induction_function_mart-new} and \Cref{coro:function_F} in \eqref{eq:decomposition-variance-alter-new}, we have
\begin{align*}
\VarDeux{x}{\hat{\pi}_n^N(f)} &\leq  8 \kappa^{-2} \norm{f}_{\Lip}^2 \defEns{\Gamma_{N+2,N+n+1}^{-2} \Gamma_{N+1,N+n} + \kappa^{-1}\Gamma_{N+2,N+n+1}^{-2}} \\
&\leq 8\kappa^{-2} \norm{f}_{\Lip}^2 \defEns{\Gamma_{N+2,N+n+1}^{-1} + \Gamma_{N+2,N+n+1}^{-2}(\gamma_{N+1} + \kappa^{-1})} \eqsp.
\end{align*}
Using that $\gamma_{N+1} \leq 2/(m+L)$ concludes the proof of \Cref{theo:var}.

\subsubsection{Proof of \Cref{theo:var_tv}}
\label{sec:proof_theo_var_tv}
Let $k\in \defEns{N,\dots,N+n-1}$. We cannot
directly apply the Poincaré inequality \eqref{lem:var_1} since the function
$\martInc_{n,k}^N$, defined in \eqref{eq:def_mart_f-new}, is not Lipschitz. However, \Cref{theo:convergence_discrete_chain}-\ref{item:convergence_discrete_chain_3} shows that for all $ \ell, n \in \nset^*$, $n < \ell$, $Q^{n,\ell}_\gamma f$ is a Lipschitz function with
\begin{equation}
\label{eq:lipschitz-constant-Q}
\norm{Q^{n,\ell}_\gamma f}_{\Lip} \leq   \osc{f}/\{4 \uppi \LambdarMSE_{n,\ell}(\gamma)\}^{1/2} \eqsp.
\end{equation}
Using \eqref{eq:def_mart_f-new}, we may decompose
$\martInc^N_{n,k}= \weight{k+1} f + \tilde{\martInc}^N_{n,k}$, where
$\tilde{\martInc}^N_{n,k}= \sum_{i=k+2}^{N+n} \weight{i}
Q_\gamma^{k+2,i}f$ which is Lipshitz with constant
\begin{equation}
\label{eq:bound_norm_lip_phitilde}
  \norm{\tilde{\martInc}^N_{n,k}}_{\Lip} \leq \sum_{i=k+2}^{N+n} \weight{i}
\norm{  Q_\gamma^{k+2,i}f}_{\Lip} \leq  \osc{f} \sum_{i=k+2}^{N+n} \weight{i}
/ \left\{4 \uppi \LambdarMSE_{k+2,i}(\gamma) \right\}^{1/2} \eqsp.
\end{equation}
Using the inequality $(a+b)^2 \leq 2 a^2 + 2 b^2$,  \eqref{lem:var_1}, we finally get for any $y \in \rset^d$
\begin{multline}
\label{eq:lem:induction_function_mart_1_tv_1}
 R_{\gamma_{k+1}}\defEns{\martInc^N_{n,k+1}(\cdot) - R_{\gamma_{k+1}}\martInc^N_{n,k+1}(y) }^2 (y)
\leq 2 (\weight{k+1})^{2} \osc{f}^2 \\
+   \gamma_{k+1}  \osc{f}^2\defEns{\sum_{i=k+2}^{N+n} \weight{i} /\{\uppi \LambdarMSE_{k+2,i}(\gamma)\}^{1/2} }^2\eqsp.
 \end{multline}
It remains to control $\VarDeux{x}{\martIncF_n^N(X_{N})}$, where $\martIncF^{N}_{n}$ is defined in \eqref{eq:def_F_f-new}.
Using  \eqref{eq:lipschitz-constant-Q},
$\martIncF_n^N$ is a Lipschitz function with Lipschitz constant bounded by:
\begin{equation}
\label{eq:bound_norm_lip_psitilde}
\norm{\martIncF^N_n}_{\Lip}  \leq   \sum_{i=N+1}^{N+n} \weight{i} \norm{Q_\gamma^{N+1,i}f}_{\Lip} \leq  \osc{f} \sum_{i=N+1}^{N+n} \weight{i} /\{ 4 \pi \LambdarMSE_{N+1,i}(\gamma) \}^{1/2} \eqsp.
\end{equation}

By \Cref{lem:var_2}, we have the following
result which is the counterpart of \Cref{coro:function_F}:
for all $y \in \rset^d$,
\begin{equation}
\label{eq:lem:induction_function_mart_2_tv_1}
\VarDeux{y}{\martIncF^N_n(X_N)}
\leq  2 \kappa^{-1} \norm{f}_{\infty}^2 \defEns{ \sum_{i=N+1}^{N+n} \weight{i} /(\uppi \LambdarMSE_{N+1,i})^{1/2} }^2 \eqsp.
 \end{equation}
Finally, the proof follows from combining \eqref{eq:lem:induction_function_mart_1_tv_1} and \eqref{eq:lem:induction_function_mart_2_tv_1} in \eqref{eq:decomposition-variance-alter-new}.

\subsection{Bounds on $u_{0,n}^{(4)}(\gamma)$}
\label{sec:bounds-u_0-n4gamma}
Let $(\gamma_k)_{k \geq 1}$ be a non-increasing sequence of step size such that $\lim_{k \to \plusinfty} \gamma_k =0$ and $\lim_{k \to \plusinfty} \Gamma_k = \plusinfty$.
In this section, we show that  there exist $C_1,C_2 >0$ independent of $(\gamma_k)_{ k \geq 1}$  satisfying for any $n \in \nset^*$
\begin{equation}
  \label{eq:bound_u_n_4}
  C_1 \Gamma_n^{-1} \leq u_{0,n}^{(4)}(\gamma) \leq C_2 \Gamma_n^{-1} \eqsp,
\end{equation}
for 
$u_{0,n}^{(4)}$ defined in \eqref{eq:u_lap_tv}.
We consider the following decomposition of  $   u_{0,n}^{(4)}(\gaStep)$
\begin{align*}
  u_{N,n}^{(4)}(\gaStep) & =  w_n^1 + w_n^2 \eqsp,\\
  w^1_n &= \sum_{k=0}^{n-1} \gaStep_{k+1} \defEns{ \sum_{i=k+2}^{n} \frac{\weightD{i}}{(\uppi \LambdarMSE_{k+2,i}(\gamma))^{1/2}} }^2 \eqsp, \,
\,   w_n^2 =   \kappa^{-1} \defEns{ \sum_{i=1}^{n}\frac{ \weightD{i} }{(4 \uppi \LambdarMSE_{1,i}(\gamma))^{1/2}} }^2 \eqsp.
\end{align*}


Since $\kappa \Lambda_{n,\ell} =   \prod_{j=n}^{\ell}(1-\kappa \gaStep_j)^{-1} -1$  for $n,\ell \in \nset^*$ , using that for all $(a_i)_{i\in\{1,\ldots,k\}} \in \coint{0,1}^k$, $k\in \nset^*$, $\prod_{i=1}^k (1-a_i)^{-1}- 1 \geq \exp({\sum_{i=1}^k a_i}) - 1 \geq \sum_{i=1}^k a_i$, we have
\begin{equation}
  \label{eq:bound_Lambda}
   \prod_{j=n}^{\ell}(1-\kappa \gaStep_j) \leq 1/ (\kappa \Lambda_{n,\ell})  \leq 1/(\kappa^2 \Gamma_{n,\ell})\eqsp.
 \end{equation}

 From the left inequality, we conclude using the definition of
 $\weightD{i}$, $i \in \iint{1}{n}$, in \eqref{eq:def_GammaN_plusn}
 and the bound $(1-t)^{1/2} \leq 1 - t/2$ for $t \in \ccint{0,1}$,
 that there exists $C_1 >0$ independent of $(\gamma_k)_{k \geq 1}$ such that for any $n \in \nset^*$,
 \begin{equation}
   \label{eq:bound_w_n_1_1}
   C_1 \Gamma_{2,n+1} \leq w_n^1 \eqsp.
 \end{equation}

 Now from the right inequality in \eqref{eq:bound_Lambda} and using $(a+b)^2 \leq 2 (a^2+b^2)$, we have for any $n \in \nset^*$,
 \begin{equation}
   \label{eq:bound_w_n_2_0}
     w^1_n = 2 \sum_{k=0}^{n-1} \gaStep_{k+1} \defEns{ \sum_{i=k+2}^{p_k} \frac{\weightD{i}}{(\uppi \kappa^2 \Gamma_{k+2,i})^{1/2}} }^2 +  2\sum_{k=0}^{n-1} \gaStep_{k+1} \defEns{ \sum_{i=p_k+1}^{n} \frac{\weightD{i}}{(\uppi \Lambda_{k+2,i}(\gamma))^{1/2}} }^2 \eqsp,
   \end{equation}
   where $(p_k)_{k \in \nset^*}$ is any sequence of integers. Also we
   have using that $(\gamma_j)_{j \geq 1}$ is non-increasing and
   an integral comparison test that there exists $C \geq 0$ independent of $(\gamma_k)_{k \geq 1}$ such that
   for any $k,p \in \nset^*$, $k+2 \leq p$,
   \begin{align*}
    \sum_{i=k+2}^{p} \frac{\weightD{i}}{(\uppi  \kappa^2 \Gamma_{k+2,i}(\gamma))^{1/2}} &  \leq \Gamma_{2 , n+1}^{-1}  \sum_{i=k+2}^{p} \frac{\gamma_{i+1}}{(\uppi \kappa^2  \Gamma_{k+2,i})^{1/2}} \leq \Gamma_{2 , n+1}^{-1}  \sum_{i=k+2}^{p} \frac{ \Gamma_{k+2,i} - \Gamma_{k+2,i-1}}{(\uppi \kappa^2  \Gamma_{k+2,i})^{1/2}} \\
     &  \leq C  \Gamma_{2 , n+1}^{-1} \Gamma_{k+1,p}^{1/2}  \eqsp.
   \end{align*}
   Using this result in \eqref{eq:bound_w_n_2_0}, we obtain that for any $n \in \nset^*$,
 \begin{equation}
   \label{eq:bound_w_n_2_1}
     w^1_n \leq 2 C  \Gamma_{2 , n+1}^{-2} \sum_{k=0}^{n-1} \gaStep_{k+1} \Gamma_{k+1,p_k} +  2\sum_{k=0}^{n-1} \gaStep_{k+1} \defEns{ \sum_{i=p_k+1}^{n} \frac{\weightD{i}}{(\uppi \Lambda_{k+2,i}(\gamma))^{1/2}} }^2 \eqsp.
   \end{equation}
   Now taking for any $n \in \nset^*$, $k \in \iint{0}{n-1}$,
   \begin{equation}
     \label{eq:def_p_k}
     p_{k}= n \wedge \inf\defEns{ p \in \{k+1,\ldots,n-1\} \, : \, \Gamma_{k+1,p} \geq 1 }\eqsp,
   \end{equation}
   with the convention $\inf \emptyset = \plusinfty$,
   we have for any $i \in \iint{p_k+1}{n}$, $p_k+1 \leq n$, using for $t \geq 0$, $1-t \leq \rme^{-t}$,
     \begin{multline*}
\kappa       \Lambda_{k+1,i}(\gamma) =   \defEns{ \prod_{j=k+1}^{i}(1-\kappa \gaStep_j)}^{-1} \defEns{ 1 - \prod_{j=k+1}^{i}(1-\kappa \gaStep_j)} \geq    \defEns{ \prod_{j=k+1}^{i}(1-\kappa \gaStep_j)}^{-1}  (1-\rme^{-\kappa \Gamma_{k+1,p_k}})  \\ \geq \defEns{ \prod_{j=k+1}^{i}(1-\kappa \gaStep_j)}^{-1} (1-\rme^{-\kappa })   \eqsp.
     \end{multline*}
     Using this result, we get by \eqref{eq:bound_w_n_2_1} and $(1-t)^{1/2}\leq 1-t/2$ for $t \in \ccint{0,1}$, that there exists $\tilde{C} \geq 0$ independent of $(\gamma_k)_{k \geq 1}$ such that
   for any $n \in \nset^*$,
   \begin{align*}
     w^1_n &\leq 2 C  \Gamma_{2 , n+1}^{-2} \sum_{k=0}^{n-1} \gaStep_{k+1}  \Gamma_{k+1,p_k} +  2(1-\rme^{-\kappa})^{-1}  \sum_{k=0}^{n-1} \gaStep_{k+1} \defEns{ \sum_{i=p_k+1}^{n} \weightD{i}\prod_{j=k+1}^{i}(1-\kappa \gaStep_j)^{1/2} }^2 \\
           & \leq 2 C  \Gamma_{2 , n+1}^{-1} \sum_{k=0}^{n-1} \gaStep_{k+1}  \Gamma_{k+1,p_k} +  2(1-\rme^{-\kappa})^{-1}  \sum_{k=0}^{n-1} \gaStep_{k+1} \defEns{ \sum_{i=p_k+1}^{n} \weightD{i}\prod_{j=k+1}^{i}(1-\kappa \gaStep_j/2) }^2 \\
     & \leq 2 C  \Gamma_{2 , n+1}^{-1} \sum_{k=0}^{n-1} \gaStep_{k+1}  \Gamma_{k+1,p_k} +  2(1-\rme^{-\kappa})^{-1} \tilde{C}  \Gamma_{2 , n+1}^{-1}  \eqsp,
  \end{align*}
  Since  $\Gamma_{k+1,p_k} \leq 1+\gamma_1$, for any $n \in \nset^*$, $k \in \iint{0}{n-1}$ and definition of $p_k$ \eqref{eq:def_p_k}, we obtain
  that there exists $C \geq 0$ such that
  for any $n \in \nset^*$,
  \begin{equation}
    \label{eq:bound_w_n_1_2}
    w^1_n \leq C  \Gamma_{2 , n+1}^{-1} \eqsp.
  \end{equation}
  Similarly, we have that there exists $C \geq 0$ independent of $(\gamma_k)_{k \geq 1}$  satisfying for any $n\in \nset^*$, $w^2_n \leq  C  \Gamma_{2 , n+1}^{-1}$
  Combining this result, \eqref{eq:bound_w_n_1_1} and \eqref{eq:bound_w_n_1_2} concludes the proof of \eqref{eq:bound_u_n_4}.

  \subsection{Proof of \Cref{theo:concentration_gauss}}
\label{sec:proof-crefth_conc_gauss}
Let  $N \geq 0$, $n \geq 1$, $x \in \rset^d$ and $f$ be a
Lipschitz function.
 To prove \Cref{theo:concentration_gauss}, we derive an upper bound of the Laplace transform of
$\estimateur{f} - \PE_x[ \estimateur{f} ]$. Consider the decomposition by martingale increments
\begin{multline*}
\expeMarkov{x}{\rme^{\lambda \{ \estimateur{f} - \PE_x[ \estimateur{f} ] \}}}
 =  \expeMarkov{x}{\rme^{\lambda\{\CPE[x]{\hat{\pi}_n^N(f)}{\mcg_N} - \PE_x[\hat{\pi}_n^N(f)]\} + \sum_{k=N}^{N+n-1} \lambda  \{ \CPE[x]{\hat{\pi}_n^N(f)}{\mcg_{k+1}} - \CPE[x]{\hat{\pi}_n^N(f)}{\mcg_k} \} }} \eqsp.
\end{multline*}
Now using  \eqref{eq:relation_mart_var_cond-new} with the sequence of functions $(\martInc_{n,k}^N)$ and $\martIncF^N_n$ given by \eqref{eq:def_mart_f-new} and \eqref{eq:def_F_f-new}, respectively, we have by the Markov property
\begin{multline}
  \label{eq:link_lap_transf_mart_N_n_n}
\expeMarkov{x}{\rme^{\lambda \{ \estimateur{f} - \PE_x[ \estimateur{f} ] \}}}
\\
= \expeMarkov{x}{ \rme^{\lambda\defEns{\martIncF_n^N(X_n)
      -\expeMarkov{x}{\martIncF_n^N(X_n)}} } \prod_{k=N}^{N+n-1} R_{\gamma_{k+1}}\parentheseDeux{\rme^{
      \lambda \{ \martInc^N_{n,k+1}(\cdot) - R_{\gamma_{k+1}}\martInc^N_{n,k+1}(X_k) \} }}(X_k) } \eqsp,
\end{multline}
where $R_\gamma$ is given by \eqref{eq:definition-Rgamma} for $\gamma
>0$.
 We use the same strategy to get concentration inequalities than
to bound the variance term in the previous section, replacing the
Gaussian Poincar{\'e} inequality by the log-Sobolev inequality to get
 uniform bound on $$
R_{\gamma_{k+1}}\{\exp ( \lambda
\lbrace \martInc^N_{n,k+1}( \cdot ) -
R_{\gamma_{k+1}}\martInc^N_{n,k+1}(X_{k}) \rbrace ) \}(X_k)$$
\wrt~$X_k$, for all $k \in \{N+1,\dots,N+n \}$.  Indeed for all $x \in \rset^d$ and $\gamma >0$, recall
that $R_\gamma(x,\cdot)$ is a Gaussian distribution with mean $x -
\gamma \nabla U(x)$ and covariance matrix $2 \gamma
\operatorname{I}_d$. The log-Sobolev inequality \cite[Theorem
5.5]{boucheron:lugosi:massart:2013} shows that for all Lipschitz
function $g :\rset^d \to \rset$, $x \in \rset^d$, $\gamma >0$ and
$\lambda >0$,
\begin{equation}
\label{lem:log_sob_R_gamma}
\int R_\gamma(x,\rmd y) \defEns{\exp\parenthese{\lambda \{g(y) - R_\gamma g(x) \}}} \leq \exp\parenthese{\gamma \lambda^2 \norm{g}_{\Lip}^2 } \eqsp.
\end{equation}
We deduced from this result, \eqref{eq:relation_mart_var_cond-new} and  \Cref{propo:contraction_Euler-new}, an equivalent
of \Cref{coro:induction_function_mart-new} for the Laplace transform of
$\martInc^N_{n,k+1}$ under $\delta_y R_{\gamma_{k+1}}$ for $k \in \{N+1,\dots,N+n\}$ and all $y \in \rset^d$.

\begin{corollary}
\label{coro:lap_transf_mart}
Assume \Cref{assum:regularity_2} and \Cref{assum:potentialU}. Let $(\gamma_k)_{k \geq 1}$ be a non-increasing sequence with $\gamma_1 \leq 2/(m+L)$. Let $N \geq 0$ and $n \geq 1$. Then for all $k\in \defEns{N,\dots,N+n-1}$, $y \in \rset^d$ and $\lambda >0$,
\begin{equation*}
 R_{\gamma_{k+1}}\defEns{\rme^{
      \lambda \{ \martInc^N_{n,k+1}( \cdot ) -
      R_{\gamma_{k+1}}\martInc^N_{n,k+1}(y) \} }}(y)
 \leq \exp\parenthese{4  \gamma_{k+1}  \lambda^2 \norm{f}_{\Lip}^2 (\kappa \Gamma_{N+2,N+n+1})^{-2}} \eqsp,
\end{equation*}
where $\martInc_{n,k}^N$ is given by \eqref{eq:def_mart_f-new}.
\end{corollary}
It remains to control the Laplace
transform of $\martIncF_n^N$ under $\delta_x Q^N_\gamma$, where $\delta_x Q^N_\gamma$ is defined
by \eqref{eq:iterate_kernel}.  For this, using again that by
\eqref{eq:def_F_f-new} and
\Cref{propo:contraction_Euler-new},
$\martIncF_n^N$ is a Lipschitz function, we iterate \eqref{lem:log_sob_R_gamma}
to get bounds on the Laplace transform of Lipschitz function $g$ under
$Q^{n,\ell}_\gamma(y,\cdot)$ for all $y \in \rset^d$ and $n,\ell \geq 1$,
since for all $n,\ell \geq 1$, $Q^{n,\ell}_\gamma g$ is a Lipschitz function
by \Cref{propo:contraction_Euler-new}.

\begin{lemma}
\label{lem:transf_laplace_kernel}
Assume \Cref{assum:regularity_2} and \Cref{assum:potentialU}. Let $(\gamma_k)_{k \geq 1}$ be a non-increasing sequence with $\gamma_1 \leq 2/(m+L)$.  Let $g : \rset^d \to \rset$ be a Lipschitz function, then for all $n,p \geq 1$, $n \leq p$, $y \in \rset^d$ and $\lambda >0$:
\begin{equation}
\label{eq:induction-hypothesis}
Q^{n,p}_\gamma\defEns{\exp\parenthese{\lambda \{g(\cdot) -Q^{n,p}_\gamma g(y) \}}} (y) \leq
 \exp\left( \kappa^{-1}\lambda ^2 \norm{g}^2_{\Lip} \right) \eqsp,
\end{equation}
where $Q_{n,p}^\gamma$ is given by \eqref{eq:iterate_kernel}.
\end{lemma}

\begin{proof}
Let $(X_n)_{n \geq 0} $ the Euler approximation given by \eqref{eq:euler-proposal-2} and started at $y \in \rset^d$.
By decomposing $g(X_p)- \CPE[y]{g(X_p)}{\mcg_n}= \sum_{k=n+1}^p \{ \CPE[y]{g(X_p)}{\mcg_k} - \CPE[y]{g(X_p)}{\mcg_{k-1}} \}$,
and using $\CPE[y]{g(X_p)}{\mcg_k}=  Q_\gamma^{k+1,p} g(X_k)$, we get
\begin{align*}
&\CPE[y]{\exp\left(\lambda \defEns{g(X_p) - \CPE[y]{g(X_p)}{\mcg_n}} \right)}{\mcg_n} \\
&=  \CPE[y]{\prod_{k=n+1}^p\CPE[y]{\exp\left( \lambda \defEns{\CPE[y]{g(X_p)}{\mcg_k} - \CPE[y]{g(X_p)}{\mcg_{k-1}} }\right)}{\mcg_{k-1}} } {\mcg_n}\\
&= \CPE[y]{\prod_{k=n+1}^p R_{\gamma_k} \exp\left( \lambda\defEns{ Q_\gamma^{k+1,p} g(\cdot) -R_{\gamma_k} Q_\gamma^{k+1,p} g(X_{k-1})}\right) (X_{k-1})} {\mcg_n} \eqsp.
\end{align*}
By the Gaussian log-Sobolev inequality \eqref{lem:log_sob_R_gamma}, we get
\[
\CPE[y]{\exp\left(\lambda \defEns{g(X_p) - \CPE[y]{g(X_p)}{\mcg_n}} \right)}{\mcg_n} \leq \exp\left(\lambda^2 \sum_{k=n+1}^p \gamma_k \norm[2]{Q_\gamma^{k+1,p} g}_{\Lip} \right) \eqsp.
\]
 The proof follows from
\Cref{propo:contraction_Euler-new} and \Cref{lem:suite_recurrence_2}, using the bound $(1-t)^{1/2} \leq 1 - t/2$ for $t \in \ccint{0,1}$.

\end{proof}
Combining this result and $\normLigne{\martIncF^N_n}_{\Lip} \leq 2 \kappa^{-1} \norm{f}_{\Lip} \Gamma_{N+2,N+n+1}^{-1}$ by \Cref{propo:contraction_Euler-new}, we get an analogue of \Cref{coro:function_F} for the Laplace transform of $\martIncF^N_n$:
\begin{corollary}
\label{coro:lap_transf_F}
  Assume \Cref{assum:regularity_2} and \Cref{assum:potentialU}. Let
  $(\gamma_k)_{k \geq 1}$ be a non-increasing sequence with $\gamma_1
  \leq 2/(m+L)$. Let $N
  \geq 0$ and $n \geq 1$. Then for all $\lambda >0$ and $x \in \rset^d$,
  \begin{equation*}
  \expeMarkov{x}{ \rme^{\lambda\{\martIncF_n^N(X_n)
      -\expeMarkov{x}{\martIncF_n^N(X_n)}\} } }
\leq \exp\parenthese{4 \kappa^{-3}\lambda^2 \norm{f}_{\Lip}^2 \Gamma_{N+2,N+n+1}^{-2}}\eqsp,
 \end{equation*}
 where $\martIncF^N_{n}$ is given by \eqref{eq:def_F_f-new}.
\end{corollary}
 The Laplace transform of $\estimateur{f}$ can be explicitly bounded using \Cref{coro:lap_transf_mart} and \Cref{coro:lap_transf_F} in \eqref{eq:link_lap_transf_mart_N_n_n}.
\begin{proposition}
\label{lem:concentration_estimator}
Assume \Cref{assum:regularity_2} and \Cref{assum:potentialU}. Let $(\gamma_k)_{k \geq 1}$ be a non-increasing sequence with $\gamma_1 \leq 2/(m+L)$.  Then for all $N \geq 0$, $n \geq 1$, Lipschitz functions $f : \rset^d \to \rset$, $\lambda >0$ and $x \in \rset^d$:
\begin{equation*}
\expeMarkov{x}{\rme^{\lambda \{ \estimateur{f} - \PE_x[ \estimateur{f} ] \}}} \leq \exp \parenthese{ 4 \kappa^{-2} \lambda^2 \norm{f}_{\Lip}^2 \Gamma^{-1}_{N+2,N+n+1} u_{N,n}^{(3)}(\gamma) }
\eqsp,
\end{equation*}
where  $u_{N,n}^{(3)}(\gamma)$ is  given by  \eqref{eq:def_u_n_3}.
\end{proposition}

\begin{proof}[Proof of \Cref{theo:concentration_gauss}]
  Using the Markov inequality and \Cref{lem:concentration_estimator}, for all $\lambda >0$, we have:
\begin{equation*}
\probaMarkov{x}{\estimateur{f} \geq  \PE_x[ \estimateur{f} ]+r }  \leq  \exp\left( -\lambda r +4 \kappa^{-2} \lambda^2 \norm{f}_{\Lip}^2 \Gamma^{-1}_{N+2,N+n+1} v_{N,n}(\gamma)    \right) \eqsp.
\end{equation*}
Then the result follows from taking $\lambda =  (r \kappa^2 \Gamma_{N+2,N+n+1}) /(8\norm{f}_{\Lip}^2  v_{N,n}(\gamma) )$.
\end{proof}

%

\subsection{Proof of \Cref{theo:concentration_gauss_tv}}
\label{sec:proof-concentration_tv}

Let $N \geq 0$, $n \geq 1$, $x \in \rset^d$ and $f \in
\functionspace[b]{\rset^d}$. The main idea of the proof is to consider
the decomposition \eqref{eq:link_lap_transf_mart_N_n_n} again but
combined with the decomposition of $\martInc^N_{n,k+1}$, for $k \in
\{N,\ldots,N+n-1\}$, into a Lipschitz component and a bounded
measurable component as it is done in the proof of
\eqref{eq:lem:induction_function_mart_1_tv_1}.  Let $k\in
\defEns{N,\dots,N+n-1}$.  By definition
\eqref{eq:def_mart_f-new}, $\martInc^N_{n,k}= \weight{k+1} f +
\tilde{\martInc}^N_{n,k}$, where $\tilde{\martInc}^N_{n,k}=
\sum_{i=k+2}^{N+n} \weight{i} Q_\gamma^{k+2,i}f$. Using that $f$ is
bounded, we get for all $y \in \rset^d$ and $\lambda >0$,
  \begin{multline*}
 R_{\gamma_{k+1}}\defEns{\rme^{
      \lambda \{ \martInc^N_{n,k+1}( \cdot ) -
      R_{\gamma_{k+1}}\martInc^N_{n,k+1}(y) \} }}(y)
\\\leq \rme^{\lambda \, \osc{f}\, \gaStep_{k+2} (\Gamma_{N+2,N+n+1})^{-2} }R_{\gamma_{k+1}}\defEns{\rme^{
      \lambda \{ \tilde{\martInc}^N_{n,k+1}( \cdot ) -
      R_{\gamma_{k+1}}\tilde{\martInc}^N_{n,k+1}(y) \} }}(y)
\end{multline*}
By \eqref{eq:bound_norm_lip_phitilde} and \eqref{lem:log_sob_R_gamma}, we obtain for all $y \in \rset^d$ and $\lambda >0$,
\begin{multline}
\label{eq:lap_transf_mart_tv}
 R_{\gamma_{k+1}}\defEns{\rme^{
      \lambda \{ \martInc^N_{n,k+1}( \cdot ) -
      R_{\gamma_{k+1}}\martInc^N_{n,k+1}(y) \} }}(y)
\\ \leq \exp\parenthese{\lambda \, \osc{f}\, \gaStep_{k+2} (\Gamma_{N+2,N+n+1})^{-2} +  (\lambda \, \osc{f})^2 \gamma_{k+1}   \parenthese{\sum_{i=k+2}^{N+n} \weight{i}
/ (\uppi \LambdarMSE_{k+2,i})^{1/2}}^{2}} \eqsp.
\end{multline}

It remains to control the Laplace transform of $\martIncF_n^N$ under
$\delta_x Q^N_\gamma$.  For this, note that by
\eqref{eq:bound_norm_lip_psitilde} $\martIncF_n^N$ is a Lipschitz
function. Therefore using \Cref{lem:transf_laplace_kernel}, we get an
analogue of \Cref{coro:lap_transf_F}: for all $y \in \rset^d$ and
$\lambda >0$,
  \begin{equation}
\label{eq:lap_transf_F_tv}
  \expeMarkov{y}{ \rme^{\lambda\{\martIncF_n^N(X_n)
      -\expeMarkov{x}{\martIncF_n^N(X_n)}\} } }
\leq \exp\parenthese{ \kappa^{-1}\lambda^2 \osc{f}^2 \parenthese{ \sum_{i=N+1}^{N+n} \weight{i} /(\uppi \LambdarMSE_{N+1,i})^{1/2} }^2 }\eqsp,
 \end{equation}

Combining \eqref{eq:lap_transf_mart_tv} and \eqref{eq:lap_transf_F_tv} in \eqref{eq:link_lap_transf_mart_N_n_n}, the Laplace transform of $\estimateur{f}$ can be explicitly bounded: for all $\lambda >0$,
\begin{equation*}
\label{eq:concentration_estimator_tv}
\expeMarkov{x}{\rme^{\lambda \{ \estimateur{f} - \PE_x[ \estimateur{f} ] \}}}
\leq \rme^{ \lambda\, \osc{f}  (\Gamma_{N+2,N+n+1})^{-1} +   (\lambda\, \osc{f})^2 u_{N,n}^{(5)}(\gaStep)   } \eqsp.
\end{equation*}
Using this result and the Markov inequality, for all $\lambda >0$, we have:
\begin{multline*}
\probaMarkov{x}{\estimateur{f} \geq  \PE_x[ \estimateur{f} ]+r } \\ \leq  \exp\left( -\lambda r +\lambda\, \osc{f}  (\Gamma_{N+2,N+n+1})^{-1} +  (\lambda\, \osc{f})^2 u_{N,n}^{(5)}(\gaStep)   \right) \eqsp.
\end{multline*}
Then the proof follows from taking
\begin{equation*}
\lambda =  (r - \osc{f}(\Gamma_{N+2,N+n+1})^{-1}) /(2\osc{f}^2 u_{N,n}^{(5)}(\gaStep)) \eqsp.
\end{equation*}

%% file: proof_supp_autoreg.tex
\section{Additional technical results}
\label{sec:proof-crefs-technical_result}

\input{coupling}

\subsection{Distribution of hitting time of $0$ for Ornstein-Ulhenbeck processes}
\label{sec:distr-hitt-time}
Consider the one-dimensional Ornstein-Ulhenbeck process $(\tmsu_t)_{t \geq 0}$ defined for $t \geq 0$  by
\begin{equation*}
  \tmsu_t=a \rme^{-\theta t} +  \sigma \int_0^{t} \rme^{\theta(s-t)} \rmd B^1_s = a \rme^{-\theta t} + \frac{\sigma}{\sqrt{2 \theta}} \rme^{-\theta t} B^1_{\rme^{2 \theta t} -1} \eqsp,
\end{equation*}
where $a \in \rset$, $ \theta, \sigma >0$ and $(B^1_t)_{t \geq 0}$ is a one-dimensional Brownian motion. Note that with our convention, $(\tmsu_t)_{t \geq 0}$ is the solution of the SDE
$\rmd \tmsu_t = -\theta \, \tmsu_t \rmd t + \sigma \rmd B^1_t$ with initial condition $\tmsu_0= a$.
Define the hitting time of $(\tmsu_t)_{t \geq 0}$ of $0$ by $ \tT_0 = \inf\{ t \geq 0 \, : \, \tmsu_t = 0 \}$.
\begin{proposition}[\protect{\cite[Formula 2.0.2, page 542]{borodin:salminen:2002}}]
  \label{propo:hitting_time_OU}
  For all $a \in \rset$, $\theta,\sigma >0$, and $t > 0$, it holds
  \begin{equation*}
   \proba{\min_{0 \leq s \leq t} \tmsu_s  > 0} =  \proba{\tT_0 >t} = 1-2 \Phibf\parenthese{ - \frac{\sqrt{2 \theta}\abs{a}}{\sigma \sqrt{\rme^{2\theta t} -1}}} \eqsp,
  \end{equation*}
  where $\Phibf$ is the cumulative distribution function of the standard normal distribution.
\end{proposition}



%% file: coupling.tex
 \subsection{Coupling}
\label{sec:coupling}
 \begin{lemma}
Assume \Cref{assum:strict_contraction_AR}. For all $x,y \in \rset^d$ and $k \geq 1$, $\Kr_k((x,y),\cdot)$ is a transference plan of $\Pr_k(x,\cdot)$ and $\Pr_k(y,\cdot)$
\end{lemma}
\begin{proof}
  By construction,
  $\Kr_k((x,y), \cdot \times \rset^d) = \Pr_k(x,\cdot)$ for all
  $x,y \in \rset^d$ and
  $\Kr_k((x,y) , \rset^d \times \cdot) = \Pr_k(y,\cdot)$ for all
  $(x,y)$ such that $\funreg_k(x) = \funreg_k(y)$. Therefore, it
  remains to show that
  $\Kr_k((x,y),  \rset^d \times  \cdot) = \Pr_k(y,\cdot)$ for any $(x,y) \in \rset^d\times \rset^d$,  $\funreg_k(x) \neq \funreg_k(y)$.
  First for all
  $\eventA \in \borelSet(\rset^d)$, we have
 \begin{align}
   \label{eq:coupling_kernel_AR_proof_coupl_proof}
 &\Kr_k((x,y) , \rset^d \times \eventA) =  \frac{1}{ (2 \uppi \sigmakD)^{d/2}}\int_{\rset^d} \1_{\eventA}(\tildex)
 p_k(x,y,\tildex-\funreg_k(x)) \rme^{-\norm[2]{\tildex - \funreg_k(x)}/(2\sigmakD)} \rmd \tildex
 \\
 \nonumber
 &+ \frac{1}{ (2 \uppi \sigmakD)^{d/2}}\int_{\rset^d} \1_{\eventA}(\transfrr_k(x,y,\tildex- \funreg_k(x)))
 \defEns{1-p_k(x,y,\tildex-\funreg_k(x))} \rme^{-\norm[2]{\tildex - \funreg_k(x)}/(2\sigmakD)} \rmd \tildex \eqsp.
 \end{align}
 Since $(\Id - 2 \er_k(x,y) \er_k(x,y)^{\transp})$ is an orthogonal matrix, making the change of variable
 $
 \tildey =\transfrr_k(x,y,\tildex-\funreg_k(x))$ and using that
 \[
 \ps{\er_k(x,y)}{\funreg_k(y)-\tildey} = \ps{\er_k(x,y)}{\tildex - \funreg_k(x)}
 \]
 we get that
 \begin{multline}
   \label{eq:coupling_kernel_AR_proof_coupl_proof_1}
   \int_{\rset^d} \1_{\eventA}(\transfrr_k(x,y,\tildex-\funreg_k(x)))
 \defEns{1-p_k(x,y,\tildex-\funreg_k(x))} \rme^{-\norm[2]{\tildex - \funreg_k(x)}/(2\sigmakD)} \rmd \tildex \\
 = \int_{\rset^d} \1_{\eventA}(\tildey)
 \defEns{1-p_k(x,y,\funreg_k(y)-\tildey)} \rme^{-\norm[2]{\tildey - \funreg_k(y)}/(2\sigmakD)} \rmd \tildey \eqsp.
 \end{multline}
 By definition of $\alphar_k$
 \eqref{def:alphar}, we have for all $\tildex \in \rset^d$,
 \begin{equation}
   \label{eq:coupling_kernel_AR_proof_coupl_proof_0}
 \alphar_k\parenthese{x,y, \tildex-\funreg_k(x)} = \frac{\phibfs[k]\parenthese{ \ps{\er_k(x,y)}{\funreg_k(y)-\tildex}}}{\phibfs[k]\parenthese{\norm{\Er_k(x,y)}-\ps{\er_k(x,y)}{\funreg_k(y)-\tildex}}} = \frac{1}{\alphar_k\parenthese{x,y, \funreg_k(y)-\tildex}}  \eqsp.
 \end{equation}
 In addition using that
 $$
 \norm[2]{\tildex -\funreg_k(x)} = \norm[2]{\tildex -\funreg_k(y)} -
 2 \ps{\funreg_k(y)-\tildex}{\Er_k(x,y)} + \norm[2]{\Er_k(x,y)} \eqsp,
 $$
 we  obtain
 \begin{equation}
 \label{eq:coupling_kernel_AR_proof_coupl_proof_2}
 p_k(x,y,\tildex-\funreg_k(x)) \rme^{-\norm[2]{\tildex-\funreg_k(x)}/(2\sigmakD)} =
 p_k(x,y,\funreg_k(y)-\tildex) \rme^{-\norm[2]{\tildex-\funreg_k(y)}/(2\sigmakD)}  \eqsp.
 \end{equation}
 Plugging \eqref{eq:coupling_kernel_AR_proof_coupl_proof_1}
 and \eqref{eq:coupling_kernel_AR_proof_coupl_proof_2} into
 \eqref{eq:coupling_kernel_AR_proof_coupl_proof} implies that
 $\Kr_k((x,y),\rset^d\times \eventA) = \Pr_k(y,\eventA)$.
\end{proof}


%% file: main_arxiv.bbl
\begin{thebibliography}{10}

\bibitem{albert:chib:1993}
J.~H. Albert and S.~Chib.
\newblock Bayesian analysis of binary and polychotomous response data.
\newblock {\em Journal of the American Statistical Association},
  88(422):669--679, 1993.

\bibitem{borodin:salminen:2002}
A.~N. Borodin and P.~Salminen.
\newblock {\em Handbook of {B}rownian motion---facts and formulae}.
\newblock Probability and its Applications. Birkh\"auser Verlag, Basel, second
  edition, 2002.

\bibitem{boucheron:lugosi:massart:2013}
S.~Boucheron, G.~Lugosi, and P.~Massart.
\newblock {\em Concentration inequalities}.
\newblock Oxford University Press, Oxford, 2013.
\newblock A nonasymptotic theory of independence, With a foreword by Michel
  Ledoux.

\bibitem{Bubeck:2015}
S.~Bubeck, R~Eldan, and J.~Lehec.
\newblock Finite-time analysis of projected langevin monte carlo.
\newblock In {\em Proceedings of the 28th International Conference on Neural
  Information Processing Systems}, NIPS'15, pages 1243--1251, Cambridge, MA,
  USA, 2015. MIT Press.

\bibitem{bubley:dyer:jerrum:1998}
R.~Bubley, M.~Dyer, and M.~Jerrum.
\newblock An elementary analysis of a procedure for sampling points in a convex
  body.
\newblock {\em Random Structures Algorithms}, 12(3):213--235, 1998.

\bibitem{chen:shao:1989}
M.~F. Chen and S.~F. Li.
\newblock Coupling methods for multidimensional diffusion processes.
\newblock {\em Ann. Probab.}, 17(1):151--177, 1989.

\bibitem{choi:hobert:2013}
H.~M. Choi and J.~P. Hobert.
\newblock The {P}olya-{G}amma {G}ibbs sampler for {B}ayesian logistic
  regression is uniformly ergodic.
\newblock {\em Electron. J. Statist.}, 7:2054--2064, 2013.

\bibitem{chopin:ridgway:2015}
N.~Chopin and Ridgway J.
\newblock Leave {P}ima {I}ndians alone: binary regression as a benchmark for
  {B}ayesian computation.
\newblock {\em Statist. Sci.}, 32(1):64--87, 2017.

\bibitem{dalalyan:2017}
A.~S. Dalalyan.
\newblock Further and stronger analogy between sampling and optimization:
  Langevin monte carlo and gradient descent.
\newblock In {\em Proceedings of the 30th Annual Conference on Learning
  Theory}.

\bibitem{dalalyan:2014}
A.~S. Dalalyan.
\newblock Theoretical guarantees for approximate sampling from smooth and
  log-concave densities.
\newblock {\em J. R. Stat. Soc. Ser. B. Stat. Methodol.}, 79(3):651--676, 2017.

\bibitem{durmus:moulines:2015:supplement}
A.~Durmus and \'E. Moulines.
\newblock Supplement to ``high-dimensional bayesian inference via the
  unadjusted langevin algorithm'', 2015.
\newblock https://hal.inria.fr/hal-01176084/.

\bibitem{durmus:moulines:2016}
A.~Durmus and \'E. Moulines.
\newblock Nonasymptotic convergence analysis for the unadjusted {L}angevin
  algorithm.
\newblock {\em Ann. Appl. Probab.}, 27(3):1551--1587, 2017.

\bibitem{eberle:2016}
A.~Eberle.
\newblock Quantitative contraction rates for {M}arkov chains on continuous
  state spaces.
\newblock In preparation.

\bibitem{eberle:2015}
A.~Eberle.
\newblock Reflection couplings and contraction rates for diffusions.
\newblock {\em Probab. Theory Related Fields}, pages 1--36, 2015.

\bibitem{egz-16}
A.~Eberle, A.~Guillin, and R.~Zimmer.
\newblock Quantitative {H}arris type theorems for diffusions and
  {M}c{K}ean-{V}lasov processes.
\newblock To appear in Trans. Am. Math. Soc., 2018.

\bibitem{ermak:1975}
D.~L Ermak.
\newblock A computer simulation of charged particles in solution. i. technique
  and equilibrium properties.
\newblock {\em The Journal of Chemical Physics}, 62(10):4189--4196, 1975.

\bibitem{faes:ormerod:wand:2011}
C.~Faes, J.~T. Ormerod, and M.~P. Wand.
\newblock Variational {B}ayesian inference for parametric and nonparametric
  regression with missing data.
\newblock {\em Journal of the American Statistical Association},
  106(495):959--971, 2011.

\bibitem{fruehwirth:fruehwirth:2010}
S.~Fr\"{u}hwirth-Schnatter and R.~Fr\"{u}hwirth.
\newblock Data augmentation and {MCMC} for binary and multinomial logit models
  statistical modelling and regression structures.
\newblock In Thomas Kneib and Gerhard Tutz, editors, {\em Statistical Modelling
  and Regression Structures}, chapter~7, pages 111--132. Physica-Verlag HD,
  Heidelberg, 2010.

\bibitem{gramacy:polson:2012}
R.~B. Gramacy and N.~G. Polson.
\newblock Simulation-based regularized logistic regression.
\newblock {\em Bayesian Anal.}, 7(3):567--590, 09 2012.

\bibitem{grenander:1983}
U.~Grenander.
\newblock Tutorial in pattern theory.
\newblock Division of Applied Mathematics, Brown University, Providence, 1983.

\bibitem{grenander:miller:1994}
U.~Grenander and M.~I. Miller.
\newblock Representations of knowledge in complex systems.
\newblock {\em J. Roy. Statist. Soc. Ser. B}, 56(4):549--603, 1994.
\newblock With discussion and a reply by the authors.

\bibitem{hanson:branscum:wesley:2014}
T.~E. Hanson, A~J. Branscum, and W.~O. Johnson.
\newblock Informative {$g$}-priors for logistic regression.
\newblock {\em Bayesian Anal.}, 9(3):597--611, 2014.

\bibitem{holmes:held:2006}
C.~C. Holmes and L.~Held.
\newblock Bayesian auxiliary variable models for binary and multinomial
  regression.
\newblock {\em Bayesian Anal.}, 1(1):145--168, 03 2006.

\bibitem{joulin:ollivier:2010}
A.~Joulin and Y.~Ollivier.
\newblock Curvature, concentration and error estimates for {M}arkov chain
  {M}onte {C}arlo.
\newblock {\em Ann. Probab.}, 38(6):2418--2442, 2010.

\bibitem{karatzas:shreve:1991}
I.~Karatzas and S.E. Shreve.
\newblock {\em Brownian Motion and Stochastic Calculus}.
\newblock Graduate Texts in Mathematics. Springer New York, 1991.

\bibitem{klartag:2007}
B.~Klartag.
\newblock A central limit theorem for convex sets.
\newblock {\em Invent. Math.}, 168(1):91--131, 2007.

\bibitem{lamberton:pages:2002}
D.~Lamberton and G.~Pag{\`e}s.
\newblock Recursive computation of the invariant distribution of a diffusion.
\newblock {\em Bernoulli}, 8(3):367--405, 2002.

\bibitem{lamberton:pages:2003}
D.~Lamberton and G.~Pag{\`e}s.
\newblock Recursive computation of the invariant distribution of a diffusion:
  the case of a weakly mean reverting drift.
\newblock {\em Stoch. Dyn.}, 3(4):435--451, 2003.

\bibitem{lemaire:2005}
V.~Lemaire.
\newblock {\em Estimation de la mesure invariante d'un processus de diffusion}.
\newblock PhD thesis, Université Paris-Est, 2005.

\bibitem{lindvall:rogers:1986}
T.~Lindvall and L.~C.~G. Rogers.
\newblock Coupling of multidimensional diffusions by reflection.
\newblock {\em Ann. Probab.}, 14(3):860--872, 1986.

\bibitem{mattingly:stuart:higham:2002}
J.~C. Mattingly, A.~M. Stuart, and D.~J. Higham.
\newblock Ergodicity for {SDE}s and approximations: locally {L}ipschitz vector
  fields and degenerate noise.
\newblock {\em Stochastic Process. Appl.}, 101(2):185--232, 2002.

\bibitem{meyn:tweedie:2009}
S.~Meyn and R.~Tweedie.
\newblock {\em {M}arkov Chains and Stochastic Stability}.
\newblock Cambridge University Press, New York, NY, USA, 2nd edition, 2009.

\bibitem{neal:1992}
R.~M. Neal.
\newblock Bayesian learning via stochastic dynamics.
\newblock In {\em Advances in Neural Information Processing Systems 5, [NIPS
  Conference]}, pages 475--482, San Francisco, CA, USA, 1993. Morgan Kaufmann
  Publishers Inc.

\bibitem{nesterov:2004}
Y.~Nesterov.
\newblock {\em Introductory Lectures on Convex Optimization: A Basic Course}.
\newblock Applied Optimization. Springer, 2004.

\bibitem{parisi:1981}
G.~Parisi.
\newblock Correlation functions and computer simulations.
\newblock {\em Nuclear Physics B}, 180:378--384, 1981.

\bibitem{polson:scott:windle:2013}
N.~G. Polson, J.~G. Scott, and J.~Windle.
\newblock Bayesian inference for logistic models using {P}olya-{G}amma latent
  variables.
\newblock {\em Journal of the American Statistical Association},
  108(504):1339--1349, 2013.

\bibitem{roberts:tweedie:1996}
G.~O. Roberts and R.~L. Tweedie.
\newblock Exponential convergence of {L}angevin distributions and their
  discrete approximations.
\newblock {\em Bernoulli}, 2(4):341--363, 1996.

\bibitem{rossky:doll:friedman:1978}
P.~J. Rossky, J.~D. Doll, and H.~L. Friedman.
\newblock Brownian dynamics as smart {M}onte {C}arlo simulation.
\newblock {\em The Journal of Chemical Physics}, 69(10):4628--4633, 1978.

\bibitem{sabane:held:2011}
D.~Saban{\'e}s~Bov{\'e} and L.~Held.
\newblock Hyper-{$g$} priors for generalized linear models.
\newblock {\em Bayesian Anal.}, 6(3):387--410, 2011.

\bibitem{talay:tubaro:1991}
D.~Talay and L.~Tubaro.
\newblock Expansion of the global error for numerical schemes solving
  stochastic differential equations.
\newblock {\em Stochastic Anal. Appl.}, 8(4):483--509 (1991), 1990.

\bibitem{VillaniTransport}
C.~Villani.
\newblock {\em Optimal transport : old and new}.
\newblock Grundlehren der mathematischen Wissenschaften. Springer, Berlin,
  2009.

\bibitem{welling:the:2011}
M.~Welling and Y.~W. Teh.
\newblock Bayesian learning via stochastic gradient langevin dynamics.
\newblock In {\em Proceedings of the 28th International Conference on Machine
  Learning (ICML-11)}, pages 681--688, 2011.

\bibitem{windle:polson:scott:BayesLogit}
J.~Windle, N.~G. Polson, and J.~G. Scott.
\newblock Bayeslogit: Bayesian logistic regression, 2013.
\newblock http://cran.r-project.org/web/packages/BayesLogit/index.html R
  package version 0.2.

\end{thebibliography}


\begin{thebibliography}{1}

\bibitem{dalalyan:2014}
A.~S. Dalalyan.
\newblock Theoretical guarantees for approximate sampling from smooth and
  log-concave densities.
\newblock {\em J. R. Stat. Soc. Ser. B. Stat. Methodol.}, 79(3):651--676, 2017.

\bibitem{durmus:moulines:2016}
A.~Durmus and \'E. Moulines.
\newblock Nonasymptotic convergence analysis for the unadjusted {L}angevin
  algorithm.
\newblock {\em Ann. Appl. Probab.}, 27(3):1551--1587, 2017.

\end{thebibliography}
